\documentclass[letterpaper, 11pt, reqno]{amsart}
\usepackage[top = 1in, bottom = 1in, left=1in, right=1in]{geometry}
\usepackage{setspace} 
\usepackage{bm}
\usepackage{amsmath}
\usepackage{stmaryrd}
\usepackage[utf8]{inputenc}
\usepackage[T1]{fontenc}
\usepackage{lineno} 
\newcounter{propcounter}
\usepackage[inline]{enumitem}
\setlist[enumerate,1]{label={\upshape (\roman*)}}

\usepackage{amssymb,amsmath,amsthm,graphicx,xcolor,mathtools,mathrsfs}
\usepackage{comment}
\usepackage{hyperref}
\hypersetup{colorlinks, linkcolor={red!50!black}, citecolor={green!50!black}, urlcolor={blue!50!black}}

\newcommand{\cc}{\mathcal{C}}

\newcommand{\ex}{{\rm ex}}
\newcommand{\pal}{{\rm pal}}
\newcommand{\rev}{\mathrm{rev}}

\usepackage[nameinlink]{cleveref}
\theoremstyle{plain}
\newtheorem{theorem}{Theorem}[section]
\crefname{theorem}{Theorem}{Theorems}

\crefname{proposition}{Proposition}{Propositions}

\newtheorem{corollary}[theorem]{Corollary}
\crefname{corollary}{Corollary}{Corollaries}

\newtheorem{lemma}[theorem]{Lemma}
\crefname{lemma}{Lemma}{Lemmas}

\crefname{conjecture}{Conjecture}{Conjectures}

\newtheorem{problem}[theorem]{Problem}
\crefname{problem}{Problem}{Problem}

\newtheorem{claim}[theorem]{Claim}
\crefname{claim}{Claim}{Claims}

\newtheorem{observation}[theorem]{Observation}
\crefname{observation}{Observation}{Observations}

\crefname{setup}{Setup}{Setups}

\crefname{myth}{Myth}{Myths}

\crefname{fact}{Fact}{Facts}

\crefname{algorithm}{Algorithm}{Algorithms}

\newtheorem{remark}[theorem]{Remark}
\crefname{remark}{Remark}{Remarks}

\crefname{example}{Example}{Examples}

\theoremstyle{definition}
\newtheorem{definition}[theorem]{Definition}
\crefname{definition}{Definition}{Definitions}

\crefname{construction}{Construction}{Constructions}

\crefname{question}{Question}{Questions}

\numberwithin{equation}{section}

\crefformat{section}{\S#2#1#3}
\crefformat{subsection}{\S#2#1#3}
\crefformat{subsubsection}{\S#2#1#3}
\crefrangeformat{section}{\S\S#3#1#4 to~#5#2#6}
\crefmultiformat{section}{\S\S#2#1#3}{ and~#2#1#3}{, #2#1#3}{ and~#2#1#3}

\newcommand{\eps}{\varepsilon}
\allowdisplaybreaks

\begin{document}

\title{Uniform Tur\'an densities of $k$-uniform hypergraphs}

\author[H. Lin]{Hao Lin}
\address{Department of Information Security, Naval University of Engineering, Wuhan, China.}
\email{haolinz6@163.com}
\author[G. Sun]{Guowei Sun}
\address{School of Mathematics, Shandong University, Jinan, China.}
\email{1521264577@qq.com}
\author[G. Wang]{Guanghui Wang}
\author[W. Zhou]{Wenling Zhou}
\address{School of Mathematics, Shandong University, Jinan,  China.}
\email{\{\,ghwang\,|\,gracezhou\,\}@sdu.edu.cn}
\thanks{G. Wang is partially supported by the National Key R\&D Program of China (2020YFA0712400) and the Natural Science Foundation of China (12231018).
W. Zhou is partially supported by the China Postdoctoral Science Foundation (2024M761780), the Natural Science Foundation of China (12401457), Natural Science Foundation of Shandong Province (ZR2024QA067) and Young Talent of Lifting engineering for Science and Technology in Shandong, China (SDAST2025QTA074).}
\keywords{Hypergraph, uniform Tur\'an density, palette}

\begin{abstract}
For $k\ge 3$, the $(k-2)$-uniform Tur\'an density $\pi_{k-2}(F)$
of a $k$-graph $F$ is the supremum of $d$ for which there are
arbitrarily large $F$-free $k$-graphs that are uniformly $d$-dense
with respect to the $k$-vertex cliques of every $(k-2)$-graph on the
same vertex set. We develop a \emph{palette framework} for this density. For
every family $\mathcal F$ of $k$-graphs, we prove that
$\pi_{k-2}(\mathcal F)$ equals the corresponding palette Tur\'an density.
We further establish palette classification tools for the existence of $k$-graphs satisfying prescribed palette colorability constraints. Those together allow us to reduce exact density computations to a palette-homomorphism framework without relying on the hypergraph regularity method.

As applications, for all $k\ge 3$ and $r\ge 2$, we establish the following values
\[
\frac{r-1}{r},\quad 
\frac{(r-1)^2}{r^2},\quad 
\frac{r-1}{2r},\quad 
\frac{(k-1)^k}{k^k},\quad 
\frac{4(k-2)^{k-2}}{k^k},\quad 
\frac{4(k-2)^{k-2}}{3k^k}
\] as $(k-2)$-uniform Tur\'an densities of single
$k$-graphs. Finally, for every $k\ge3$, we show that there exist $k$-graphs
$F_1,F_2$ such that
\[
\pi_{k-2}(\{F_1,F_2\})<
\min\{\pi_{k-2}(F_1),\pi_{k-2}(F_2)\},
\]
which provides the first examples of \emph{non-principal} families for this density.
\end{abstract}

\maketitle
\thispagestyle{empty}
\vspace{-2.15em}

\section{Introduction}
For a positive integer $k$ and a finite set $V$, let $ \binom{V}{k} = \{ X \subseteq V : |X| = k \}$ denote the family of all $k$-sets of $V$. A {\it $k$-uniform hypergraph} (or {\it $k$-graph}) $H = (V, E)$ consists of a vertex set $V$ and an edge set $E \subseteq \binom{V}{k}$. Given a family $\mathcal F$ of $k$-graphs, a $k$-graph $H$ is \emph{$\mathcal F$-free} if it contains no member of $\mathcal F$ as a subgraph. The maximum number of edges in an $\mathcal F$-free $k$-graph on $n$ vertices is the \emph{Tur\'an number} of $\mathcal F$, denoted by $\ex(n,\mathcal F)$.
A standard averaging argument shows that the sequence $\ex(n, \mathcal{F})/\binom{n}{k}$ is non-increasing in $n$.
Therefore, one often focuses on the \emph{Tur\'an density} of $\mathcal{F}$, defined by
\[
\pi(\mathcal{F})=\lim_{n\to \infty}{\ex(n,\mathcal{F})}/{\binom{n}{k}}.
\]
If $\mathcal F=\{F\}$, we write $\ex(n,F)$ and $\pi(F)$ instead of $\ex(n,\mathcal F)$ and $\pi(\mathcal F)$.
Let $\Pi^{(k)}_{\infty}$ be the set of Tur\'an densities of arbitrary families of $k$-graphs, let $\Pi^{(k)}_{\rm fin}$ be the set of Tur\'an densities of finite families of $k$-graphs, and let $\Pi^{(k)}$ be the set of Tur\'an densities of single $k$-graphs. Clearly, for every $k\ge 2$, 
\begin{equation}\label{eq:classic-turan-density}
\Pi^{(k)} \subseteq \Pi^{(k)}_{\rm fin}\subseteq \Pi^{(k)}_{\infty}.  
\end{equation}

For $k=2$ (that is, for graphs), $\Pi^{(2)}$ is well understood.
The celebrated Erdős-Stone-Simonovits Theorem~\cite{E-Stone, E-Simonovits} determines the Tur\'an density of every graph in terms of its chromatic number, and consequently
\[
\Pi^{(2)}_{\infty}=\Pi^{(2)}_{\rm fin}=\Pi^{(2)} =\{(r-1)/r: r\in \mathbb N^+\}.
\] 
However, for $k\ge 3$, we do not have an analog to the Erdős-Stone-Simonovits Theorem, and much less is known about the Tur\'an densities of single $k$-graphs. Tur\'an in his original work~\cite{Tu-graph} asked for the Tur\'an number of complete $k$-graphs. Later, Erd\H{o}s offered a monetary prize for determining $\pi(K_r^{(k)})$ for a complete $k$-graph $K_r^{(k)}$ with $r> k\ge3$. To this day, no case of this problem with $r> k\ge 3$ has been resolved, even asymptotically, despite substantial effort.

Hypergraphs also exhibit phenomena absent in graphs. Mubayi and R\"odl~\cite{Mubayi-rodl-02} conjectured that there exist \emph{non-principal} families $\mathcal F$ for $k\ge 3$, that is, families $\mathcal F$ such that
\[
\pi(\mathcal F)<\min\{\pi(F):F\in\mathcal F\},
\]
and remarked that this should hold even when $|\mathcal F|=2$. 
Balogh~\cite{Balogh2002} confirmed this conjecture by constructing a finite non-principal family of $3$-graphs whose size is much larger than two. 
Later, Mubayi and Pikhurko~\cite{MubayiPikhurko2008} constructed a non-principal family of $k$-graphs of size two for every $k\ge 3$. 
These examples illustrate the difference between graphs and hypergraphs.

\subsection{Tur\'an problems in uniformly dense hypergraphs}

The difficulty of determining Tur\'an densities for hypergraphs has motivated the study of several variants. For $3$-graphs, Erd\H{o}s and S\'os~\cite{E-Sos} initiated the study of Tur\'an problems under a uniform density condition on large vertex subsets. For $d\in[0,1]$ and $\mu>0$, a $3$-graph $H=(V,E)$ is \emph{$(d,\mu,1)$-dense} if
\[
\left|\binom U3\cap E\right|\ge d\binom{|U|}{3}-\mu |V|^3
\]
for every $U\subseteq V$. The corresponding $1$-uniform Tur\'an density $\pi_1(F)$ of a given $k$-graph $F$ is the supremum of all $d$ such that, for every $\mu>0$ and every $n_0$, there exists an $F$-free $(d,\mu,1)$-dense $3$-graph with at least $n_0$ vertices.

With this notation, Erd\H{o}s and S\'os asked to determine $\pi_1(K^{(3)}_4)$ and $\pi_1(K^{(3)-}_4)$, where $K^{(3)-}_4$ is obtained from $K^{(3)}_4$ by removing one edge.
Recently, it was shown that $\pi_1(K^{(3)-}_4)=1/4$ by Glebov, Kr\'al' and Volec~\cite{k43minus-1}, and independently by Reiher, R\"odl and Schacht~\cite{k43minus-2}. 
The conjecture that $\pi_1(K_4^{(3)})=1/2$ remains open; R\"odl~\cite{k43-rodl} provided a randomized construction giving the lower bound in 1986. 
For further results on uniformly dense $3$-graphs, we refer the reader to Reiher's survey~\cite{reiher2020extremal}.

The study of Tur\'an problems in uniformly dense $k$-graphs has gained further momentum due to the work of Reiher, R\"odl and Schacht~\cite{RRS-Mantel,vanishing,k43minus-2}. In particular, they~\cite{RRS-Mantel} introduced a hierarchy of uniform density notions for $k$-graphs. We focus on the top non-trivial level.
Given $k\ge2$ and a $(k-2)$-graph $G^{(k-2)}$, let $\mathcal K_k(G^{(k-2)})$ be the family of $k$-subsets of $V(G^{(k-2)})$ spanning a copy of the complete $(k-2)$-graph $K_k^{(k-2)}$. When $k=3$, a $1$-graph on vertex set $U$ consists of the singletons of $U$, and hence $\mathcal{K}_3(G^{(1)})=\binom{U}{3}$.

\begin{definition}\label{def:k-2-dense}
Given integers $n\ge k\ge2$, let $d\in[0,1]$, $\mu>0$, and let $H=(V,E)$ be a $k$-graph with $|V|=n$. We say that $H$ is {\it $(d, \mu, k-2)$-dense} if
\begin{equation}\label{eq:j-dense}
\left|\mathcal{K}_k(G^{(k-2)})\cap E\right|\ge d\left|\mathcal{K}_k(G^{(k-2)})\right|-\mu n^k
\end{equation}
holds for all $(k-2)$-graphs $G^{(k-2)}$ with vertex set $V$.
\end{definition}

The corresponding \emph{$(k-2)$-uniform Tur\'an density} of a given $k$-graph $F$ is defined as
\begin{equation}\label{pi-k-2-definition}
\begin{split}
\pi_{k-2}(F) = \sup \{ d\in [0,1] & : \text{for\ every\ } \mu>0 \ \text{and\ } n_0\in \mathbb{N},\ \text{there\ exists\ an\ } F \text{-free} \\
		&\quad (d,\mu,k-2)\text{-dense}~ k \text{-graph~} H\ \text{with~} |V(H)|\geq n_0 \}.
	\end{split}
\end{equation}

Note that for $k=2$, $H$ is $(d,\mu,0)$-dense if $|E|\ge d\binom{|V|}{2}-\mu n^2$, since on any vertex set there are only two $0$-graphs: the one with empty edge set and the one with the empty set being an edge.  Therefore $\pi_0(F)=\pi(F)$ for every graph $F$.
In 2016, Reiher, R\"odl and Schacht~\cite{RRS-Mantel} generalized Mantel's theorem by showing that $\pi_{k-2}(F^{(k)})=2^{1-k}$, where $F^{(k)}$ is the unique $k$-graph on $(k+1)$ vertices with three edges. Moreover, they~\cite{vanishing} characterized all $k$-graphs $F$ with $\pi_{k-2}(F)=0$. As a consequence, they deduced a jump phenomenon: $\pi_{k-2}(F)$ is either $0$ or at least $k^{-k}$. 
Garbe, Kr\'al' and Lamaison~\cite{1/27} constructed $3$-graphs $F$ with $\pi_{1}(F)=1/27$. Subsequently, for all $k\ge 3$, the first author together with the latter two authors~\cite{LWZ-25} gave a regularity-based framework for studying $\pi_{k-2}(F)$ and constructed $k$-graphs $F$ with $\pi_{k-2}(F)=k^{-k}$.

Recently, there has been breakthrough progress on the $1$-uniform Tur\'an density of $3$-graphs
\cite{chen2022beyond,4/27,Ander,Ander-Wu,8/27,KKLT-25,King-SS-24,King-piga-SS-25}, yet results for general $k$-graphs
remain largely unknown. One reason why so many results recently have been proved about the $1$-uniform Tur\'an density of $3$-graphs is the development of a tool for its study, the so-called {\it palettes}.
For $3$-graphs, a {\it palette} $\mathscr{P}$ consists of a finite color set and a collection of admissible ordered triples of colors.
Roughly speaking, given a palette $\mathscr{P}$ with density $d$, randomly coloring pairs and keeping triples with admissible color patterns yields
$(d,\mu,1)$-dense $3$-graphs. This viewpoint was introduced by Reiher~\cite{reiher2020extremal} as a generalization of the construction of R\"odl~\cite{k43-rodl}.
As a breakthrough, 
Lamaison~\cite{Ander} recently proved that palettes determine the $1$-uniform Tur\'an density of every single $3$-graph; 
Kr\'al', Kučer\'ak, Lamaison and Tardos~\cite{KKLT-25} gave a general classification tool for palettes; 
King, Piga, Sales and Sch\"ulke~\cite{King-piga-SS-25} showed
that every real number that is the Lagrangian of a palette is the $1$-uniform Tur\'an density of a finite family of
$3$-graphs.

The starting point of this paper is the remark of Reiher, R\"odl and Schacht~\cite{vanishing}: for the $(k-2)$-uniform Tur\'an densities of $k$-graphs,
one can establish a theory that resembles to some extent the classical theory for graphs initiated by Tur\'an himself and developed further by Erd\H{o}s, Stone, Simonovits and many others. 
We develop such a theory through a $k$-uniform palette framework.

\subsection{Our results}

Analogous to the definition in~\eqref{pi-k-2-definition} for a single $k$-graph, we extend the notion of $(k-2)$-uniform Tur\'an density to families of $k$-graphs, and denote it by $\pi_{k-2}(\mathcal F)$.

Motivated by the role that palettes play in the study of $1$-uniform Tur\'an density of $3$-graphs, we first extend the ``palette'' method to arbitrary $k\ge 3$ and to arbitrary families of $k$-graphs.

\begin{definition}\label{def:k-palette}
For $k\ge 3$, a {\it $k$-palette} is a pair $\mathscr{P}=(\mathcal{C}, \mathcal T)$, where $\mathcal C$ is a finite set of colors and $\mathcal T\subseteq\mathcal C^k$ is a set of admissible ordered $k$-tuples. The density of $\mathscr{P}$ is defined by
\[
d(\mathscr P):={|\mathcal T|}/{|\mathcal C|^k}.
\] 
A $k$-graph $F$ is \emph{$\mathscr P$-colorable} if it admits a linear order $\prec$ on $V(F)$ and a coloring $\varphi: \binom{V(F)}{k-1} \to \mathcal{C}$ 
such that for every edge
$e=\{v_{i_1},v_{i_2},\dots,v_{i_k}\}\in E(F)$ with $v_{i_1}\prec v_{i_2}\prec\cdots\prec v_{i_k}$,
\[
\bigl(\varphi(e\setminus\{v_{i_1}\}),\varphi(e\setminus\{v_{i_2}\}),\dots,\varphi(e\setminus\{v_{i_k}\})\bigr)\in \mathcal T.
\]
Moreover, a family $\mathcal F$ of $k$-graphs is {\it $\mathscr{P}$-colorable} if there exists some $k$-graph $F\in \mathcal F$ such that $F$ is $\mathscr{P}$-colorable.
\end{definition}

The palette construction provides extremal examples of uniformly dense hypergraphs and therefore gives a natural lower bound for $\pi_{k-2}(\mathcal F)$. This motivates the following notion.

\begin{definition}
Given a family $\mathcal F$ of $k$-graphs, the {\it palette Tur\'an density} of $\mathcal F$ is defined by
\[
\pi^{\pal}_{k}(\mathcal F)= \sup \{d(\mathscr{P}): \mathscr{P}\text{ is a $k$-palette and every } F\in\mathcal{F}\text{ is not $\mathscr{P}$-colorable}\}.
\]
\end{definition}

Our first main theorem shows that the palette viewpoint exactly captures $(k-2)$-uniform Tur\'an density for arbitrary families of $k$-graphs.

\begin{theorem}\label{Thm:main theorem 1} 
Let $k\ge 3$ and $\mathcal F$ be a family of $k$-graphs. Then
\[
\pi_{k-2}(\mathcal F)=\pi^{\pal}_{k}(\mathcal F).
\]
\end{theorem}

When $\mathcal F=\{F\}$ is a singleton, \cref{Thm:main theorem 1} confirms a conjecture from~\cite[Conjecture~7.6]{LWZ-25}. In particular, when $k=3$, \cref{Thm:main theorem 1} generalizes a breakthrough result of Lamaison~\cite{Ander}, who proved that palettes determine the $1$-uniform Tur\'an density of single $3$-graphs.

Similarly to the original Tur\'an density, let $\Pi^{(k)}_{{\rm u},\infty}$, $\Pi^{(k)}_{{\rm u},{\rm fin}}$ and $\Pi^{(k)}_{\rm u}$ denote the sets of all possible $(k-2)$-uniform Tur\'an densities of arbitrary families, finite families, and single $k$-graphs, respectively. Then
\begin{equation}\label{eq:k-2-turan-density}
\Pi^{(k)}_{\rm u}\subseteq \Pi^{(k)}_{{\rm u},{\rm fin}}\subseteq \Pi^{(k)}_{{\rm u},\infty}.
\end{equation}

For $k=2$, these sets coincide. For $k\ge3$, much less is known.
A natural question is whether, for a given $\alpha\in[0,1]$, one can decide whether $\alpha \in \Pi^{(k)}_{\rm u}$ (or $\alpha \in \Pi^{(k)}_{\rm u,fin}$) by means of \cref{Thm:main theorem 1}. For $k=3$, Kr\'al', Ku\v{c}er\'ak, Lamaison and Tardos~\cite{KKLT-25}
gave a necessary and sufficient criterion in terms of homomorphisms between palettes, for verifying the
existence of a $3$-graph $F$ with $\pi_{1}(F)=\alpha$. Our next result establishes an analogous criterion for general $k$. Before stating it, we introduce the relevant notions.

\begin{definition}\label{def:k-palette-relation}
Given a $k$-palette $\mathscr P=(\mathcal C,\mathcal T)$, its \emph{reverse} is
\[
{\rm rev}(\mathscr P):=(\mathcal C,{\rm rev}(\mathcal T)),
\qquad
{\rm rev}(\mathcal T):=\{(c_k,c_{k-1},\dots,c_1):(c_1,\dots,c_k)\in\mathcal T\}.
\]
A \emph{homomorphism} from $\mathscr P=(\mathcal C,\mathcal T)$ to $\mathscr P'=(\mathcal C',\mathcal T')$ is a map $f:\mathcal C\to\mathcal C'$ such that for every
$(c_1,c_2,\dots,c_k)\in\mathcal{T}$, we have $(f(c_1),f(c_2),\dots,f(c_k))\in\mathcal{T}'$.
\end{definition}

The following two basic observations are immediate:
\begin{itemize}
\item A $k$-graph is $\mathscr P$-colorable if and only if it is ${\rm rev}(\mathscr P)$-colorable, by reversing the witnessing vertex order.
\item\label{item:homomorphism} If there is a homomorphism $\mathscr P\to\mathscr P'$, then every $\mathscr P$-colorable $k$-graph is also $\mathscr P'$-colorable.
\end{itemize}

Using these notions, we obtain the following characterization.

\begin{theorem}\label{thm: classification for single}
For $k\ge 3$, let ${\mathscr P}$ and ${\mathscr P}_0$ be two $k$-palettes. There exists a $k$-graph
$F$ that is ${\mathscr P}$-colorable but not ${\mathscr P}_0$-colorable if and only if there is no homomorphism from $\mathscr P$ to $\mathscr P_0$ and no homomorphism from $\mathscr P$ to ${\rm rev}(\mathscr P_0)$.
\end{theorem}

In addition, we prove a more general multi-palette classification theorem (see~\cref{thm: classification for multicase}) in \cref{sec: proof of multi-palette classification}, which gives a necessary and sufficient condition for the existence of a $k$-graph $F$ that is colorable by one set of $k$-palettes but not by another.

As mentioned above, previous works show that $0$, $k^{-k}$, and $2^{1-k}$ belong to $\Pi^{(k)}_{\rm u}$. Combining~\cref{Thm:main theorem 1} with these classification theorems, we determine several further values in $\Pi^{(k)}_{\rm u}$.

\begin{theorem}\label{thm:values of k-2-pi} 
For all integers $k\geq 3$ and $r\geq 2$, we have 
\begin{equation*}
\left\{ \ \frac{r-1}{r},\ \frac{(r-1)^2}{r^2},\ \frac{r-1}{2r}, \ \frac{(k-1)^k}{k^k},\ \frac{4(k-2)^{k-2}}{k^k},\ \frac{4(k-2)^{k-2}}{3k^k}\right\}\subseteq \Pi^{(k)}_{\rm u}.
\end{equation*}
\end{theorem}

Besides identifying further elements of $\Pi^{(k)}_{\rm u}$, it is also natural to ask which of the inclusions in~\eqref{eq:k-2-turan-density} are strict. As noted above, we have $\Pi^{(2)}_{\rm u}=\Pi^{(2)}_{\rm u,fin}$. For $k\ge 3$, however, the non-principal phenomena suggests that the first inclusion in~\eqref{eq:k-2-turan-density} may be strict. This motivates us to investigate whether an analogous
non-principal phenomenon occurs for the $(k-2)$-uniform Tur\'an density.

Our final main result establishes the existence of non-principal families for $(k-2)$-uniform Tur\'an density, in the strongest possible form, for families of size two.

\begin{theorem}\label{thm:non-principal family}
For every $k\ge 3$, there exist two $k$-graphs $F_1$ and $F_2$ such that
\[
\pi_{k-2}(\{F_1,F_2\})<\min\{\pi_{k-2}(F_1),\ \pi_{k-2}(F_2)\}.
\]
\end{theorem}

It is worth noting that, for the classical Tur\'an density, Balogh~\cite{Balogh2002} and Mubayi--Pikhurko~\cite{MubayiPikhurko2008}
proved the existence of non-principal $k$-graph families by establishing gaps between upper and lower bounds.
Unlike in~\cite{Balogh2002,MubayiPikhurko2008}, we prove~\cref{thm:non-principal family} by determining the exact values of $\pi_{k-2}(F_1)$, $\pi_{k-2}(F_2)$, and $\pi_{k-2}(\{F_1,F_2\})$.

\subsection{Notation}
For a positive integer $\ell$, let $[\ell]:=\{1,\dots,\ell\}$. For integers $a< b$, write $[a,b]:=\{a,a+1,\dots,b\}$. Given $k>i\ge1$ and an $i$-graph $G^{(i)}$, let $\mathcal K_k(G^{(i)})$ denote the family of $k$-subsets of $V(G^{(i)})$ that span a copy of $K_k^{(i)}$ in $G^{(i)}$. Given a $k$-graph $H$, let
\[
\partial H:=\{S\in\binom{V(H)}{k-1}: S\subset e\text{ for some }e\in E(H)\}
\]
be the \emph{shadow} of $H$. For real numbers $x,y,z$, we write $x=y\pm z$ to mean $y-z\le x\le y+z$.

Given distinct positive integers $i_1,\dots,i_\ell$, we write $\llbracket i_1,i_2,\dots,i_\ell\rrbracket$
for the set $\{i_1,i_2,\dots,i_\ell\}$ with the natural order $i_1<i_2<\cdots<i_\ell$. Let $R_k(n,r)$ be the minimum integer $N$ such that every $r$-edge-coloring of the complete $k$-graph on $N$ vertices contains a monochromatic complete subgraph on $n$ vertices. The hypergraph Ramsey theorem implies the existence of $R_k(n,r)$.

Some $k$-graphs considered in this paper are equipped with a linear order on their vertex set; we call them \emph{ordered $k$-graphs}. We explicitly use the adjective ``ordered'' whenever the order is part of the structure. If $H$ is an ordered $k$-graph, we write $\prec_H$ for its vertex order, omitting the subscript when clear from the context. Moreover, we write $\widetilde{H}$
for the unordered $k$-graph obtained from $H$ by forgetting its
linear order.  
If $H$ is an unordered $k$-graph and $\prec$ is a linear order on $V(H)$, then $H^\prec$ denotes the resulting ordered $k$-graph. Conversely, if $H^\prec$ is an ordered $k$-graph, then $H$ denotes the underlying unordered $k$-graph.

\subsection{Organization}
The paper is organized as follows.  In \cref{sec:proof of palette density theorem}, we prove \cref{Thm:main theorem 1}, reducing $(k-2)$-uniform Tur\'an density to palette Tur\'an density via reduced hypergraphs and a density-preserving contraction lemma.
In \cref{sec-regular-method}, we state the basic definitions and results from the hypergraph regularity method used later. In~\cref{sec: proof of single palette classification}, we give a proof of \cref{thm: classification for single}, which is helpful for understanding the proof of \cref{thm: classification for multicase}. In~\cref{sec: proof of multi-palette classification}, we introduce the notion of symmetrization of $k$-palettes and prove two generalizations of~\cref{thm: classification for single}; see~\cref{thm: classification for multicase,Cor:classification for multi case}.
In~\cref{sec:example}, we prove~\cref{thm:values of k-2-pi,thm:non-principal family} as applications of~\cref{thm: classification for multicase}. Finally, we conclude with some remarks and further problems in the last section.

%%%%%%%%%%%%%%%%%%%%%%%%%%%%%%%%%%%%%%%
\section{Reducing $(k-2)$-uniform Tur\'an density to palette Tur\'an density}\label{sec:proof of palette density theorem}

In this section, we prove~\cref{Thm:main theorem 1}. Our proof is inspired by the strategy of Lamaison~\cite{Ander} for single $3$-graphs, but it requires a 
substantially more careful compactness argument in order to handle arbitrary families of $k$-graphs.

We begin with the following theorem, which yields the lower bound $\pi_{k-2}(\mathcal{F})\geq\pi^{\pal}_k(\mathcal{F})$.

\begin{theorem}[\cite{LWZ-25}]\label{thm: palette lower bound}
Let $k\geq 3$, $\mathcal F$ be a family of $k$-graphs and $\mathscr{P}$ be a $k$-palette.
 If $\mathcal{F}$ is not $\mathscr{P}$-colorable, then $\pi_{k-2}(\mathcal{F})\geq d(\mathscr{P})$.
\end{theorem}  

\cref{thm: palette lower bound} is obtained by randomly generating $\mathcal{F}$-free $(d(\mathscr{P}),\mu,k-2)$-dense $k$-graphs from the palette $\mathscr{P}$. A proof in the case where $\mathcal{F}$ consists of a single $k$-graph can be found in~\cite{LWZ-25}, and the same argument still works for arbitrary families $\mathcal{F}$ of $k$-graphs.

Therefore, in order to prove~\cref{Thm:main theorem 1}, it remains to establish the reverse inequality
\[
\pi_{k-2}(\mathcal{F})\le \pi^{\pal}_k(\mathcal{F}).
\]
A key concept in this part of the proof is the notion of reduced $k$-graphs,
which provides a structural abstraction of the original hypergraph and captures
the essential combinatorial features relevant to the $(k-2)$-uniform Tur\'an
density. We introduce this notion and prove the auxiliary lemmas needed later.

\subsection{Reduced $k$-graphs}\label{subsec1:Preliminary}

We now define reduced $k$-graphs, following the formalism introduced in~\cite{RRS-Mantel,LWZ-25}.

\begin{definition}\label{def-reduced-graph}
Given a finite index set $I \subset \mathbb N$ with $| I |=m$, for each $(k-1)$-set $\mathcal X\in \binom{I}{k-1}$, let $\mathcal {P}_{\mathcal X}$ denote a finite nonempty vertex set such that for any two distinct $\mathcal X, \mathcal X'\in \binom{I}{k-1}$ the sets $\mathcal {P}_{\mathcal X}$ and $\mathcal {P}_{\mathcal X'}$ are disjoint.
For each  $k$-set $\mathcal Y \in \binom{I}{k}$, let $\mathcal A_{\mathcal Y}$ denote a $k$-partite 
$k$-graph with a vertex partition $\{\mathcal {P}_{\mathcal X}: \mathcal X \in \binom{\mathcal Y}{k-1}\}$.
Then the $\binom{m}{k-1}$-partite $k$-graph $\mathcal A$ with
\[
V(\mathcal A)=\bigcup_{\mathcal X \in \binom{I}{k-1}}\mathcal {P}_{\mathcal X}, ~\text{and}~~ E(\mathcal A)=\bigcup_{\mathcal Y\in \binom{I}{k}} E(\mathcal A_{\mathcal Y})
\]
is called an {\it $m$-reduced $k$-graph} with \emph{index set} $I$, \emph{vertex classes} $\mathcal {P}_{\mathcal X}$ and {\it constituents}  $\mathcal A_{\mathcal Y}$ and is denoted by $\mathcal A=(I, \mathcal {P}_{\mathcal X}, \mathcal A_{\mathcal Y})$. 
\end{definition}

For brevity, we often simply write ``$\mathcal A$ is an $m$-reduced $k$-graph''  instead of ``$\mathcal A$ is an $m$-reduced $k$-graph with index set $[m]$, vertex classes $\mathcal {P}_{\mathcal X}$ and
constituents $\mathcal A_{\mathcal Y}$''. 
Given an $m$-reduced $k$-graph $\mathcal A$ and a real number $d \in[0, 1] $, we say that 
$\mathcal A$ is {\it $d$-dense} if 
\[
|E(\mathcal A_{\mathcal Y})|\ge d \cdot \prod_{{\mathcal X}\in \binom{\mathcal Y}{k-1}}|\mathcal {P}_{\mathcal X}|
\]
for every $\mathcal Y \in \binom{[m]}{k}$.

We next describe which configurations in a reduced $k$-graph correspond to copies of forbidden $k$-graphs in the original $(d,\mu,k-2)$-dense hypergraph.

\begin{definition}\label{def-reduce-map}
A \emph{reduced map} from a $k$-graph $F$ to an $m$-reduced $k$-graph $\mathcal A=([m], \mathcal P_{\mathcal X}, \mathcal A_{\mathcal Y})$ is a pair $(\phi,\psi)$ such that
\begin{enumerate}[label=(\arabic*)]
\item $\phi:V(F)\to [m]$ and $\psi: \partial F\to V(\mathcal A)$ are both injective;
\item if $S=\{i_1,i_2,\dots,i_{k-1}\}\in \partial F$, then $\mathcal X=\{\phi(i_1),\phi(i_2),\dots,\phi(i_{k-1})\}\in \binom{[m]}{k-1}$ and $\psi(S)\in \mathcal P_{\mathcal X}$;
\item if $e=\{i_1,i_2,\dots, i_{k} \} \in E(F)$, then $\mathcal Y= \{\phi(i_1), \phi(i_2), \dots, \phi(i_{k})\}\in \binom{[m]}{k}$ and
\[
\{\psi(e\setminus \{i_{1}\}), \psi(e\setminus \{i_{2}\}), \dots, \psi(e\setminus \{i_{k}\})\}\in E( \mathcal A_{\mathcal Y}).
\]
\end{enumerate}
Given an $m$-reduced $k$-graph $\mathcal A$ and a family $\mathcal{F}$ of $k$-graphs, we say that $\mathcal A$ \emph{embeds} $\mathcal{F}$ if there exists $F\in\mathcal F$ such that there is a reduced map from $F$ to $\mathcal A$.
\end{definition}

With this concept in hand, the first author together with the latter two authors~\cite{LWZ-25} established a general result reducing the problem of proving an upper bound on the $(k-2)$-uniform Tur\'an density of a single $k$-graph to the problem of embedding it into sufficiently dense reduced $k$-graphs. Although the statement in~\cite{LWZ-25} is formulated for single $k$-graphs, its proof works for families of $k$-graphs as well.

\begin{theorem}[\cite{LWZ-25}] \label{thm-turan-reduced}
Let $\mathcal{F}$ be a family of $k$-graphs with $k\ge 3$, and let $d\in [0,1]$.
If for every $\eps>0$ there exists $m_0\in \mathbb N$ such that for every $m\geq m_0$, every $(d+\eps)$-dense $m$-reduced $k$-graph $\mathcal A$ embeds $\mathcal{F}$, then $\pi_{k-2}(\mathcal{F})\leq d$.
\end{theorem}

Note that for any family of $k$-graphs $\mathcal{F}$ with $\pi_{k-2}(\mathcal{F})=d$ and any $\eps > 0$, by \cref{thm-turan-reduced}, there exists a $(d-\eps)$-dense $m$-reduced $k$-graph $\mathcal A$ that does not embed any $k$-graph from $\mathcal{F}$ for sufficiently large $m$. If we can ``extract'' a $k$-palette $\mathscr{P}$ from $\mathcal A$ such that $\mathcal{F}$ is not $\mathscr{P}$-colorable and  $d(\mathscr{P})\approx d$, then we may prove that $\pi_{k-2}(\mathcal{F})\le \pi^{\rm pal}_{k}(\mathcal{F})$. 
Indeed, this is precisely the strategy we shall follow. Since $\mathcal F$ may be infinite, however, a more delicate analysis is required.

Therefore, in what follows, we turn to the relationship between reduced $k$-graphs and $k$-palettes.
Given a $k$-palette $\mathscr{P}=(\mathcal{C},\mathcal{T})$ with density $d$ and a positive integer $m$, we define a $d$-dense $m$-reduced $k$-graph $\mathcal{A}(\mathscr{P})$ as follows. For each $\mathcal X\in \binom{[m]}{k-1}$, let $\mathcal P_{\mathcal X}$ be a copy of $\mathcal C$. For each $\mathcal Y=\llbracket i_1,i_2,\ldots,i_k\rrbracket\in\binom{[m]}{k}$, let $\mathcal X_j=\mathcal Y\setminus\{i_j\}$ for $j\in[k]$, and define the constituent $\mathcal A_{\mathcal Y}$ so that, for every $(c_1,c_2,\dots,c_k)\in\mathcal T$, the set
\[
\{c_{\mathcal X_1}^1,c_{\mathcal X_2}^2,\dots,c_{\mathcal X_k}^k\}
\]
is an edge of $\mathcal A_{\mathcal Y}$, where $c_{\mathcal X_j}^j$ denotes the copy of $c_j$ in $\mathcal P_{\mathcal X_j}$. We now establish the following lemma for finite families of $k$-graphs.

\begin{lemma}\label{lm:colorable equals embedding}
Let $\mathcal{F}$ be a family of $k$-graphs each having at most $n$ vertices, and let $\mathscr{P}=(\mathcal{C},\mathcal{T})$ be a $k$-palette. Then $\mathcal{F}$ is $\mathscr{P}$-colorable if and only if the $n$-reduced $k$-graph $\mathcal{A}(\mathscr{P})$ embeds $\mathcal{F}$.
\end{lemma}

\begin{proof}
Suppose first that $\mathcal{F}$ is $\mathscr{P}$-colorable. Then there exists $F\in \mathcal{F}$ such that $F$ is $\mathscr{P}$-colorable. Set $m=|V(F)|\le n$. Let $v_1\prec v_2\prec \cdots \prec v_m$
be an ordering of $V(F)$ and let $\varphi:\binom{V(F)}{k-1}\to \mathcal C$ be a function witnessing that $F$ is $\mathscr{P}$-colorable. Define $\phi:V(F)\to [n]$ by $\phi(v_i)=i$ for each $i\in[m]$. Next, for every $S=\{v_{i_1},v_{i_2},\dots,v_{i_{k-1}}\}\in \partial F$, define $\psi(S)$
to be the copy of the color $\varphi(S)$ in the vertex class $\mathcal P_{\mathcal{I}}$ of $\mathcal A(\mathscr P)$, where $\mathcal{I}=\{i_1,i_2,\dots, i_{k-1}\}$. Then condition~(2) in~\cref{def-reduce-map} holds by construction. Moreover, if
\[
e=\{v_{i_1},v_{i_2},\dots,v_{i_k}\}\in E(F)
\quad\text{with}\quad
v_{i_1}\prec v_{i_2}\prec \cdots \prec v_{i_k},
\]
then, since $\varphi$ witnesses that $F$ is $\mathscr P$-colorable,
\[
\bigl(\varphi(e\setminus\{v_{i_1}\}),\varphi(e\setminus\{v_{i_2}\}),\dots,\varphi(e\setminus\{v_{i_k}\})\bigr)\in\mathcal T.
\]
By the definition of $\mathcal A(\mathscr P)$, this implies that
\[
\{\psi(e\setminus\{v_{i_1}\}),\psi(e\setminus\{v_{i_2}\}),\dots,\psi(e\setminus\{v_{i_k}\})\}\in E(\mathcal A_{\phi(e)}),
\]
where $\phi(e)=\{\phi(v):v\in e\}$. Hence $(\phi, \psi)$ is a reduced map from $F$ to $\mathcal A(\mathscr P)$, so $\mathcal A(\mathscr P)$ embeds $\mathcal F$.

Conversely, suppose that $\mathcal A(\mathscr P)$ embeds $\mathcal F$. Then there exist $F\in\mathcal F$ and a reduced map $(\phi,\psi)$ from $F$ to $\mathcal A(\mathscr P)$. The map $\phi:V(F)\to [n]$ induces an ordering ``$\prec$'' on $V(F)$: $v_i\prec v_j$ if and only if $\phi(v_i)<\phi(v_j)$. For each $S\in \partial F$, define $\varphi(S)\in\mathcal C$ to be the color represented by the vertex $\psi(S)\in \mathcal P_{\phi(S)}$, where $\phi(S)=\{\phi(v):v\in S\}$. Extend $\varphi$ arbitrarily to all $(k-1)$-subsets of $V(F)$ not in $\partial F$.
It is then immediate from the definition of
$\mathcal A(\mathscr P)$ that the order $\prec$ and the function $\varphi$
witness that $F$ is $\mathscr P$-colorable. Hence $\mathcal F$ is
$\mathscr P$-colorable.
\end{proof}

Note that \cref{lm:colorable equals embedding} applies only to finite families of $k$-graphs. 
For infinite families, there is no fixed $n$ for which the $n$-reduced 
$k$-graph $\mathcal A(\mathscr P)$ embeds $\mathcal F$.

The final ingredient needed in the proof of \cref{Thm:main theorem 1} is a ``density-preserving contraction lemma'' for $N$-reduced $k$-graphs. Roughly speaking, it asserts that every sufficiently large $N$-reduced $k$-graph can be contracted to an $m$-reduced $k$-graph with bounded size of vertex classes while almost preserving the density. 

\begin{lemma}[Density-preserving contraction lemma]\label{lem:key lemma for palette theorem}
For every $\varepsilon>0$ and integer $k\ge 3$, there exists an integer
$s=s(\varepsilon,k)$ such that for every $m\in \mathbb N$ there exists an integer
$N_0=N_0(m,\varepsilon,k)$ with the following property. For every $N\geq N_0$, if $\mathcal A$ is a
$d$-dense $N$-reduced $k$-graph, then there exist a subset
$M\subseteq [N]$ of size $m$ and, for each $\mathcal X\in \binom{M}{k-1}$, a multiset
$S_{\mathcal X}$ of $s$ vertices 
(not necessarily distinct) in $\mathcal P_{\mathcal X}$ such that the 
$m$-reduced $k$-graph induced by $\mathcal A$ on the multisets
$S_{\mathcal X}$ is $(d-\varepsilon)$-dense.
\end{lemma}

Whenever a multiset $S_{\mathcal X}$ contains repeated vertices, we regard its elements as labeled copies. 
We postpone the proof of \cref{lem:key lemma for palette theorem} to the next subsection and now complete the proof of \cref{Thm:main theorem 1}.

\begin{proof}[Proof of \cref{Thm:main theorem 1}]
If $\mathcal F=\emptyset$, then both sides are equal to $1$, so there is nothing
to prove. Hence we may assume that $\mathcal F$ is nonempty.
For each integer $n\ge 1$, let $\mathcal F_n\subseteq \mathcal F$ denote the subfamily consisting of all members of $\mathcal F$ with at most $n$ vertices. Then
\[
\mathcal F=\bigcup_{n=1}^{\infty}\mathcal F_n.
\]

Set $\pi:=\pi_{k-2}(\mathcal F)$. Fix $\varepsilon>0$.
Choose $n_0$ such that $\mathcal F_{n_0}\neq\emptyset$.
For every $n\ge n_0$, we shall construct a $k$-palette
$\mathscr P_n=(\cc_n,\mathcal T_n)$ with density at least $\pi-\frac{2\varepsilon}{3}$
such that $\mathcal F_n$ is not $\mathscr P_n$-colorable and
$|\cc_n|$ depends only on $\varepsilon$ and $k$.

We first claim that, for every $n\geq n_0$, there exists $N'(n)$ such that for every $N\geq N'(n)$ there exists an $N$-reduced
$k$-graph $\mathcal A_n$ of density at least $\pi-\varepsilon/3$ which does
not embed $\mathcal F_n$. Indeed, let $N'(n)$ be the constant $m_0$ obtained by applying~\cref{thm-turan-reduced} with
$d=\pi-\varepsilon/3$. Since $\mathcal F_n\subseteq\mathcal F$, we
have
\[
\pi_{k-2}(\mathcal F_n)\geq \pi_{k-2}(\mathcal F)=\pi.
\]
If no such $\mathcal A_n$ existed, then every $(\pi-\varepsilon/3)$-dense
$N$-reduced $k$-graph would embed $\mathcal F_n$. Hence, for every
$\eta>0$, every $(\pi-\varepsilon/3+\eta)$-dense $N$-reduced $k$-graph would
also embed $\mathcal F_n$. Applying~\cref{thm-turan-reduced}, we would obtain
\[
\pi_{k-2}(\mathcal F_n)\leq \pi-\frac{\varepsilon}{3},
\]
which is a contradiction. This proves the claim.

Fix $n$. Let $s=s({\varepsilon}/{3},k)$, $m=R_k(n,2^{s^k})$
and $N_0=N_0(m,{\varepsilon}/{3},k)$ be given by \cref{lem:key lemma for palette theorem}.
Let $N=\max\{N'(n),N_0\}$. Applying \cref{lem:key lemma for palette theorem} to $\mathcal A_n$, we obtain a subset
$M\subseteq[N]$ of size $m$ and, for each $\mathcal X\in\binom{M}{k-1}$, a multiset
$S_{\mathcal X}$ of $s$ vertices in $\mathcal P_{\mathcal X}$ such that the $m$-reduced
$k$-graph $\mathcal A'$ on these multisets is $(\pi-\frac{2\varepsilon}{3})$-dense.

Let $v_{\mathcal X}^1,\dots,v_{\mathcal X}^s$ be the vertices in 
$S_{\mathcal X}$
for each $\mathcal X\in\binom{M}{k-1}$, where repetitions are allowed. For each $\mathcal Y\in\binom{M}{k}$, the constituent
$\mathcal A'_{\mathcal Y}$ is determined by which of the $s^k$ possible $k$-tuples
\[
\{v_{\mathcal X_1}^{i_1},v_{\mathcal X_2}^{i_2},\dots,v_{\mathcal X_k}^{i_k}\},
\qquad (i_1,\dots,i_k)\in [s]^k,
\]
form edges, where $\binom{\mathcal Y}{k-1}=\{\mathcal X_1,\dots,\mathcal X_k\}$. Thus each constituent $\mathcal A'_{\mathcal Y}$ can be labeled by a subset of $[s]^k$, and there are exactly $2^{s^k}$ possible labels.

Since $m=R_k(n,2^{s^k})$, there exists a subset 
$M_0\subseteq M$ of size $n$ such that all constituents
$\mathcal A'_{\mathcal Y}$ with $\mathcal Y\in\binom{M_0}{k}$ receive the same label
$L\subseteq [s]^k$. Since $\mathcal A'$ is $(\pi-\frac{2\varepsilon}{3})$-dense, every constituent
$\mathcal A'_{\mathcal Y}$ with $\mathcal Y\in\binom{M_0}{k}$ contains at least $(\pi-\frac{2\varepsilon}{3})s^k$ edges. Hence
\[
\frac{|L|}{s^k}\ge \pi-\frac{2\varepsilon}{3}.
\]

Define a $k$-palette $\mathscr P_n=(\mathcal C_n,\mathcal T_n)$ by $\mathcal C_n=[s]$ and $\mathcal T_n=L$.
Then
\[
d(\mathscr P_n)=\frac{|L|}{s^k}\ge \pi-\frac{2\varepsilon}{3}.
\]
Moreover, after identifying each multiset $S_{\mathcal X}$ with the color set $[s]$, the $n$-reduced $k$-graph obtained from $\mathcal A'$ by restricting the index set to $M_0$ is precisely of the form $\mathcal A_n(\mathscr P_n)$.
Since this reduced $k$-graph is a subgraph of the blow-up of $\mathcal A_n$, it also does not embed $\mathcal F_{n}$. Therefore, by \cref{lm:colorable equals embedding}, the family $\mathcal F_{n}$ is not $\mathscr P_n$-colorable. 

Since $s$ is a constant, there are only finitely many possible
palettes $\mathscr P_n$.
Hence there exists an infinite increasing
sequence
\[
  n_1<n_2<\cdots
\]
with $n_1>n_0$ such that all palettes $\mathscr P_{n_j}$ are equal.
Let $\mathscr P=\mathscr P_{n_1}$.
We claim $\mathcal F$ is not $\mathscr P$-colorable. Indeed, suppose there is some $k$-graph $F\in \mathcal{F}$ that is $\mathscr P$-colorable. Choose $j$ such that $n_j\geq |V(F)|$. Then
$F\in\mathcal F_{n_j}$. Since $\mathcal F_{n_j}$ is not
$\mathscr P_{n_j}$-colorable and $\mathscr P_{n_j}=\mathscr P$, the graph $F$
is not $\mathscr P$-colorable, leading to a contradiction.

Then $\pi_k^{\pal}(\mathcal F)\ge d(\mathscr P)\ge \pi-\frac{2\varepsilon}{3}$. 
Since $\varepsilon>0$ was arbitrary, we conclude that $\pi_k^{\pal}(\mathcal F)\ge \pi_{k-2}(\mathcal F)$.
Together with \cref{thm: palette lower bound}, this yields
\[
\pi_k^{\pal}(\mathcal F)=\pi_{k-2}(\mathcal F).
\]
\end{proof}

\subsection{Density-preserving contraction lemma}

We establish the proof of~\cref{lem:key lemma for palette theorem} in this subsection.
Given a set $V$, a multiset $X$ contained in $V$, and a function
$f \colon V \to [0,1]$, we denote by $\bar{f}(X)$ the average value of $f$ on $X$.
We shall use the following concentration inequality, in which a vector is viewed as a multiset. 

\begin{lemma}[Hoeffding's inequality]\label{hoe-ineq}
Let $f\colon V\to [0,1]$ be a function, $t$ be a positive integer, and $\varepsilon>0$. Suppose that a vector
$X=(x_1,x_2,\ldots,x_t)$ is sampled uniformly at random from $V^t$. Then
\[
\Pr\bigl(\bar{f}(X)<\bar{f}(V)-\varepsilon\bigr)\le e^{-2\varepsilon^2 t}.
\]
\end{lemma}

The following lemma states that in a reduced $k$-graph, despite excluding a small set of vertices for each possible extension of a $(k-1)$-tuple, one can still choose a large index subset on which all these forbidden sets can be simultaneously avoided.

\begin{lemma}\label{lem: reduced-ramsey lem}
For all integers $m\ge k\ge 3$, there exists $n\in \mathbb N$ with the following
property. Let $\mathcal A$ be an $n$-reduced $k$-graph. Suppose that for every $\mathcal X\in \binom{[n]}{k-1}$ and every
$i_k\in [n]\setminus \mathcal X$, we are given a subset
$B_{\mathcal X}^{i_k}\subset \mathcal P_{\mathcal X}$ satisfying
$|B_{\mathcal X}^{i_k}|\le \frac{1}{3k}\,|\mathcal P_{\mathcal X}|$.
Then there exists a subset $M\subseteq [n]$ with $|M|=m$ such that for every
$\mathcal X\in \binom{M}{k-1}$, one can choose a vertex
$v_{\mathcal X}\in \mathcal P_{\mathcal X}$ satisfying
$v_{\mathcal X}\notin B_{\mathcal X}^{i_k}
~\text{for all } i_k\in M\setminus \mathcal X$ .
\end{lemma}

\begin{proof}
Assume for contradiction that for every $M\in \binom{[n]}{m}$, 
there exists some
$\mathcal{X}=\{ x_1,x_2,\ldots,x_{k-1}\}
\in \binom{M}{k-1}$
such that
\[
\bigcup_{x_k\in M\setminus\mathcal X}
B^{x_k}_{\mathcal X}
=\mathcal P_{\mathcal X}.
\]
We label each ordered $m$-set $M=\llbracket u_1,u_2,\ldots,u_m\rrbracket\in \binom{[n]}{m}$ by the $(k-1)$-subset 
\[
\mathcal L(M):=\{\,i\in[m]:u_i\in \mathcal X\,\}\in \binom{[m]}{k-1}.
\]

Let $n=R_m(2km,\binom{m}{k-1})$.
By Ramsey's theorem, there exists a subset $J\subseteq[n]$ of size $2km$ such that every $M\in \binom{J}{m}$ receives the same label
$\mathcal L=\{\ell_1,\dots,\ell_{k-1}\}\in\binom{[m]}{k-1}$.
Relabeling the indices if necessary, we may assume that
$J=[2km]$.

Now let $\mathcal I=\{2m,4m,\dots,2(k-1)m\}$.
For each index $i_k\in[2km]\setminus\mathcal I$, we have
$ |B_{\mathcal I}^{i_k}|\le \frac{1}{3k}\,|\mathcal P_{\mathcal I}|$.
By double counting, at most $\frac{2}{3}|\mathcal P_{\mathcal I}|$
vertices in $\mathcal P_{\mathcal I}$ belong to at least $m$
of the sets $B_{\mathcal I}^{i_k}$.
Therefore, there exists a vertex
$v^*\in\mathcal P_{\mathcal I}$ that belongs to at most $m-1$
of these sets.

We now extend $\mathcal I$ to an ordered $m$-set
$M_0=\llbracket u_1',\dots,u_m'\rrbracket\subseteq J$ such that the positions indexed by $\mathcal L=\{\ell_1,\ldots,\ell_{k-1}\}$ are occupied by the elements of $\mathcal I$, namely
\[
u'_{\ell_j}=2jm \qquad\text{for every } j\in[k-1],
\]
and such that for every $u_i'\in M_0\setminus \mathcal I$, we have $v^*\notin B_{\mathcal I}^{u_i'}$.
This is possible because $v^*$ belongs to at most $m-1$ forbidden sets, whereas there are $2m-(k-1)\ge m$ available indices in $\{2im+1,2im+2,\ldots,2(i+1)m-1\}$ for $0\le i\le k-1$.
This ensures that
\[
v^*\notin
\bigcup_{j_k\in M_0\setminus\mathcal I}
B_{\mathcal I}^{j_k}.
\]

However, since all $m$-sets of $J$ receive the same label
$\mathcal L$, $M_0$ is labeled by $\mathcal L$, which implies
\[
\bigcup_{j_k\in M_0\setminus\mathcal I}
B_{\mathcal I}^{j_k}
=\mathcal P_{\mathcal I}.
\]
This contradicts the choice of $v^*$.
\end{proof}

\begin{proof}[Proof of \cref{lem:key lemma for palette theorem}]
We begin with a brief outline. The proof proceeds in $r$ rounds. In each round, we choose a multiset of $t$ vertices from every relevant vertex class and then apply \cref{lem: reduced-ramsey lem} to shrink the index set (denoted $M_i$) to a smaller index set (denoted $M_{i+1}$), on which all these choices simultaneously satisfy the required density conditions. After $r$ rounds, we will have selected $s=rt$ vertices in each surviving vertex class. The accumulated density losses are controlled such that every constituent still contains at least $(d-\varepsilon)s^k$ edges.

We now turn to the details. Let $r=\lceil k(k-1)\varepsilon^{-1}\rceil$, $t=\lceil 2k^2\varepsilon^{-2}\log(3kr^{k-1})\rceil$, and set $s=rt$. 
Starting from $m_r=m$, choose integers
$m_{r-1},m_{r-2},\ldots,m_0$
recursively such that \cref{lem: reduced-ramsey lem} can be applied with parameters $m_i$ and $m_{i-1}$ for each $i\in[r]$, and let $N_0:=m_0$.
Let $\mathcal A$ be a $d$-dense $N_0$-reduced $k$-graph (this suffices for the proof: for any $N>N_0$, a $d$-dense $N$-reduced $k$-graph contains a $d$-dense $N_0$-reduced $k$-graph as a subgraph). We recursively construct index sets $[N_0]=M_0\supseteq M_1\supseteq \cdots \supseteq M_r$ with
$|M_i|=m_i$.

For each round $i\in [r]$ and each $\mathcal X \in \binom{M_{i-1}}{k-1}$, we select a multiset $T_{\mathcal X}^i \subseteq \mathcal P_{\mathcal X}$ of size $t$. Finally, for every  $\mathcal X \in \binom{M_r}{k-1}$, we define $S_{\mathcal X}=\bigcup_{i=1}^r T_{\mathcal X}^i$, and hence $S_{\mathcal X}$ is a multiset of size $s$ in  $\mathcal P_{\mathcal X}$.
Given a $k$-set $\mathcal Y=\llbracket y_1,\dots,y_k\rrbracket\in \binom{M_r}{k}$ of indices, let $\mathcal X_i=\mathcal Y\setminus\{y_i\}$. Our goal is to prove that
\begin{equation}\label{eq:goal contraction}
|E(S_{\mathcal X_1},\dots,S_{\mathcal X_k})|\ge (d-\varepsilon)s^k.
\end{equation}
	
To achieve this, it suffices to ensure that for every choice of distinct rounds
$i_1,\dots,i_k\in[r]$ and every $\mathcal Y\in \binom{M_r}{k}$, we have
\begin{equation}\label{eq:goal contraction2}
|E(T_{\mathcal X_1}^{i_1},\dots,T_{\mathcal X_k}^{i_k})|
\ge \left(d-\frac{\varepsilon}{2}\right)t^k.
\end{equation}
Since each $k$-tuple of distinct rounds contributes at least $(d-\frac{\varepsilon}{2})t^k$ edges, summing over all such tuples gives
\[|E(S_{\mathcal X_1},\dots,S_{\mathcal X_k})|
\ge \frac{r!}{(r - k)!} (d-\tfrac{\varepsilon}{2}) t^k \ge (d-\varepsilon)s^k.\]

To ensure that \cref{eq:goal contraction2} holds, we need to maintain a stronger family of inductive invariants. For $j\in\{0,1,\dots,r\}$, let
\[
W_j:=\Bigl\{(w_1,\dots,w_k):
0\le w_i\le j \text{ for all } i\in[k], \text{ and all nonzero entries are distinct}\Bigr\}.
\]
Note that $|W_j|=1+\sum_{i=1}^k \binom{k}{i}\binom{j}{i}\cdot i!$ if $k\leq j$ and $|W_j|=1+\sum_{i=1}^j\binom{k}{i}\binom{j}{i} \cdot i!$ if $k> j$. 
For $\bar{w}=(w_1,w_2,\dots,w_k)\in W_j$, let
\[
|\bar w|_+ := |\{i\in[k]:w_i\ne 0\}|
\qquad\text{and}\qquad
w_{\max}:=\max\{w_1,\dots,w_k\}.
\]
Moreover, for $\mathcal Y=\llbracket y_1,\dots,y_k\rrbracket\in \binom{M_{w_{\max}}}{k}$, let
\[
\mathcal X_i=\mathcal Y\setminus \{y_i\}
\quad \text{~for each~} i\in[k].
\]
Recall that  $T_{\mathcal X}^i \subseteq \mathcal P_{\mathcal X}$ is a multiset of size $t$ that we will select for each $i\in [r]$ and let $T_{\mathcal X}^0=\mathcal P_{\mathcal X}$. Hence, it suffices to ensure that the following invariants throughout the construction, 
\begin{align}\label{eqlm2.6.2}
\nu(T_{\mathcal X_1}^{w_1},\dots,T_{\mathcal X_k}^{w_k})  \ge d-\frac{\varepsilon \cdot|\bar{w}|_+ }{2k} ~\text{ for each}~ \bar{w}\in W_r  ~\text{and}~ \mathcal Y\in \binom{M_{w_{\max}}}{k}, 
\end{align}
where $\nu(T_{\mathcal X_1}^{w_1},\dots,T_{\mathcal X_k}^{w_k})$ denotes the edge density of $\mathcal A$
induced on $T_{\mathcal X_1}^{w_1},\dots,T_{\mathcal X_k}^{w_k}$.

We now prove \cref{eqlm2.6.2} by induction on $j$. 
Assume that $M_{j-1}$ has already been constructed, together with multisets  $T_{\mathcal X}^{i}$ for each $i\in [j-1]$ and $\mathcal X\in \binom{M_{j-1}}{k-1}$ and that \cref{eqlm2.6.2} holds for all $\bar{w}\in W_{j-1}$.
We now construct $M_j$ and the multisets $T_{\mathcal X}^{j}$ for each $\mathcal X\in \binom{M_{j}}{k-1}$ and verify that all invariants hold for $\bar{w}\in W_{j}$. 
%We select multisets $T_{\mathcal X}^j$ for all $(k-1)$-sets $\mathcal X\subseteq M_{j-1}$ and then apply \cref{lem: reduced-ramsey lem} to obtain a subset $M_j\subseteq M_{j-1}$ of size $m_j$ so that the invariants remain valid.
	
Fix $\mathcal X\in \binom{M_{j-1}}{k-1}$. Let $(\mathcal P_{\mathcal X})^t$ denote the set of vectors of length $t$ with entries in $\mathcal P_{\mathcal X}$ and view a vector  as a multiset. We now start to choose $T_{\mathcal X}^j \in (\mathcal P_{\mathcal X})^t$. For each $z\in M_{j-1}\setminus\mathcal X$, let $B_{\mathcal X}^z\subseteq (\mathcal P_{\mathcal X})^t$ denote the set of vectors which would violate at least one invariant (with $\mathcal{Y}=\mathcal{X}\cup \{z\}$) if chosen as $T_{\mathcal X}^j$. We claim that
\[ |B_{\mathcal X}^z| \le \frac{1}{3k}|\mathcal P_{\mathcal X}|^t . \]
	
To prove this, fix $\mathcal X$ and $z$, and let $z=y_1$, $\mathcal X=\{y_2,\dots,y_k\}$, $\mathcal Y=\{y_1,\dots,y_k\}$ and $\mathcal X_i=\mathcal Y\setminus\{y_i\}$. Analogous to the definition of $W_j$, define 
\[
W_{j-1}':=\Bigl\{(w_2,\dots,w_k):
0\le w_i\le j-1 \text{ for } i=2,\dots,k,\text{ and all nonzero entries are distinct}\Bigr\}.
\]
For each $\bar w'=(w_2,\dots,w_k)\in W_{j-1}'$, define a function $f_{\bar{w}'}:\mathcal P_{\mathcal X_1}\to[0,1]$ as follows:
\[ 
f_{\bar w'}(v):=\nu\bigl(v,T_{\mathcal X_2}^{w_2},\dots,T_{\mathcal X_k}^{w_k}\bigr). 
\]
By the inductive hypothesis, we have 
\[\bar f_{\bar{w}'}(\mathcal P_{\mathcal X_1})\ge d - \frac{\varepsilon\cdot|\bar{w}'|_+}{2k},\]
where $\bar f_{\bar{w}'}(\mathcal P_{\mathcal X_1})$ is the average value of $f_{\bar{w}'}$ over $\mathcal P_{\mathcal X_1}$. 

Now choose $T_{\mathcal X_1}^j$ uniformly at random from $(\mathcal P_{\mathcal X_1})^t$. By \cref{hoe-ineq}, we have
\[
\Pr\!\left( \bar f_{\bar{w}'}(T_{\mathcal X_1}^j) \le \bar f_{\bar{w}'}(\mathcal P_{\mathcal X_1}) -\frac{\varepsilon}{2k} \right) \le e^{-2(\frac{\varepsilon}{2k})^2 t} \le \frac{1}{3kr^{k-1}}.
\]
Since the number of such functions is at most $|W_{j-1}'|\le j^{k-1}\le r^{k-1}$,
the union bound implies that with probability at least $1-\frac1{3k}$ none of the above bad events occurs. Hence, \[ |B_{\mathcal X}^z| \le \frac1{3k}|\mathcal P_{\mathcal X}|^t .\]

We now apply \cref{lem: reduced-ramsey lem} to the set $(\mathcal P_{\mathcal X})^t$ and $B_{\mathcal X}^z$ for $\mathcal X\in\binom{M_{j-1}}{k-1}$. This yields a subset $M_j\subseteq M_{j-1}$ with $|M_j|=m_j$ such that for every $\mathcal X\in\binom{M_j}{k-1}$, one can choose
$T_{\mathcal X}^j\in (\mathcal P_{\mathcal X})^t$
avoiding all corresponding forbidden sets $B_{\mathcal X}^z$ with
$z\in M_j\setminus \mathcal X$. By construction, this guarantees that \cref{eqlm2.6.2} holds for all $\bar w\in W_j$.
After completing $r$ such rounds,  we conclude that \cref{eq:goal contraction2} holds as desired.
\end{proof}

%%%%%%%%%%%%%%%%%%%%%%%%%%%%%
\section{Hypergraph regularity method}\label{sec-regular-method}

In this section, we present the regularity lemma for edge-colored hypergraphs, together with a corresponding embedding lemma. Our presentation follows the framework of R\"odl and Schacht~\cite{regularity-lemmas}, combined with results from~\cite{Embeddings,count-dense}.

The key notions underlying the hypergraph regularity method are regular complexes and equitable partitions. We begin by introducing these notions.

\subsection{Regular complex}

A \emph{mixed hypergraph} $\mathcal{H}$ consists of a vertex set $V(\mathcal{H})$ and an edge set $E(\mathcal{H})$, where each edge $e\in E(\mathcal{H})$ is a non-empty subset of $V(\mathcal{H})$. Thus, a $k$-graph, as defined earlier, is a mixed hypergraph in which every edge has size $k$. A mixed hypergraph $\mathcal{H}$ is called a \emph{complex} if every non-empty subset of every edge of $\mathcal{H}$ is also an edge of $\mathcal{H}$. Throughout this paper, we assume that every vertex of a complex lies in at least one edge. A complex is called a \emph{$k$-complex} if all its edges have size at most $k$. Given a $k$-complex $\mathcal{H}$ and $i\in [k]$, the edges of size $i$ are called the \emph{$i$-edges} of $\mathcal{H}$. We denote by $H^{(i)}$ the \emph{underlying $i$-graph} of $\mathcal{H}$, whose vertex set is $V(\mathcal{H})$ and whose edge set consists of all $i$-edges of $\mathcal{H}$. Note that every $k$-graph $H$ can be turned into a $k$-complex by replacing each edge with the complete $i$-graph $K_k^{(i)}$ on the same $k$ vertices for every $i\in [k]$.

Given $i\ge 2$, let $H^{(i)}$ be an $i$-graph and let $H^{(i-1)}$ be an $(i-1)$-graph on the same vertex set. We define the \emph{relative density} of $H^{(i)}$ with respect to $H^{(i-1)}$ by
\[
d(H^{(i)}\mid H^{(i-1)}):=
\begin{cases}
\dfrac{|E(H^{(i)})\cap \mathcal{K}_i(H^{(i-1)})|}{|\mathcal{K}_i(H^{(i-1)})|} & \text{if } |\mathcal{K}_i(H^{(i-1)})|>0,\\
0 & \text{otherwise}.
\end{cases}
\]
More generally, 
if $\mathbf Q=(Q(1),Q(2),\dots,Q(r))$ is an $r$-tuple of subhypergraphs of $H^{(i-1)}$, then we define $\mathcal{K}_i(\mathbf{Q}):= \bigcup^r_{j=1}\mathcal{K}_i(Q(j))$ and
\[
d(H^{(i)}\mid \mathbf{Q}):=
\begin{cases}
\dfrac{|E(H^{(i)})\cap \mathcal{K}_i(\mathbf{Q})|}{|\mathcal{K}_i(\mathbf{Q})|} & \text{if } |\mathcal{K}_i(\mathbf{Q})|>0,\\
0 & \text{otherwise}.
\end{cases}
\]

Given integers $s\ge k\ge 2$, an \emph{$(s,k)$-graph} $H^{(k)}$ is an $s$-partite $k$-graph, that is, the vertex set of $H^{(k)}$ can be partitioned into sets $V_1,\dots,V_s$ such that every edge of $H^{(k)}$ meets each $V_i$ in at most one vertex for every $i\in [s]$. Similarly, an \emph{$(s,k)$-complex} $\mathcal{H}^{\le k}$ is an $s$-partite $k$-complex.

Let $r\ge 1$ be an integer, let $d_i\ge 0$ and $\delta>0$ be real numbers, and let $H^{(i)}$ and $H^{(i-1)}$ be an $(i,i)$-graph and an $(i,i-1)$-graph, respectively, on the same vertex set.
We say $H^{(i)}$ is \emph{$(d_i, \delta, r)$-regular}  with respect to $H^{(i-1)}$ if every $r$-tuple $\mathbf{Q}$ of subhypergraphs of $H^{(i-1)}$ with $|\mathcal{K}_i(\mathbf{Q})| \geq \delta|\mathcal{K}_i(H^{(i-1)})|$ satisfies $d(H^{(i)}|\mathbf{Q}) = d_i \pm \delta$. 
Moreover, for two $s$-partite $i$-graph $H^{(i)}$ and $(i-1)$-graph $H^{(i-1)}$,  on the same vertex partition $V_1 \cup\dots \cup V_s$, we say that $H^{(i)}$ is \emph{$(d_i, \delta,r)$-regular} with respect to $H^{(i-1)}$ if for every $\Lambda_i\in \binom{[s]}{i}$ the restriction $H^{(i)}[\Lambda_i]:=H^{(i)}[\cup_{\lambda\in \Lambda_i }V_{\lambda}]$ is \emph{$(d_i, \delta,r)$-regular} with respect to the restriction $H^{(i-1)}[\Lambda_i]:=H^{(i-1)}[\cup_{\lambda\in \Lambda_i }V_{\lambda}]$.

\begin{definition}[Regular complex]
Let integers $s\ge k\ge 3$,  $\delta>0$ be a real number, and $\mathbf{d}=(d_2,\dots,d_{k-1})\in \mathbb{R}_{\ge 0}^{k-2}$. We say that an $(s,k-1)$-complex $\mathcal{H}^{\le k-1}=\{H^{(i)}\}_{i=1}^{k-1}$ is \emph{$(\mathbf{d},\delta,1)$-regular} if the $i$-graph $H^{(i)}$ is $(d_i,\delta,1)$-regular with respect to $H^{(i-1)}$ for every $i=2,\dots,k-1$.
\end{definition}

\subsection{Equitable partitions}\label{section-equ-partition}

Suppose that $V$ is a finite vertex set and that $\mathcal{P}^{(1)}=\{V_1,\dots,V_{a_1}\}$ is a partition of $V$, which will be called \emph{clusters}. A
partition is an \emph{equipartition} if the sizes of the clusters are same.

Given $k\ge 3$ and any $j\in [k]$, we denote by $\mathrm{Cross}_j=\mathrm{Cross}_j(\mathcal{P}^{(1)})$ the family of all crossing $j$-sets $J\in \binom{V}{j}$ such that $|J\cap V_i|\le 1$ for every cluster $V_i\in \mathcal{P}^{(1)}$.
For every index set $\Lambda\subseteq [a_1]$ with $2\le |\Lambda|\le k-1$, we write $\mathrm{Cross}_{\Lambda}$ for the family of all $|\Lambda|$-subsets of $V$ that meet each $V_i$ with $i\in \Lambda$. Let $\mathcal{P}_{\Lambda}$ be a partition of $\mathrm{Cross}_{\Lambda}$. We refer to the partition classes of $\mathcal{P}_{\Lambda}$ as \emph{$|\Lambda|$-cells}. For each $i=2,\dots,k-1$, let $\mathcal{P}^{(i)}$ be the union of all the $\mathcal{P}_{\Lambda}$ with $|\Lambda|=i$. Thus, $\mathcal{P}^{(i)}$ is a partition of $\mathrm{Cross}_i$ into several $(i,i)$-graphs.

Let $1\le i<j\le k$. For every $i$-set $I\in \mathrm{Cross}_i$, there exists a unique $i$-cell $P_I^{(i)}\in \mathcal{P}^{(i)}$ such that $I\in P_I^{(i)}$. For every $j$-set $J\in \mathrm{Cross}_j$, we define the \emph{polyad} of $J$ by
\[
\hat{P}_J^{(i)}:=\bigcup \bigl\{P_I^{(i)}: I\in \binom{J}{i}\bigr\}.
\]
Thus, $\hat{P}_J^{(i)}$ may be viewed as a $(j,i)$-graph whose vertex classes are the clusters intersecting $J$ and whose edge set is $\bigcup_{I\in \binom{J}{i}}E(P_I^{(i)})$. Let
\[
\hat{\mathcal P}^{(j-1)}:=\{\hat P_J^{(j-1)}: J\in \mathrm{Cross}_j\}.
\]
It is easy to verify  that $\{\mathcal{K}_j(\hat{P}^{(j-1)}) : \hat{P}^{(j-1)}\in \mathcal{\hat{P}}^{(j-1)}\}$ is also a partition of $\mathrm{Cross}_j$.

\begin{definition}[Family of partitions]\label{def-f-partition}
Suppose that $V$ is a vertex set,  $k\ge 2$ is an integer, and $\mathbf a=(a_1,\dots,a_{k-1})$ is a vector of positive integers. We say that
\[
\bm{\mathcal P}=\bm{\mathcal P}(k-1,\mathbf a)=\{\mathcal P^{(1)},\dots,\mathcal P^{(k-1)}\}
\]
is a \emph{family of partitions} on $V$ if the following hold:
\begin{enumerate}
\item[$\bullet$] $\mathcal P^{(1)}$ is a partition of $V$ into $a_1$ clusters.
\item[$\bullet$] For each $i=2,\dots,k-1$, $\mathcal P^{(i)}$ is a partition of $\mathrm{Cross}_i$ satisfying
\begin{equation}\label{cells}
\bigl|\{P^{(i)}\in \mathcal P^{(i)}: P^{(i)}\subseteq \mathcal K_i(\hat P^{(i-1)})\}\bigr|=a_i
\end{equation}
for every $\hat P^{(i-1)}\in \hat{\mathcal P}^{(i-1)}$.
\end{enumerate}
\end{definition}	

Note that, for each $J \in \mathrm{Cross}_j$ we can view $\bigcup^{j-1}_{i=1}\hat{P}^{(i)}_J$ as a $(j, j-1)$-complex.

Suppose that $Z\subseteq X$, that $\mathcal X$ is a partition
of $X$, and that $\mathcal Z$ is a partition of $Z$. We say
that $\mathcal X$ refines $\mathcal Z$, and write
$\mathcal X\prec \mathcal Z$, if for every $A\in\mathcal X$
there exists $B\in\mathcal Z$ such that either $A\subseteq B$
or $A\subseteq X\setminus Z$.

\begin{definition}[Refinement of partition families]\label{def-refinement}
Let $\bm{\mathcal{P}} = \bm{\mathcal{P}}(k-1, \textbf{a}^{\bm{\mathcal{P}}})$ and $\bm{\mathcal{Q}} = \bm{\mathcal{Q}}(k-1, \textbf{a}^{\bm{\mathcal{Q}}})$ be families of partitions on the same vertex set $V$. We say that $\bm{\mathcal P}$ \emph{refines} $\bm{\mathcal Q}$, and write $\bm{\mathcal P}\prec \bm{\mathcal Q}$, if $\mathcal{P}^{(j)} \prec \mathcal{Q}^{(j)}$  for every $j \in [k-1]$.
\end{definition}

\begin{definition}[$(\eta,\delta,t)$-equitable]\label{def-eq-partition}
Suppose that $V$ is a set of $n$ vertices, that $t\in \mathbb N$, that $\mathbf a=(a_1,\dots,a_{k-1})\in \mathbb N^{k-1}$, and that $\eta,\delta>0$. We say that a family of partitions  $\bm{\mathcal P} = \bm{\mathcal P}(k-1,\mathbf{a})$ is \emph{$(\eta,\delta,t)$-equitable} if it satisfies the following:
\stepcounter{propcounter}
\begin{enumerate}[label = ({\bfseries \Alph{propcounter}\arabic{enumi}})]
\item\label{p1} $\mathcal P^{(1)}$ is a partition of $V$ into $a_1$ clusters of equal size and $|\binom{V}{k}\setminus \mathrm{Cross}_k(\mathcal{P}^{(1)})|\leq \eta\binom{n}{k}$.
\item\label{p2} $\max\{a_1, \dots, a_{k-1}\}\le t$.
\item\label{p3} For every $k$-set $K \in \mathrm{Cross}_k$, the $(k, k-1)$-complex $\bigcup^{k-1}_{i=1}\hat{P}^{(i)}_K$ is $(\mathbf{d}, \delta,1)$-regular, where  $\mathbf{d}= (1/a_2, \dots , 1/a_{k-1})$. 
\item\label{p4}  For every $j\in [k-1] $ and every $k$-set $K \in \mathrm{Cross}_k$, we have
\[
|\mathcal{K}_k(\hat{P}^{(j)}_K)|=
(1\pm \eta) \frac{n^k}{\prod_{\ell=1}^ja_{\ell}^{\binom{k}{\ell}}}.
\]
\end{enumerate}
\end{definition}

\begin{remark}
Condition~\ref{p3} of \cref{def-eq-partition} implies that the $i$-cells of $\mathcal P^{(i)}$ have almost equal size. Moreover, condition~\ref{p4} is not part of the original definition of $(\eta,\delta,t)$-equitable partitions in R\"odl and Schacht~\cite{regularity-lemmas}. However, condition~\ref{p4} follows from conditions~\ref{p1} and~\ref{p3} together with the dense counting lemma from~\cite[Theorem 6.5]{count-dense}; see also~\cite[Theorem 3.1]{regularity-lemmas} and~\cite[Theorem 2.1]{counting-lemmas}.
\end{remark}

\subsection{Statements of the regularity lemma and embedding lemma }

Suppose that $\delta_k>0$ is a real number and $r\in \mathbb N$. Let $H$ be a $k$-graph on $V$, and $\bm{\mathcal P}=\bm{\mathcal P}(k-1,\mathbf a)$
be a family of partitions on $V$. Given a polyad $\hat P^{(k-1)}\in \hat{\mathcal P}^{(k-1)}$, we say that $H$ is \emph{$(\delta_k,r)$-regular} with respect to $\hat P^{(k-1)}$ if $H$ is $(d_k,\delta_k,r)$-regular with respect to $\hat P^{(k-1)}$, where $d_k=d(H\mid \hat P^{(k-1)})$.   
We now define what it means for $H$ to be \emph{$(\delta_k,r)$-regular} with respect to $\bm{\mathcal P}$.

\begin{definition}[\emph{$(\delta_k, r)$-regular} w.r.t. $\bm{\mathcal P}$]
We say a $k$-graph $H=(V,E)$ is \emph{$(\delta_k, r)$-regular} with respect to $\bm{\mathcal P}$ if
\[
\big|\bigcup\big\{\mathcal{K}_k(\hat{P}^{(k-1)}) : \hat{P}^{(k-1)}\in \mathcal{\hat{P}}^{(k-1)}
\text{ and } H \text{\ is\ not\ } (\delta_k, r)\text{-regular\ w.r.t.\ } \hat{P}^{(k-1)} \big\}\big| \le \delta_k |\mathrm{Cross}_k|.
\]
\end{definition}

This means that at most a $\delta_k$-fraction of the $k$-sets of $V$ span a $K_k^{(k-1)}$ contained in a polyad $\hat{P}^{(k-1)}$ with respect to which $H$ fails to be regular.

Finally, we introduce the notion of an $m$-colored $k$-graph.
Given a positive integer $m$, we say that a $k$-graph $H$ is $m$-colored if there exists a mapping $\chi: E(H) \to [m]$. For each $x \in [m]$, let $H_x$ denote the spanning subgraph of $H$ with edge set
$E(H_x)=\{e\in E(H): \chi(e)=x\}$.

We are now ready to state the regularity lemma for $m$-colored $k$-graphs, which is a direct consequence of~\cite[Lemma 4.1]{regularity-lemmas}.

\begin{theorem}[Regularity lemma]\label{multicolor hypergraph regularity lemma}
Let $\ell \ge k \ge 2$ and $m\ge 1$ be fixed integers. For all positive reals $\eta$ and $\delta_k$, and all functions 
$r:\mathbb{N}^{k-1}\to\mathbb{N}$ and $\delta:\mathbb{N}^{k-1}\to (0,1]$, there exist integers $t_0$ and $n_0$ such that the following holds.

For every $m$-colored $k$-graph $H$ with an initial equipartition $V(H)=U_1\cup\dots\cup U_\ell$, where $|V(H)|=n\ge n_0$, $(\ell \cdot t_0!)$ divides $n$, there exists a family of partitions $\bm{\mathcal{P}}=\bm{\mathcal{P}}(k-1,\mathbf{a})$ such that
\begin{enumerate}
\item $1/\eta<a_1<t_0$;
\item $\bm{\mathcal P}$ is $(\eta,\delta(\mathbf a),t_0)$-equitable and $\mathcal P^{(1)}$ refines the partition $U_1\cup\cdots\cup U_\ell$;
\item for each $x\in [m]$, the $k$-graph $H_x$ is $(\delta_k,r(\mathbf a))$-regular with respect to $\bm{\mathcal P}$.
\end{enumerate}
\end{theorem}

Finally, we state a general embedding lemma, which is a direct consequence of~\cite[Theorem 2]{Embeddings}.

\begin{theorem}[Embedding lemma]\label{E-lemma}
Let $f,k,r,n_1$ be positive integers, and let $d_2,d_3,\dots,d_k,\delta,\delta_k$ be positive constants such that $1/d_i\in \mathbb N$ for all $i<k$, and
\[
n_1^{-1}\ll r^{-1},\delta \ll \min\{\delta_k,d_2,\dots,d_{k-1}\}\le \delta_k \ll d_k,1/f<1.
\]
Then the following holds for all integers $n\ge n_1$. Let $F$ be a $k$-graph with vertex set $[f]$, and let $\mathbf d=(d_2,\dots,d_{k-1})$. Suppose that
$\mathcal H=\{H^{(j)}\}_{j=1}^{k-1}$
is a $(\mathbf d,\delta,1)$-regular $(f,k-1)$-complex with clusters $V_1,\dots,V_f$, all of size $n$. Suppose also that $H$ is an $f$-partite $k$-graph on the same vertex partition such that for each edge $\{i_1,\dots,i_k\}\in E(F)$, the graph $H$ is $(\delta_k,r)$-regular with respect to the restriction
$H^{(k-1)}[V_{i_1}\cup \dots \cup V_{i_k}]$
and
$d\bigl(H\mid H^{(k-1)}[V_{i_1}\cup \dots \cup V_{i_k}]\bigr)\ge d_k$.
Then $H$ contains a copy of $F$.
\end{theorem}

%%%%%%%%%%%%%%%%%%%%%%
\section{Single palette classification}\label{sec: proof of single palette classification}

In this section, we prove~\cref{thm: classification for single}. We begin with several preparatory lemmas. The first is the classical Erd\H{o}s--Szekeres theorem.

\begin{lemma}[\cite{ERDSZE1935}]\label{lem:Erdos-Szekeres}
Let $n_1$ and $n_2$ be positive integers. Any sequence of $(n_1-1)(n_2-1)+1$ distinct real numbers contains either an increasing subsequence of length $n_1$ or a decreasing subsequence of length $n_2$.
\end{lemma}

The following two lemmas form the core of the proof of~\cref{thm: classification for single}.

\begin{lemma}\label{lem:single order}
Given $k\ge 3$, let $\mathscr{P}=(\mathcal{C},\mathcal{T})$ and $\mathscr{P}_0=(\mathcal{C}_0,\mathcal{T}_0)$ be two $k$-palettes. If every ordered $k$-graph that is $\mathscr{P}$-colorable is also $\mathscr{P}_0$-colorable, then there exists a homomorphism from $\mathscr{P}$ to $\mathscr{P}_0$.
\end{lemma}

\begin{proof}
Let $\mathscr{P}=(\mathcal{C}, \mathcal{T})$ and $\mathscr{P}_0=(\mathcal{C}_0, \mathcal{T}_0)$ be two $k$-palettes as in the statement. Set
\[
m:=|\mathcal C|\cdot |\mathcal C_0|,
\qquad
\ell:=|\mathcal C_0|^{|\mathcal C|},
\qquad
R:=R_{k-1}(k,\ell),
\qquad
d_{k-1}:=\frac{1}{2m}.
\] 
Choose auxiliary constants satisfying
\[
{1}/{n_1}\ll {1}/{r}, \delta \ll d_2,\dots,d_{k-2}\ll \delta_{k-1} \ll d_{k-1}, {1}/{k},
\]
where $1/d_i\in \mathbb N$ for all $i=2,\dots,k-2$, so that the assumptions of~\cref{E-lemma} are satisfied with $k$ replaced by $k-1$ and $F=K_k^{(k-1)}$. In addition, choose $\delta_{k-1}>0$ sufficiently small such that
\begin{equation}\label{eq:delta{k-1}-single}
2\sqrt{\delta_{k-1}}\,m\, R^{k-1}<1.
\end{equation}

Apply \cref{multicolor hypergraph regularity lemma} with parameters
$R, k-1, m,  \eta=1/R,  \delta_{k-1}, r, \delta$
to obtain integers $t_0$ and $n_0$. Next choose $\eps>0$ sufficiently small such that 
\begin{equation}\label{eq:eps-single}
\eps<\min\left\{\frac{1}{2},\ (1-\eta)t_0^{-(k-1)}\prod_{j=2}^{k-2} d_j^{\binom{k-1}{j}}\right\},
\end{equation}
and let $N\ge \max\{n_0,n_1\}$ be  large enough such that  $(R\cdot t_0!)\mid N$ and
\begin{equation}\label{eq:positive probability single-order}
1-2^{\binom{N}{k-2}}|\mathcal C|e^{-\frac{\varepsilon^3N^{k-1}}{2|\mathcal C|}}>0.
\end{equation}

Let $G$ be the complete $(k-1)$-graph on vertex set $[N]$. Color the edges of $G$ independently and uniformly at random with colors from $\mathcal C$. For any $(k-2)$-graph $G^{(k-2)}$ on $[N]$ with $|\mathcal K_{k-1}(G^{(k-2)})|\ge \varepsilon N^{k-1}$
and any color $x\in \mathcal C$, let $B_x(G^{(k-2)})$ denote the event that fewer than
\[
(1-\varepsilon)\frac{|\mathcal K_{k-1}(G^{(k-2)})|}{|\mathcal C|}
\]
members of $\mathcal K_{k-1}(G^{(k-2)})$ receive color $x$. By the Chernoff bound,
\[
\mathbb P(B_x(G^{(k-2)}))
\le
e^{-\frac{\varepsilon^2|\mathcal K_{k-1}(G^{(k-2)})|}{2|\mathcal C|}}
\le
e^{-\frac{\varepsilon^3N^{k-1}}{2|\mathcal C|}}.
\]
Since there are at most $2^{\binom{N}{k-2}}$ choices for $G^{(k-2)}$, the union bound together with \eqref{eq:positive probability single-order} implies that there exists a coloring $f:E(G)\to \mathcal C$
such that for every color $x\in \mathcal C$ and every $(k-2)$-graph $G^{(k-2)}$ on $[N]$ with $|\mathcal K_{k-1}(G^{(k-2)})|\ge \varepsilon N^{k-1}$, 
we have
\begin{equation}\label{eq:random-color-property}
|\mathcal K_{k-1}(G^{(k-2)})\cap E(G_x)|
\ge
(1-\varepsilon)\frac{|\mathcal K_{k-1}(G^{(k-2)})|}{|\mathcal C|},
\end{equation}
where $G_x$ denotes the spanning subgraph of $G$ consisting of all edges colored $x$.

Now define a $k$-graph $H$ on $[N]$ by declaring that for every
\[
e=\llbracket i_1,i_2,\dots,i_k\rrbracket\in \binom{[N]}{k},
\]
we have $e\in E(H)$ if and only if
\[
\bigl(f(e\setminus\{i_1\}),f(e\setminus\{i_2\}),\dots,f(e\setminus\{i_k\})\bigr)\in \mathcal T.
\]
By construction, the ordered $k$-graph $H$ with the natural order  is $\mathscr P$-colorable. By assumption, $H$ is also $\mathscr{P}_0$-colorable and let $f_0: E(G)\to \mathcal C_0$ be a coloring witnessing this property. 

For each $(x,x_0)\in \mathcal C\times \mathcal C_0$, let $G_{x,x_0}$ denote the spanning subgraph of $G$ with edge set
\[
E(G_{x,x_0})=\{e\in E(G): f(e)=x,\ f_0(e)=x_0\}.
\]
Let $[N]=U_1\cup \cdots \cup U_{R}$
be an ordered equipartition such that $\max\{u: u\in U_i\}<
\min\{u: u\in U_j\}$ for $1\le i<j\le R$.
Apply~\cref{multicolor hypergraph regularity lemma} to the $(\mathcal C\times \mathcal C_0)$-colored $(k-1)$-graph $G$ with initial partition $U_1\cup\cdots\cup U_{R}$. We obtain a family of partitions
\[
\bm{\mathcal P}=\bm{\mathcal P}(k-2,\mathbf a),
\]
where $a_i=1/d_i$ for $i=2,\dots,k-2$ and $1/\eta<a_1<t_0$, such that
\begin{enumerate}
\item $\bm{\mathcal P}$ is $(\eta,\delta(\mathbf a),t_0)$-equitable and $\mathcal P^{(1)}$ refines $U_1\cup\cdots\cup U_{R}$;
\item for each $(x,x_0)\in \mathcal C\times\mathcal C_0$, the $(k-1)$-graph $G_{x,x_0}$ is $(\delta_{k-1},r(\mathbf a))$-regular with respect to $\bm{\mathcal P}$.
\end{enumerate}

Recall that the elements of $\mathcal P^{(1)}$ are called clusters. We say that a $(k-1)$-tuple of clusters is \emph{good} for $(x,x_0)\in \mathcal C\times \mathcal C_0$, if $G_{x,x_0}$ is $(\delta_{k-1},r(\mathbf a))$-regular with respect to all but at most a $\sqrt{\delta_{k-1}}$-fraction of the polyads induced on this $(k-1)$-tuple of clusters. A $(k-1)$-tuple of clusters is called \emph{all-good} if it is good for every $(x,x_0)\in \mathcal C\times \mathcal C_0$.

\begin{claim}\label{claim-single}
There exist clusters $V_1,\dots,V_{R}\in \mathcal P^{(1)}$ such that $V_i\subseteq U_i$ for every $i\in[R]$ and all $(k-1)$-tuples among $V_1,\dots,V_{R}$ are all-good.
\end{claim}

\begin{proof}
Fix $(x,x_0)\in \mathcal C\times \mathcal C_0$. By condition~\ref{p4} of \cref{def-eq-partition}, for every $J\in \mathrm{Cross}_{k-1}$, we have
\[
|\mathcal K_{k-1}(\hat P_J^{(k-2)})|
=
(1\pm \eta)\prod_{j=1}^{k-2}\left(\frac{1}{a_j}\right)^{\binom{k-1}{j}}N^{k-1}.
\]
Since $G_{x,x_0}$ is $(\delta_{k-1},r(\mathbf a))$-regular with respect to $\bm{\mathcal P}$, the number of polyads $\hat P^{(k-2)}$ with respect to which $G_{x,x_0}$ is not $(\delta_{k-1},r(\mathbf a))$-regular is at most
\[
\frac{\delta_{k-1}N^{k-1}}
{(1-\eta)\prod_{j=1}^{k-2}\left(\frac{1}{a_j}\right)^{\binom{k-1}{j}}N^{k-1}}
\le
2\delta_{k-1}\prod_{j=1}^{k-2}a_j^{\binom{k-1}{j}}.
\]
Each $(k-1)$-tuple of clusters induces exactly
$\prod_{j=2}^{k-2}a_j^{\binom{k-1}{j}}$
polyads $\hat P^{(k-2)}$. Hence all but at most
$2\sqrt{\delta_{k-1}}\,a_1^{k-1}$ 
$(k-1)$-tuples of clusters are good for $(x,x_0)$. Summing over all $(x,x_0)\in \mathcal C\times \mathcal C_0$, we conclude that all but at most
\[
2\sqrt{\delta_{k-1}}\,m\,a_1^{k-1}
\]
$(k-1)$-tuples of clusters are all-good.

Since $\mathcal P^{(1)}$ refines the partition $U_1\cup\cdots\cup U_{R}$,  the number of ways to choose one cluster $V_i\subseteq U_i$ for each $i\in [R]$
is $(a_1/R)^{R}$. If every choice of $V_i$ yielded one $(k-1)$-tuple that is not all-good, then by double counting the number of not all-good $(k-1)$-tuples would be at least
\[
\frac{(a_1/R)^{R}}{(a_1/R)^{R-k+1}}
>
2\sqrt{\delta_{k-1}}\,m\,a_1^{k-1},
\]
which follows from $2\sqrt{\delta_{k-1}}\,m\,R^{k-1}<1$, a contradiction. 
\end{proof}

Now fix a choice of $V_1,V_2,\dots,V_{R}$ satisfying \cref{claim-single}. We next choose one “good” polyad for all $(k-1)$-tuples of these clusters in a consistent way.
\begin{claim}\label{claim:choice of polyads}
Assume that the clusters $V_1,\dots,V_R$ have been chosen as in \cref{claim-single}.
Then one can choose an $(R, k-2)$-complex $\{Q^{(i)}\}_{i=1}^{k-2}$ with clusters $V_1,\dots,V_R$ such that the following hold.
\stepcounter{propcounter}
\begin{enumerate}[label = {{\rm (\Alph{propcounter}\arabic{enumi})}}]
\item \label{polyad-1} $Q^{(1)}[V_{\lambda}]=V_{\lambda}$ for each $\lambda\in [R]$.
\item \label{polyad-2} For every $2\le i\le k-2$ and every $\Lambda_i\in \binom{[R]}{i}$, the restriction $Q^{(i)}_{\Lambda_i}:=Q^{(i)}[\cup_{\lambda\in \Lambda_i }V_{\lambda}]$ is a $i$-cell.
Moreover, these cells are chosen consistently: for every
$2\le i\le k-3$ and every $\Lambda_{i+1}\in\binom{[R]}{i+1}$, $Q_{\Lambda_{i+1}}^{(i+1)}
\subseteq
\mathcal K_{i+1}\bigl(\hat Q_{\Lambda_{i+1}}^{(i)}\bigr)$.

\item \label{polyad-3} For every $I\in \binom{[R]}{k-1}$, the union 
\[\hat Q_{I}^{(k-2)}
:=
\bigcup\Bigl\{Q_{\Lambda_{k-2}}^{(k-2)}: \Lambda_{k-2}\in \binom{I}{k-2}\Bigr\}
\]
is a polyad on the $(k-1)$-tuple of clusters indexed by $I$.

\item \label{polyad-4} For every $I\in \binom{[R]}{k-1}$
and each pair $(x,x_0)\in \mathcal C\times \mathcal C_0$,
the $(k-1)$-graph $G_{x,x_0}$ is $(\delta_{k-1},r(\mathbf a))$-regular with respect to $\hat Q_I^{(k-2)}$.
\end{enumerate}
\end{claim}

\begin{proof}
First define $Q^{(1)}$ to be the $1$-graph with vertex set
$V_1\cup\cdots\cup V_R$.
Thus, condition~\ref{polyad-1} holds.
For $k=3$, there are no higher-level cells to choose. 
For every $I=\{\lambda_1,\lambda_2\}\in\binom{[R]}{2}$, $\mathcal K_2(\hat Q_I^{(1)})$ is the complete bipartite graph between $V_{\lambda_1}$ and $V_{\lambda_2}$. Since every pair among $V_1,\dots,V_R$ is all-good by \cref{claim-single}, it follows that for each $I\in\binom{[R]}{2}$ and $(x,x_0)\in\mathcal C\times\mathcal C_0$,
the graph $G_{x,x_0}$ is $(\delta_{k-1},r(\mathbf a))$-regular with respect to $\hat Q_I^{(1)}$. 

Assume now that $k\ge 4$. We construct the desired complex level by level.
For every $\Lambda_2\in\binom{[R]}{2}$,
choose uniformly and independently one $2$-cell of $\mathcal P^{(2)}$ between the two clusters indexed by $\Lambda_2$. Denote this chosen $2$-cell by $Q_{\Lambda_2}^{(2)}$
and let $Q^{(2)}$ be the union of all these chosen $2$-cells.
Suppose inductively that $Q^{(i)}$ has already been chosen for some $2\le i\le k-3$.
For each $\Lambda_{i+1}\in\binom{[R]}{i+1}$,
the union
\[
\hat Q_{\Lambda_{i+1}}^{(i)}
:=
\bigcup\Bigl\{
Q_{\Lambda_i}^{(i)}:\Lambda_i\in\binom{\Lambda_{i+1}}{i}
\Bigr\}
\]
is an $i$-polyad on the $(i+1)$-tuple of clusters indexed by $\Lambda_{i+1}$. Choose uniformly and independently one $(i+1)$-cell of $\mathcal P^{(i+1)}$ contained in $\mathcal K_{i+1}\bigl(\hat Q_{\Lambda_{i+1}}^{(i)}\bigr)$ and denote it by $Q_{\Lambda_{i+1}}^{(i+1)}$.
Let $Q^{(i+1)}$ be the union of all these chosen $(i+1)$-cells.
Proceeding inductively up to level $k-2$, we obtain an $(R,k-2)$-complex
$\{Q^{(i)}\}_{i=1}^{k-2}$
with clusters $V_1,\dots,V_R$. By construction, conditions~\ref{polyad-1}-\ref{polyad-3} hold. 

It remains to prove condition~\ref{polyad-4}. For each $I\in\binom{[R]}{k-1}$, the above random procedure makes $\hat Q_I^{(k-2)}$ a uniformly random polyad among all polyads induced by $\bm{\mathcal P}$ on the $(k-1)$-tuple of clusters indexed by $I$. Since the clusters $V_1,\dots,V_R$ were chosen as in \cref{claim-single}, every $(k-1)$-tuple among them is all-good. Therefore, for each fixed 
\[
I\in\binom{[R]}{k-1}
\qquad\text{and}\qquad
(x,x_0)\in\mathcal C\times\mathcal C_0,
\]
the probability that
$G_{x,x_0}$
is not $(\delta_{k-1},r(\mathbf a))$-regular with respect to $\hat Q_I^{(k-2)}$ is at most
$\sqrt{\delta_{k-1}}$.
By the union bound, the probability that there exists some $(x,x_0)\in\mathcal C\times\mathcal C_0$ and some $I\in\binom{[R]}{k-1}$ such that $G_{x,x_0}$  is not $(\delta_{k-1},r(\mathbf a))$-regular with respect to $\hat Q_I^{(k-2)}$ is at most 
\[
\binom{R}{k-1}m\sqrt{\delta_{k-1}}
<1,
\]
where the last inequality follows 
from~\eqref{eq:delta{k-1}-single}.
Therefore we may fix a choice of the complex $\{Q^{(i)}\}_{i=2}^{k-2}$ on $V_1\cup V_2\cup\dots\cup V_{R}$ as desired.
\end{proof}

Next, fix an $(R, k-2)$-complex $\{Q^{(i)}\}_{i=1}^{k-2}$ with clusters $V_1,\dots,V_R$ satisfying \cref{claim:choice of polyads}.
For each $I\in \binom{[R]}{k-1}$, define a function
\[
h_I:\mathcal C\to \mathcal C_0
\]
as follows: for each $x\in\mathcal C$, let $h_I(x)$ be a color in $\mathcal C_0$ maximizing
\[
\bigl|\mathcal K_{k-1}(\hat Q_I^{(k-2)})\cap E(G_{x,x_0})\bigr|
\]
over all $x_0\in\mathcal C_0$.
Since there are at most $\ell$ possible such functions and $R=R_{k-1}(k,\ell)$, there exists a $k$-set of indices whose every $(k-1)$-subset receives the same color function, denoted by
\[
h:\mathcal C\to \mathcal C_0.
\]
Without loss of generality, we may assume that this $k$-set of indices is $[k]$. We claim that $h$ is a homomorphism from $\mathscr P$ to $\mathscr P_0$.

Let $(x_1,\dots,x_k)\in\mathcal T$. For each $i\in[k]$, set
\[
I_i:=[k]\setminus\{i\}.
\]
By condition~\ref{p4} of \cref{def-eq-partition}, we have
\[
|\mathcal K_{k-1}(\hat Q_{I_i}^{(k-2)})|
\ge
(1-\eta)\prod_{j=1}^{k-2}\left(\frac1{a_j}\right)^{\binom{k-1}{j}}N^{k-1}\ge
\varepsilon N^{k-1},
\]
where the last inequality follows from  \eqref{eq:eps-single}.
Therefore, by \eqref{eq:random-color-property}, for every $i\in[k]$ and every $x\in\mathcal C$,
\[
|\mathcal K_{k-1}(\hat Q_{I_i}^{(k-2)})\cap E(G_x)|
\ge
(1-\eps)\frac{|\mathcal K_{k-1}(\hat Q_{I_i}^{(k-2)})|}{|\mathcal C|}.
\] 
Since $h_{I_i}(x)=h(x)$ is chosen as the majority $f_0$-color among those members of $\mathcal K_{k-1}(\hat Q_{I_i}^{(k-2)})$ whose $f$-color is $x$, we obtain
\[
d\bigl(G_{x,h(x)}\mid \hat Q_{I_i}^{(k-2)}\bigr)
\ge
\frac{1-\varepsilon}{|\mathcal C||\mathcal C_0|}
=
\frac{1-\varepsilon}{m}
\ge
d_{k-1}.
\]
Moreover, by~\cref{claim:choice of polyads}, for every $i\in[k]$, the $(k-1)$-graph $G_{x_i,h(x_i)}$ is $(\delta_{k-1},r(\mathbf a))$-regular with respect to $\hat Q_{I_i}^{(k-2)}$.

Now consider the $(k-1)$-graph
\[
G':=\bigcup_{i=1}^k G_{x_i,h(x_i)}\Bigl[\bigcup_{\lambda\in I_i} V_\lambda\Bigr].
\]
By construction, the complex
\[
\{Q^{(i)}[V_1\cup\cdots\cup V_k]\}_{i=1}^{k-2}
\]
together with the $(k-1)$-graph $G'$ satisfies the assumptions of \cref{E-lemma} for the complete $(k-1)$-graph $K_k^{(k-1)}$. Hence $G'$ contains a copy of $K_k^{(k-1)}$ on the clusters $V_1,\dots,V_k$.

Let this copy correspond to vertices $v_1\in V_1,\dots,v_k\in V_k$. Since $V_i\subseteq U_i$ for each $i\in[k]$ and the partition $U_1\cup\cdots\cup U_{R}$ is ordered, we have
\[
v_1< v_2<\cdots < v_k.
\]
For each $i\in[k]$, the $(k-1)$-set $\{v_1,\dots,v_k\}\setminus\{v_i\}$
belongs to $G_{x_i,h(x_i)}$. Because $(x_1,\dots,x_k)\in\mathcal T$, the $k$-set $\{v_1,\dots,v_k\}$ is an edge of $H$. Since $f_0$ witnesses that $H$ is $\mathscr P_0$-colorable, it follows that
\[
(h(x_1),h(x_2),\dots,h(x_k))\in \mathcal T_0.
\]
Therefore $h$ is a homomorphism from $\mathscr P$ to $\mathscr P_0$, as required.
\end{proof}

Recall that a hypergraph is \emph{linear} if any two edges have at most one vertex in common.

\begin{lemma}\label{lm4.8}
Given $k\ge 3$, let $\mathscr P$ and $\mathscr P_0$ be two $k$-palettes. If there exist ordered $k$-graphs $H^+$ and $H^-$ such that
\begin{itemize}
\item $H^+$ and $H^-$ are $\mathscr P$-colorable,
\item $H^+$ is not $\mathscr P_0$-colorable, and
\item $H^-$ is not ${\rm rev}(\mathscr{P}_0)$-colorable.
\end{itemize}
Then, there exists a $k$-graph $H$ that is $\mathscr P$-colorable but not $\mathscr P_0$-colorable.
\end{lemma}

\begin{proof}
Let $\mathscr P=(\mathcal C,\mathcal T)$ and $\mathscr P_0=(\mathcal C_0,\mathcal T_0)$ be two $k$-palettes as in the statement, and let $H^+$ and $H^-$ be the corresponding ordered $k$-graphs. 
Set
\[
n_1:=|V(H^+)|,
\qquad
n_2:=|V(H^-)|,
\qquad
n:=\max\{n_1,n_2\},
\qquad
m:=(n_1-1)(n_2-1)+1,
\]
and
\[
p:=|\mathcal C|^{-n^{k-1}}.
\]
Choose $N$ sufficiently large such that
\begin{equation}\label{eq:random-single-palette}
N!\cdot \left(1-p\right)^{N^2/(m^2(m-1)^2)}<1.
\end{equation}
Such an $N$ exists because $p>0$ is fixed, $N!\le N^N=e^{N\log N}$
and 
$(1-p)^{N^2/(m^2(m-1)^2)}\le e^{-cN^2}$
for some constant $c=c(p,m)>0$.

Now define a random $k$-graph $H$ on $[N]$ as follows. Choose a coloring
\[
f:\binom{[N]}{k-1}\to\mathcal C
\]
uniformly at random. For every $e=\llbracket i_1,i_2,\dots,i_k\rrbracket\in \binom{[N]}{k}$,
we choose $e\in E(H)$ if and only if
\[
\bigl(f(e\setminus\{i_1\}),f(e\setminus\{i_2\}),\dots,f(e\setminus\{i_k\})\bigr)\in\mathcal T.
\]
Then $H$ is $\mathscr P$-colorable by construction. Next, we shall prove that with positive probability $H$ is not $\mathscr P_0$-colorable

Let $L$ be a linear $m$-graph on $[N]$ with the maximum possible number of edges. Since an $m$-edge can intersect fewer than
\[
\binom{m}{2}\binom{N-2}{m-2}
\]
other $m$-edges in the complete $m$-graph on $[N]$, we have
\[
|E(L)|\ge
\frac{\binom{N}{m}}{\binom{m}{2}\binom{N-2}{m-2}}
=
\frac{2N(N-1)}{m^2(m-1)^2}
\ge
\frac{N^2}{m^2(m-1)^2}.
\]
Enumerate the edges of $L$ as $e_1,e_2,\dots,e_{|E(L)|}$.

Fix any linear order $\prec$ on $V(H)$. By~\cref{lem:Erdos-Szekeres}, for each $j\in[|E(L)|]$, $e_j$ contains either
\begin{itemize}
\item an $n_1$-subset $S_j$ whose natural order and $\prec$-order coincide, or
\item an $n_2$-subset $T_j$ whose natural order and $\prec$-order are opposite.
\end{itemize}

Suppose first that $e_j$ contains such an $n_1$-subset $S_j$. Since $H^+$ is $\mathscr P$-colorable, there exist a natural order
$u_1< \cdots< u_{n_1}$
of $V(H^+)$ and a witness coloring
\[
\varphi_1:\binom{V(H^+)}{k-1}\to\mathcal C.
\]
Identify $u_1,\dots,u_{n_1}$ with the vertices of $S_j$ in their common natural order and $\prec$-order. If, for every $X\in\binom{S_j}{k-1}$,
the random color $f(X)$ equals the corresponding value prescribed by $\varphi_1$, then every edge of $H^+$ appears in the ordered $k$-graph $H^\prec[S_j]$. The probability of this event is at least
\[
|\mathcal C|^{-\binom{n_1}{k-1}}
\ge
|\mathcal C|^{-n^{k-1}}
=
p.
\]
Hence, with probability at least $p$, the ordered $k$-graph $H^\prec[S_j]$ contains an ordered copy of $H^+$. Since $H^+$ is not $\mathscr P_0$-colorable, this implies that $H^\prec[e_j]$ is not $\mathscr P_0$-colorable.
Now suppose that $e_j$ contains an $n_2$-subset $T_j$ whose natural order and $\prec$-order are opposite. Let
$\rm{rev}(H^{-})$
denote the ordered $k$-graph obtained from $H^{-}$ by reversing its linear order. Similar to the first case, since the natural order on $T_j$ is the reverse of the  $\prec$-order, $H^\prec[T_j]$ contains an ordered copy of $\rm{rev}(H^{-})$ with probability at least $p$.
Since $H^{-}$ is not ${\rm rev}(\mathscr P_0)$-colorable, the ordered $k$-graph $\rm{rev}(H^{-})$ is not $\mathscr P_0$-colorable, which implies that 
$H^\prec[e_j]$ is not $\mathscr P_0$-colorable with probability at least $p$.
Therefore, for every $j\in[|E(L)|]$, we have
\[
\mathbb P\bigl(H^\prec[e_j]\text{ is }\mathscr P_0\text{-colorable}\bigr)\le 1-p.
\]

Since $L$ is linear and $k\ge 3$, any two distinct edges of $L$ share at most one vertex, and hence no $(k-1)$-subset of vertices is contained in two different edges of $L$. Therefore, the events
\[
\bigl\{H^\prec[e_j]\text{ is }\mathscr P_0\text{-colorable}\bigr\}_{j=1}^{|E(L)|}
\]
are independent.

If $k$-graph $H$ is $\mathscr P_0$-colorable, then there exists some linear order $\prec$ on $[N]$ such that the ordered $k$-graph $H^\prec$ is $\mathscr P_0$-colorable. In particular, for this order $\prec$, every restriction $H^\prec[e_j]$ would be $\mathscr P_0$-colorable. Hence, by independence,
\[
\mathbb P\bigl(H^\prec\text{ is }\mathscr P_0\text{-colorable}\bigr)
\le
(1-p)^{|E(L)|}.
\]
Taking the union bound over all $N!$ linear orders, we obtain
\[
\mathbb P\bigl(H\text{ is }\mathscr P_0\text{-colorable}\bigr)
\le
N!\cdot (1-p)^{|E(L)|}
\le
N!\cdot (1-p)^{N^2/(m^2(m-1)^2)}
<1
\]
by \eqref{eq:random-single-palette}. Therefore, with positive probability, the random $k$-graph $H$ is $\mathscr P$-colorable but not $\mathscr P_0$-colorable.
\end{proof}

Finally, we use \cref{lem:single order} and \cref{lm4.8} to prove \cref{thm: classification for single}.

\begin{proof}[Proof of \cref{thm: classification for single}]
Suppose first that there exists a $k$-graph $H$ that is $\mathscr P$-colorable but not $\mathscr P_0$-colorable. By the basic property of palette homomorphisms, there is no homomorphism from $\mathscr P$ to $\mathscr P_0$. Since every $k$-graph that is ${\rm rev}(\mathscr P_0)$-colorable becomes $\mathscr P_0$-colorable after reversing the vertex order, there is also no homomorphism from $\mathscr P$ to ${\rm rev}(\mathscr P_0)$.

Conversely, suppose that there is no homomorphism from $\mathscr P$ to $\mathscr P_0$ and no homomorphism from $\mathscr P$ to ${\rm rev}(\mathscr P_0)$. By \cref{lem:single order}, there exists an ordered $k$-graph $H_1$ that is $\mathscr P$-colorable but not $\mathscr P_0$-colorable, and an ordered $k$-graph $H_2$ that is $\mathscr P$-colorable but not ${\rm rev}(\mathscr P_0)$-colorable. Applying \cref{lm4.8}, we obtain a $k$-graph $H$ that is $\mathscr P$-colorable but not $\mathscr P_0$-colorable.
\end{proof}

%%%%%%%%%%%%%%%%%%%%%%%%
\section{More general palette classification}\label{sec: proof of multi-palette classification}

In this section, we extend the single-palette classification from~\cref{thm: classification for single} to finite families of palettes; see~\cref{thm: classification for multicase}. 
This extension is useful for the applications developed later. Indeed, when one aims to show that a prescribed value $\alpha\in(0,1)$ belongs to $\Pi^{(k)}_{\rm u}$, it is often more convenient to construct a $k$-graph that is simultaneously colorable by finitely many $k$-palettes $\mathscr P_1,\dots,\mathscr P_t$ than to work with a single specially chosen $k$-palette. In such a situation, once one finds a palette $\mathscr P_0$ of density $\alpha$ that does not color the $k$-graph, the upper-bound argument reduces to showing that every $k$-palette $\mathscr P$ with $d(\mathscr P)>\alpha$ contains at least one of $\mathscr P_1,\dots,\mathscr P_t$ as a subpalette, that is, there exists a homomorphism from some $\mathscr P_i$ to $\mathscr P$. Thus, the multi-palette setting provides a more flexible framework for certifying realizable values of $(k-2)$-uniform Tur\'an density.

The passage from one palette to finitely many palettes is, however, not merely formal. 
The first natural operation in this setting is the product of $k$-palettes, which encodes simultaneous colorability when the underlying vertex order is fixed.

\begin{definition}[Product of palettes]
Given $k \ge 3$ and $t \ge 2$, let $\mathscr P_i = (\mathcal C_i, \mathcal T_i)$ be a $k$-palette for each $i \in [t]$. Their \emph{product} is the $k$-palette
\[
\prod_{i=1}^t \mathscr P_i = \left( \prod_{i=1}^t \mathcal C_i,  \mathcal T^{\times} \right),
\]
where $\mathcal T^{\times}$ consists of all $k$-tuples
\[
\left( (x_{1,1}, x_{2,1}, \dots, x_{t,1}), (x_{1,2}, x_{2,2}, \dots, x_{t,2}), \dots, (x_{1,k}, x_{2,k}, \dots, x_{t,k}) \right)
\]
such that $(x_{i,1}, x_{i,2}, \dots, x_{i,k}) \in \mathcal T_i$ for every $i \in [t]$.
% Furthermore, given $\mathbf x_{1,[k]} \circ \mathbf x_{2,[k]} \circ \dots \circ \mathbf x_{t,[k]}$, we define the \emph{$i$-th column} $x_{[t],i}$ for each $i \in [k]$ by
% \[
% x_{[t],i} = (x_{1,i}, x_{2,i}, \dots, x_{t,i}).
% \]
\end{definition}

Note that  
an ordered $k$-graph $H^{\prec}$ is $\mathscr P_i$-colorable for every $i\in[t]$ with respect to the same linear order $\prec$ if and only if $H^{\prec}$ is $\prod_{i=1}^t \mathscr P_i$-colorable.
For unordered $k$-graphs, however, the situation is subtler. Although a $k$-graph may be colorable by each of several palettes, the vertex orders witnessing these colorings need not coincide. To overcome this difficulty, we use a symmetrization operation on palettes, based on the symmetric groups $\mathbb S_k$, which records how a palette-coloring transforms under a change of the underlying vertex order.

Let $r\ge 2$ be an integer. Let $\mathbb S_r$ be the symmetric group on $[r]$, throughout this section.
Given pairwise distinct integers $x_1,\dots,x_r$,  let $x_1' < \dots < x_r'$ be their increasing rearrangement under the natural order.
We define $\alpha^*[(x_1,\dots,x_r)]$ to be the permutation in $\mathbb S_r$ such that
\[
\alpha^*[(x_1,\dots,x_r) ]= \begin{pmatrix}
x_1' & \dots & x_r' \\
x_1  & \dots & x_r
\end{pmatrix}.
\]
Given a permutation $(x_1,\dots,x_r)$ of $x_1,\dots,x_r$, for each $i\in [r]$, let $(x_1,\dots,x_r)\setminus \{x_i\}$ be the $(r-1)$-tuple obtained by deleting the entry $x_i$ from $(x_1,\dots,x_r)$. 
Given $\tau\in\mathbb S_r$ with
\[
\tau = \begin{pmatrix}
1 & \dots & r \\
\tau(1) & \dots & \tau(r)
\end{pmatrix},
\]
let $\partial_i\tau =\alpha^*[(\tau(1),\dots, \tau(r))\setminus \{\tau(i)\}]\in\mathbb S_{r-1}$.

%Specifically, let $(1,2,\dots,k)\setminus \{j\}$ be the $(k-1)$-tuple obtained by deleting the element $j$ from $(1,2,\dots,k)$.Then $\partial_j\tau$ is the permutation that maps the $(k-1)$-tuple $(1,2,\dots,k)\setminus \{\tau(j)\}$ to the $(k-1)$-tuple $(\tau(1),\tau(2),\dots,\tau(k))\setminus \{\tau(j)\}$.

\begin{definition} [Symmetrization]
Given $k\ge 3$, let $\mathscr P=(\mathcal C,\mathcal T)$ be a $k$-palette. Let $c^\sigma$ be a copy of color $c\in\mathcal C$ for each $\sigma\in \mathbb S_{k-1}$.
Set $\mathcal C^{\mathcal S}:=\{c^\sigma:\ c\in\mathcal C,\ \sigma\in \mathbb S_{k-1}\}$.
For each $\tau\in\mathbb S_k$ and each $k$-tuple $(c_1,\dots,c_k)\in\mathcal T$, define
\[
\tau\cdot(c_1,\dots,c_k)
:=
\bigl(c_{\tau(1)}^{\partial_1\tau},\dots,c_{\tau(k)}^{\partial_k\tau}\bigr),
\]
and set
\[
\mathcal T^{\mathcal S}
:=
\bigcup_{(c_1,\dots,c_k)\in\mathcal T}\ \bigcup_{\tau\in \mathbb S_k}\ \{\tau\cdot(c_1,\dots,c_k)\}.
\]
The \emph{symmetrization} of $\mathscr P$ is the $k$-palette $\mathscr P^{\mathcal S}=(\mathcal C^{\mathcal S},\mathcal T^{\mathcal S})$. 
\end{definition}

Based on the above definition, 
we have the following observation.
\begin{observation}\label{obs:symmetrization-order}
Given $k\ge 3$, let $\mathscr P=(\mathcal C,\mathcal T)$ be a $k$-palette. If a $k$-graph $H$ is $\mathscr P$-colorable, then for every linear order $\prec$ on $V(H)$, the ordered $k$-graph $H^\prec$ is $\mathscr P^{\mathcal S}$-colorable.
\end{observation}

This observation comes from the following fact: Let $\prec_1,\prec_2$ be two different orderings on $V(H)=[N]$. Without loss of generality, we assume that $\prec_1$ is the natural order on $[N]$. Fix any $e=\llbracket i_1,i_2,\dots,i_k\rrbracket\in E(H)$ and suppose $i_1'\prec_2i_2'\prec_2\dots\prec_2i_k'$ is a rearrangement of elements of $e$. Let $\tau_{e}=\alpha^*[(i_1',i_2',\dots,i_k')]$. If $e=i_1\prec_1i_2,\prec_1\dots\prec_1i_k$ is colored by some $(c_1,c_2,\dots,c_k)\in\mathcal{T}$, then $e=i_1'\prec_2i_2',\prec_2\dots\prec_2i_k'$ can be colored by  $\tau_{e}(c_1,c_2,\dots,c_k)\in\mathcal{T}^{\mathcal{S}}$.

Motivated by~\cref{obs:symmetrization-order}, once one palette is chosen as the reference order, the remaining palettes can be incorporated through their symmetrization. 
For this reason, 
the correct object in the multi-palette classification theorem is not simply the ordinary product $\prod_{i=1}^t \mathscr P_i$, but rather the mixed product
\[
\mathscr P_j \times \prod_{i\in[t]\setminus\{j\}} \mathscr P_i^{\mathcal S}.
\]

We now state the corresponding multi-palette classification theorem.

\begin{theorem}\label{thm: classification for multicase}
Given $k\ge 3$ and $t\in \mathbb N$,
let ${\mathscr P}_1, \ldots, {\mathscr P}_t$ and ${\mathscr P}_0$ be $k$-palettes. There exists a $k$-graph
$H$ that is ${\mathscr P}_i$-colorable for every $i \in [t]$ but not ${\mathscr P}_0$-colorable if and only if for
every $j \in [t]$, there is no homomorphism from the palette ${\mathscr P}_j \times \prod_{i \in [t] \setminus \{j\}} {\mathscr P}_i^{\mathcal S}$ to ${\mathscr P}_0$
or to ${\rm rev}({\mathscr P}_0)$.
\end{theorem}

As an immediate corollary, in fact, we have the following more general result.

\begin{corollary}\label{Cor:classification for multi case}
Given $k\ge 3$ and $s,t\in\mathbb N$, let ${\mathscr P}_1,\dots,{\mathscr P}_s$ and ${\mathscr Q}_1,\dots,{\mathscr Q}_t$ be $k$-palettes. There exists a $k$-graph
$F$ that is ${\mathscr P}_i$-colorable for every $i\in[s]$ but not ${\mathscr Q}_j$-colorable for every $j\in[t]$ if and only if for every $i\in[s]$ and every $j\in[t]$, there is no homomorphism from the palette ${\mathscr P}_{i} \times \prod_{i' \in [s] \setminus \{i\}} {\mathscr P}_{i'}^{\mathcal S}$ to ${\mathscr Q}_j$ or to ${\rm rev}({\mathscr Q}_j)$.
\end{corollary}

\begin{proof} Suppose first that there exists a $k$-graph $F$ that is $\mathscr P_i$-colorable for every $i\in[s]$ but not $\mathscr Q_j$-colorable for every $j\in[t]$. Fix $i\in[s]$ and $j\in[t]$. Applying \cref{thm: classification for multicase} with ${\mathscr P}_0={\mathscr Q}_j$ yields the statement that there is no homomorphism from the palette ${\mathscr P}_i\times \prod_{i'\in[s]\setminus\{i\}} {\mathscr P}_{i'}^{\mathcal S}$ to ${\mathscr Q}_j$ or to ${\rm rev}({\mathscr Q}_j)$.

For the converse, fix $j\in[t]$. Applying \cref{thm: classification for multicase} with
${\mathscr P}_0={\mathscr Q}_j$,
we obtain a $k$-graph $F_j$ that is ${\mathscr P}_i$-colorable for every $i\in[s]$ but not ${\mathscr Q}_j$-colorable. 
Let $F$ be the disjoint union of $F_1,\dots,F_t$. Then $F$ is ${\mathscr P}_i$-colorable for every $i\in[s]$ but not ${\mathscr Q}_j$-colorable for every $j \in [t]$.    
\end{proof}

Next, we state and prove two lemmas that are useful for the proof of \cref{thm: classification for multicase}. The first is a Ramsey-type majority lemma.

\begin{theorem}\label{lm: exists majority function lemma}
Given integers $k\ge 3$ and $n,t,\ell \in \mathbb{N}$, there exists $N \in \mathbb{N}$ such that the following holds. For every function $F : ({[N]}^{t})^{k-1} \to [\ell]$, there exist pairwise disjoint $J_1, J_2,\dots, J_t \subseteq [N]$ of each size $n$ and a function $\Gamma:(\mathbb S_{k-1})^t\to[\ell]$ such that
\[
F(\mathbf{x}_1, \mathbf{x}_2,\dots, \mathbf{x}_{k-1})= \Gamma \bigl(\tau_1,\tau_2,\dots,\tau_{t}\bigr),
\]
whenever the following conditions hold:
\stepcounter{propcounter}
\begin{enumerate}[label = {{\rm (\Alph{propcounter}\arabic{enumi})}}]
\item \label{J1} each $\mathbf x_r=(x_r^1,\dots,x_r^t)\in J_1\times\cdots\times J_t$ for every $r\in[k-1]$;
\item \label{J2} for every $s\in[t]$, the numbers $x_1^s,\dots,x_{k-1}^s$
are pairwise distinct;
\item \label{J3} $\tau_s=\alpha^*[(x_1^s,\dots,x_{k-1}^s)]$ for every $s\in[t]$.
\end{enumerate}
\end{theorem}

\begin{proof}
Let $N:=R_{(k-1)t}\bigl(tn,\ell^{((k-1)!)^t}\bigr)$.
Fix a function
$F:([N]^t)^{k-1}\to[\ell]$.
For each
\[
Y=\llbracket z_1,\dots,z_{(k-1)t}\rrbracket\in \binom{[N]}{(k-1)t},
\]
write $y_r^s:=z_{(s-1)(k-1)+r}$ for each $s\in[t]$ and $r\in[k-1]$. Thus, for each fixed $s$,
\[
y_1^s<\cdots<y_{k-1}^s.
\]
We color $Y$ by the function
\[
\Gamma_Y:(\mathbb S_{k-1})^t\to[\ell]
\]
defined as follows.

Given $(\tau_1,\ldots,\tau_t)\in (S_{k-1})^t$, with
\[
  \tau_s=
  \begin{pmatrix}
  1 & \cdots & k-1\\
  \tau_s(1) & \cdots & \tau_s(k-1)
  \end{pmatrix}
\]
for each $s\in[t]$, let
\[
  \mathbf x_r=\bigl(y^1_{\tau_1(r)},\ldots,y^t_{\tau_t(r)}\bigr)
  \in [N]^t
\]
for each $r\in[k-1]$.

Set
\[
\Gamma_Y(\tau_1,\dots,\tau_t):=F(\mathbf x_1,\dots,\mathbf x_{k-1}).
\]
Thus each $(k-1)t$-subset of $[N]$ receives one of
$\ell^{((k-1)!)^t}$ possible colors.

By the choice of $N$, there exists a subset
$S=\llbracket x_1,\dots, x_{tn}\rrbracket\subseteq[N]$
of size $tn$ such that every $(k-1)t$-subset of $S$ receives the same function, say $\Gamma$. For each $s\in[t]$, define
\[
J_s:=\{x_{(s-1)n+1},\dots,x_{sn}\}.
\]
Then each $J_s$ has size $n$.

Now for each $r\in[k-1]$,
let
$\mathbf x_r=(x_r^1,\dots,x_r^t)\in J_1\times\cdots\times J_t$
be as in the statement. For each $s\in[t]$, the numbers
$x_1^s,\dots,x_{k-1}^s$
are pairwise distinct and belong to $J_s$, so their union over all $s$ is a $(k-1)t$-subset
$Y\subseteq S$.  
Since every $(k-1)t$-subset of $S$ receives the same function $\Gamma$, we have
$\Gamma_Y=\Gamma$.
By construction of $\Gamma_Y$, if
\[
\tau_s=\alpha^*[(x_1^s,\dots,x_{k-1}^s)]
 \text{ for every }s\in[t],
\]
then
\[
F(\mathbf x_1,\dots,\mathbf x_{k-1})=\Gamma_Y(\tau_1,\dots,\tau_t)=\Gamma(\tau_1,\dots,\tau_t).
\]
This proves the lemma.
\end{proof}

The next lemma is the multi-palette analog of \cref{lem:single order}. It is the key step in the proof of \cref{thm: classification for multicase}.

\begin{lemma}\label{lemma: key lemma for classification for multicase}
Let ${\mathscr P}_1,\dots,{\mathscr P}_t$ and ${\mathscr P}_0$ be $k$-palettes. Suppose that every ordered $k$-graph $H^{\prec}$ such that
\begin{itemize}
\item $H^{\prec}$ is ${\mathscr P}_1$-colorable, and
\item $H$ is ${\mathscr P}_j$-colorable for every $2\le j\le t$,
\end{itemize}
is also ${\mathscr P}_0$-colorable.
Then there exists a homomorphism
from the palette ${\mathscr P}_1\times\prod_{j=2}^t{\mathscr P}_j^{\mathcal S}$ to ${\mathscr P}_0$.
\end{lemma}

\begin{proof}
Let $\mathscr P_s=(\mathcal C_s,\mathcal T_s)$ be a $k$-palette for every $s\in[t]\cup\{0\}$. Set
\[
\mathcal C^{\times}:=\prod_{s=1}^t \mathcal C_s,
\qquad
m_0:=|\mathcal C^{\times}|=\prod_{s=1}^t |\mathcal C_s|,
\qquad
m:=m_0\cdot |\mathcal C_0|,
\qquad
\ell:=|\mathcal C_0|^{m_0},
\qquad
d_{k-1}:=\frac{1}{2m}.
\]
Let $R$ be the integer obtained from \cref{lm: exists majority function lemma}
with parameters $k, n=k$, $t$, and $\ell$.
Choose auxiliary constants satisfying
\[
{1}/{n_1}\ll {1}/{r}, \delta \ll d_2,\dots,d_{k-2}\ll \delta_{k-1} \ll d_{k-1}, {1}/{k},
\]
where $1/d_i\in \mathbb N$ for all $i=2,\dots,k-2$, such that the assumptions of~\cref{E-lemma} are satisfied with $k$ replaced by $k-1$ and $F=K_k^{(k-1)}$. Set $\eta=R^{-t}$ and choose $\delta_{k-1}>0$ sufficiently small such that
\begin{equation}\label{eq:delta-multicase}
2\sqrt{\delta_{k-1}}\,m\,R^{t(k-1)}<1.
\end{equation}
Apply \cref{multicolor hypergraph regularity lemma} with parameters
$R^t, k-1, m, \eta,  \delta_{k-1}, r, \delta$
to obtain integers $t_0$ and $n_0$. Next choose $\eps>0$ sufficiently small such that 
\begin{equation}\label{eq:eps-multicase}
\eps<\min\left\{\frac12,\ (1-\eta)t_0^{-(k-1)}\prod_{j=2}^{k-2}d_j^{\binom{k-1}{j}}\right\},
\end{equation}
and let $N\ge \max\{n_0,n_1\}$ be  large enough such that  $R^t\cdot t_0!\mid N$ and
\begin{equation}\label{eq:positive probability-order}
1-2^{\binom{N}{k-2}}\cdot m_0\cdot  e^{-\frac{\varepsilon^3 N^{k-1}}{2m_0}}>0.
\end{equation}

Let $G$ be the complete $(k-1)$-graph on vertex set $[N]$. Partition $[N]$ into equal parts
\[
[N]=\bigcup_{\mathbf i\in[R]^t}U_{\mathbf i},
\qquad
\mathbf i=(i^1,\dots,i^t).
\]
For each $s\in[t]$, choose a linear order $\prec_s$ on $[N]$ such that
\[
u\prec_s u'
\quad\text{whenever}\quad
u\in U_{(i_1,\dots,i_t)},\ u'\in U_{(i_1',\dots,i_t')},\ \text{and } i_s<i'_s.
\]
In particular, we can take $\prec_1$ to be the natural order on $[N]$.

Now color the edges of $G$ independently and uniformly at random with colors from $\mathcal C^{\times}$. For every $\mathbf c\in\mathcal C^{\times}$, let $G_{\mathbf c}$ denote the spanning subgraph of $G$ consisting of all edges colored $\mathbf c$.
By the same Chernoff-bound argument as in the proof of \cref{lem:single order}, there exists a coloring
\[
\mathbf f:E(G)\to \mathcal C^{\times}
\]
such that for every $(k-2)$-graph $G^{(k-2)}$ on $[N]$ with
$|\mathcal K_{k-1}(G^{(k-2)})|\ge \eps N^{k-1}$
and every $\mathbf c\in\mathcal C^{\times}$, we have
\begin{equation}\label{eq:random-color-property-multicase}
|\mathcal K_{k-1}(G^{(k-2)})\cap E(G_{\mathbf c})|
\ge
(1-\eps)\frac{|\mathcal K_{k-1}(G^{(k-2)})|}{m_0}.
\end{equation}
Write
\[
\mathbf f=(f_1,\dots,f_t),
\]
where $f_s:E(G)\to \mathcal C_s$ is the $s$-th coordinate coloring.

We now define a $k$-graph $H$ on $[N]$.
A $k$-set $e\in\binom{[N]}{k}$ forms an edge of $H$ if and only if for every $s\in[t]$, when the vertices of $e$ are listed in increasing order with respect to $\prec_s$ as
$w_1^s\prec_s \cdots \prec_s w_k^s$,
we have
\[
\bigl(
f_s(e\setminus\{w_1^s\}),
f_s(e\setminus\{w_2^s\}),
\dots,
f_s(e\setminus\{w_k^s\})
\bigr)\in\mathcal T_s.
\] 
By construction, $H^{\prec_1}$ is $\mathscr P_1$-colorable, and $H$ is $\mathscr P_s$-colorable for every $2\le s\le t$.
By the assumption of the lemma, $H^{\prec_1}$ is also $\mathscr P_0$-colorable, and let
\[
f_0:E(G)\to\mathcal C_0
\]
be a coloring witnessing this.

For each
$(\mathbf c,c_0)\in\mathcal C^{\times}\times\mathcal C_0$,
let $G_{\mathbf c,c_0}$ be the spanning subgraph of $G$ with edge set \[
E(G_{\mathbf c,c_0})
=
\{e\in E(G):\mathbf f(e)=\mathbf c,\ f_0(e)=c_0\}.
\]

Applying \cref{multicolor hypergraph regularity lemma} to this $(\mathcal C^{\times}\times\mathcal C_0)$-colored $(k-1)$-graph $G$ with the initial equipartition $\{U_{\mathbf i}: \mathbf i\in [R]^{t}\}$, we obtain a family of partitions
$\bm{\mathcal P}=\bm{\mathcal P}(k-2,\mathbf a)$
where $a_i=1/d_i$ for $i=2,\dots,k-2$ and $1/\eta<a_1<t_0$, such that
\begin{enumerate}
\item $\bm{\mathcal P}$ is $(\eta,\delta(\mathbf a),t_0)$-equitable and $\mathcal P^{(1)}$ refines the partition $\{U_{\mathbf i}:\mathbf i\in [R]^{t}\}$;
\item for each
$(\mathbf c,c_0)\in\mathcal C^{\times}\times\mathcal C_0$, the $(k-1)$-graph $G_{\mathbf c,c_0}$ is $(\delta_{k-1},r(\mathbf a))$-regular with respect to $\bm{\mathcal P}$.
\end{enumerate}

We say that a $(k-1)$-tuple of clusters is \emph{good} for
$(\mathbf c,c_0)\in\mathcal C^{\times}\times\mathcal C_0$
if $G_{\mathbf c,c_0}$ is $(\delta_{k-1},r(\mathbf a))$-regular with respect to all but at most a $\sqrt{\delta_{k-1}}$-fraction of the polyads induced on this $(k-1)$-tuple of clusters. A $(k-1)$-tuple of clusters is called \emph{all-good} if it is good for every $(\mathbf c,c_0)\in\mathcal C^{\times}\times\mathcal C_0$. Similarly as \cref{claim-single} and \cref{claim:choice of polyads}, we have the following statement.

\begin{claim}\label{claim:number of all good tuple}
There exist clusters $V_{\mathbf i}\in \mathcal P^{(1)}$ with $V_{\mathbf i}\subseteq U_{\mathbf i}$ for every $\mathbf i\in[R]^t$ such that every $(k-1)$-tuple among the clusters $\{V_{\mathbf i}:\mathbf i\in[R]^t\}$ is all-good.
Moreover, one can choose a $(|R|^t, k-2)$-complex $\{Q^{(i)}\}_{i=1}^{k-2}$ with clusters $\{V_{\mathbf i}:\mathbf i\in[R]^t\}$ such that the following hold.
\stepcounter{propcounter}
\begin{enumerate}[label = {{\rm (\Alph{propcounter}\arabic{enumi})}}]
\item \label{polyad-11} $Q^{(1)}[V_{\mathbf i}]=V_{\mathbf i}$ for each $\mathbf i \in [R]^t$.

\item \label{polyad-22} For every $2\le i\le k-2$ and every $\Lambda_i\in \binom{[R]^t}{i}$, the restriction $Q^{(i)}_{\Lambda_i}:=Q^{(i)}[\cup_{\mathbf j\in \Lambda_i }V_{\mathbf j}]$ is a $i$-cell.
Moreover, these cells are chosen consistently: for every
$2\le i\le k-3$ and every $\Lambda_{i+1}\in\binom{[R]^t}{i+1}$, $Q_{\Lambda_{i+1}}^{(i+1)}
\subseteq
\mathcal K_{i+1}\bigl(\hat Q_{\Lambda_{i+1}}^{(i)}\bigr)$.

\item \label{polyad-33} For every $I\in \binom{[R]^t}{k-1}$, the union 
\[\hat Q_{I}^{(k-2)}
:=
\bigcup\Bigl\{Q_{\Lambda_{k-2}}^{(k-2)}: \Lambda_{k-2}\in \binom{I}{k-2}\Bigr\}
\]
is a polyad on the $(k-1)$-tuple of clusters indexed by $I$.

\item \label{polyad-44} For every $I\in \binom{[R]^t}{k-1}$
and each pair $(\mathbf c,c_0)\in\mathcal C^{\times}\times\mathcal C_0$,
the $(k-1)$-graph $G_{\mathbf c,c_0}$ is $(\delta_{k-1},r(\mathbf a))$-regular with respect to $\hat Q_I^{(k-2)}$.
\end{enumerate}
\end{claim}

Since the proof is identical to that of~\cref{claim-single} and~\cref{claim:choice of polyads}, we omit the proof of~\cref{claim:number of all good tuple}.
Next, fix an $(R^t,k-2)$-complex
$\{Q^{(i)}\}_{i=1}^{k-2}$
with clusters $\{V_{\mathbf i}:\mathbf i\in[R]^t\}$ satisfying~\cref{claim:number of all good tuple}. For each $I\in \binom{[R]^t}{k-1}$, define a function
\[
\phi_I:\mathcal C^{\times}\to \mathcal C_0
\]
as follows: for each $\mathbf c\in \mathcal C^{\times}$, let $\phi_I(\mathbf c)$ be a color in $\mathcal C_0$ maximizing
\[
\bigl|\mathcal K_{k-1}(\hat Q_I^{(k-2)})\cap E(G_{\mathbf c,c_0})\bigr|
\]
over all $c_0\in\mathcal C_0$.

Let $\Phi:=\{\phi:\mathcal C^{\times}\to\mathcal C_0\}$. Then we have $|\Phi|=\ell$.
Define a function
\[
F_{\mathrm{maj}}:([R]^t)^{k-1}\to\Phi
\]
as follows: 
if $\mathbf z_1,\dots,\mathbf z_{k-1}\in[R]^t$  pairwise have no common coordinate, then set
\[
F_{\mathrm{maj}}(\mathbf z_1,\dots,\mathbf z_{k-1})
:=
\phi_{\{\mathbf z_1,\dots,\mathbf z_{k-1}\}}.
\]
Otherwise, set $F_{\mathrm{maj}}$ to be an arbitrary function in $\Phi$.

We identify $\Phi$ with $[\ell]$ and apply \cref{lm: exists majority function lemma} to the function $F_{\mathrm{maj}}$. Then we obtain 
pairwise disjoint $k$-element sets $J_1,\dots,J_t\subseteq [R]$
and a function
\[
\Gamma:(\mathbb S_{k-1})^t\to \Phi
\]
such that
\begin{equation}\label{eq:gamma-multicase}
F_{\mathrm{maj}}(\mathbf z_1,\dots,\mathbf z_{k-1})
=
\Gamma(\sigma_1,\dots,\sigma_t)
\end{equation}
whenever each
$\mathbf z_r=(z_r^1,\dots,z_r^t)\in J_1\times\cdots\times J_t$ for every $r\in[k-1]$,
any two distinct tuples $\mathbf z_r$ and $\mathbf z_{r'}$ have no common coordinate in any position, and
$\sigma_s=\alpha^*[(z_1^s,\dots,z_{k-1}^s)]$ for every $s\in[t]$.

For each $\mathscr P_s^{\mathcal S}=(\mathcal C_s^{\mathcal S},\mathcal T_s^{\mathcal S})$ and $x=c^\sigma\in\mathcal C_s^{\mathcal S}$, we write
\[
\operatorname{org}_s(x):=c
\qquad\text{and}\qquad
\operatorname{pat}_s(x):=\sigma.
\]
Now define
\[
h:\mathcal C_1\times\prod_{s=2}^t \mathcal C_s^{\mathcal S}\to \mathcal C_0
\]
by
\[
h(x_1,x_2,\dots,x_t)
:=
\Gamma\bigl(\mathrm{id},\operatorname{pat}_2(x_2),\dots,\operatorname{pat}_t(x_t)\bigr)
\bigl(x_1,\operatorname{org}_2(x_2),\dots,\operatorname{org}_t(x_t)\bigr).
\]
Note that $\Gamma\bigl(\mathrm{id},\operatorname{pat}_2(x_2),\dots,\operatorname{pat}_t(x_t)\bigr)\in \Phi$, i.e., $\Gamma\bigl(\mathrm{id},\operatorname{pat}_2(x_2),\dots,\operatorname{pat}_t(x_t)\bigr)$ is a function from $\mathcal C^{\times}$ to $\mathcal C_0$
We claim that $h$ is a homomorphism from
\[
\mathscr P_1\times\prod_{s=2}^t \mathscr P_s^{\mathcal S}
\]
to $\mathscr P_0$.

Let
\[
(\mathbf y_1,\dots,\mathbf y_k)\in \mathcal  T\!\left({\mathscr P}_1\times\prod_{j=2}^t{\mathscr P}_j^{\mathcal S}\right)
\]
and write $\mathbf y_q=(x_{1,q},x_{2,q},\dots,x_{t,q})$ for each $q\in[k]$.
Thus, $(x_{1,1},\dots,x_{1,k})\in \mathcal T_1$
and for each $s\in\{2,\dots,t\}$,
$(x_{s,1},\dots,x_{s,k})\in \mathcal T_s^{\mathcal S}$.

For $s=1$, set $\tau_1:=\mathrm{id}$ and $c_{1,q}:=x_{1,q}$ for each $q\in[k]$.
For each $s\in\{2,\dots,t\}$, since
$(x_{s,1},\dots,x_{s,k})\in \mathcal T_s^{\mathcal S}$,
there exist a permutation $\tau_s\in \mathbb S_k$ and a tuple
$(c_{s,1},\dots,c_{s,k})\in \mathcal T_s$
such that
\[
(x_{s,1},\dots,x_{s,k})
=
\tau_s\cdot (c_{s,1},\dots,c_{s,k}).
\]
Equivalently,
\[
x_{s,q}=c_{s,\tau_s(q)}^{\partial_q\tau_s}
\qquad\text{for every } q\in[k].
\]
Hence
\begin{equation}\label{eq:org-pat-multicase}
\operatorname{org}_s(x_{s,q})=c_{s,\tau_s(q)}
\qquad\text{and}\qquad
\operatorname{pat}_s(x_{s,q})=\partial_q\tau_s
\qquad\text{for every } q\in[k].
\end{equation}
For each $s\in[t]$, write
\[
J_s=\llbracket j_1^s, j_2^s, \cdots, j_k^s\rrbracket \subseteq [R],
\]
and define
\[
i_q^s:=j_{\tau_s(q)}^s
\qquad\text{for } q\in[k],\ s\in[t].
\]
For each $q\in[k]$, let
\[
\mathbf i_q:=(i_q^1,\dots,i_q^t)\in J_1\times\cdots\times J_t.
\]
Then for every $s\in[t]$,
\[
\alpha^*[(i_1^s,\dots,i_k^s)]=\tau_s.
\]
In particular, since $\tau_1=\mathrm{id}$, we have
$i_1^1<i_2^1<\cdots<i_k^1$.

For each $q\in[k]$, let
\[
\Lambda_q:=[k]\setminus\{q\}
\qquad\text{and}\qquad
I_q:=\{\mathbf i_r:r\in\Lambda_q\}.
\]
Note that $\partial_q\tau_s=\alpha^*[((i_1^s,\dots,i_k^s)\setminus\{i_q^s\})]$

Recalling the definition of $\partial_q\tau_s$, deleting the $q$-th entry from the sequence
\[
i_1^s,\dots,i_k^s
\]
leaves a $(k-1)$-tuple with pattern $\partial_q\tau_s$.
Therefore, applying \eqref{eq:gamma-multicase} to the ordered tuple $(\mathbf i_r)_{r\in\Lambda_q}$, taken in the natural order of the indices $r$, gives
\[
F_{\mathrm{maj}}\bigl((\mathbf i_r)_{r\in\Lambda_q}\bigr) 
=
\Gamma\bigl(\partial_q\tau_1,\partial_q\tau_2,\dots,\partial_q\tau_t\bigr)
=
\Gamma\bigl(\mathrm{id},\partial_q\tau_2,\dots,\partial_q\tau_t\bigr).
\]
Since the tuples $\mathbf i_r$ with $r\in\Lambda_q$ are pairwise distinct, we also have
\[
F_{\mathrm{maj}}\bigl((\mathbf i_r)_{r\in\Lambda_q}\bigr)=\phi_{I_q}.
\]
Thus
\begin{equation}\label{eq:phiIq}
\phi_{I_q}
=
\Gamma\bigl(\mathrm{id},\partial_q\tau_2,\dots,\partial_q\tau_t\bigr).
\end{equation}

Now for each $q\in[k]$, define
\[
\mathbf c_q:=
\bigl(x_{1,q},\,\operatorname{org}_2(x_{2,q}),\,\dots,\,\operatorname{org}_t(x_{t,q})\bigr)
\in \mathcal C^{\times}.
\]
By \eqref{eq:org-pat-multicase} and \eqref{eq:phiIq}, we have
\[
\begin{aligned}
h(\mathbf y_q)
&=
\Gamma\bigl(\mathrm{id},\operatorname{pat}_2(x_{2,q}),\dots,\operatorname{pat}_t(x_{t,q})\bigr)
\bigl(x_{1,q},\operatorname{org}_2(x_{2,q}),\dots,\operatorname{org}_t(x_{t,q})\bigr) \\
&=
\Gamma\bigl(\mathrm{id},\partial_q\tau_2,\dots,\partial_q\tau_t\bigr)(\mathbf c_q) \\
&=
\phi_{I_q}(\mathbf c_q).
\end{aligned}
\]
By condition~\ref{p4} of \cref{def-eq-partition}, we have
\[
|\mathcal K_{k-1}(\hat Q_{I_q}^{(k-2)})|
\ge
(1-\eta)\prod_{j=1}^{k-2}\left(\frac1{a_j}\right)^{\binom{k-1}{j}}N^{k-1}\ge \eps N^{k-1},
\]
where the last inequality follows from  \eqref{eq:eps-multicase}.
Therefore, by~\eqref{eq:random-color-property-multicase},
\[
|\mathcal K_{k-1}(\hat Q_{I_q}^{(k-2)})\cap E(G_{\mathbf c_q})|
\ge
(1-\eps)
\frac{|\mathcal K_{k-1}(\hat Q_{I_q}^{(k-2)})|}{m_0}.
\]
Since $\phi_{I_q}(\mathbf c_q)=h(\mathbf y_q)$ is chosen as a majority $f_0$-color among those members of
$\mathcal K_{k-1}(\hat Q_{I_q}^{(k-2)})$ whose $\mathbf f$-color is $\mathbf c_q$, we obtain
\[
d\bigl(G_{\mathbf c_q,h(\mathbf y_q)}\mid \hat Q_{I_q}^{(k-2)}\bigr)
\ge
\frac{1-\eps}{m_0|\mathcal C_0|}
=
\frac{1-\eps}{m}
\ge
d_{k-1}.
\]
Moreover, by~\cref{claim:number of all good tuple}, the $(k-1)$-graph
$G_{\mathbf c_q,h(\mathbf y_q)}$
is $(\delta_{k-1},r(\mathbf a))$-regular with respect to $\hat Q_{I_q}^{(k-2)}$ for every $q\in[k]$.

Now consider the $(k-1)$-graph
\[
G':=\bigcup_{q=1}^k
G_{\mathbf c_q,h(\mathbf y_q)}
\Bigl[\bigcup_{r\in\Lambda_q}V_{\mathbf i_r}\Bigr].
\]
By~\cref{E-lemma}, $G'$ contains a copy of $K_k^{(k-1)}$ on the clusters $V_{\mathbf i_1},V_{\mathbf i_2},\dots,V_{\mathbf i_k}$.
Let this copy correspond to vertices $v_q\in V_{\mathbf i_q}$ for $q\in[k]$.
Since $i_1^1<i_2^1<\cdots<i_k^1$ and $\prec_1$ is the natural order, we have
$v_1<v_2<\cdots<v_k$.
Set $e:=\{v_1,\dots,v_k\}$.
For each $q\in[k]$, the $(k-1)$-set $e\setminus\{v_q\}$ belongs to
$G_{\mathbf c_q,h(\mathbf y_q)}$, so its $\mathbf f$-color is $\mathbf c_q$ and its $f_0$-color is $h(\mathbf y_q)$.

We now verify that $e\in E(H)$.
Fix $s\in[t]$.
Because $\alpha^*[(i_1^s,\dots,i_k^s)]=\tau_s$, the increasing rearrangement of the numbers
$i_1^s,\dots,i_k^s$ is
\[
i_{\tau_s^{-1}(1)}^s<i_{\tau_s^{-1}(2)}^s<\cdots<i_{\tau_s^{-1}(k)}^s.
\]
Hence the vertices of $e$ are ordered by $\prec_s$ as
\[
v_{\tau_s^{-1}(1)}\prec_s v_{\tau_s^{-1}(2)}\prec_s \cdots \prec_s v_{\tau_s^{-1}(k)}.
\]
For each $j\in[k]$, let $q=\tau_s^{-1}(j)$.
Then the $s$-th coordinate of the vector $\mathbf c_q$ is
\[
\begin{cases}
x_{1,q}=c_{1,j}, & \text{if } s=1,\\[1ex]
\operatorname{org}_s(x_{s,q})=c_{s,\tau_s(q)}=c_{s,j}, & \text{if } s\ge 2.
\end{cases}
\]
Therefore
\[
f_s(e\setminus\{v_{\tau_s^{-1}(j)}\})=c_{s,j}
\qquad\text{for every } j\in[k].
\]
It follows that
\[
\bigl(
f_s(e\setminus\{v_{\tau_s^{-1}(1)}\}),
f_s(e\setminus\{v_{\tau_s^{-1}(2)}\}),
\dots,
f_s(e\setminus\{v_{\tau_s^{-1}(k)}\})
\bigr)
=
(c_{s,1},c_{s,2},\dots,c_{s,k})
\in \mathcal T_s.
\]
Since this holds for every $s\in[t]$, the definition of $H$ gives $e\in E(H)$.

Finally, $v_1<\cdots<v_k$ and $f_0$ witnesses that $H^{\prec_1}$ is $\mathscr P_0$-colorable, so
\[
\bigl(h(\mathbf y_1),h(\mathbf y_2),\dots,h(\mathbf y_k)\bigr)\in \mathcal T_0.
\]
Therefore, $h$ is a homomorphism from
$\mathscr P_1\times\prod_{s=2}^t \mathscr P_s^{\mathcal S}$
to $\mathscr P_0$.
\end{proof}

We shall also need the following ordering lemma, which is due to Fishburn and Graham~\cite{F-G-order-93}. 

\begin{lemma}[\cite{F-G-order-93}]\label{lem:order lemma}
For every $n\in\mathbb N$ and $t\in\mathbb N$, there exists $R\in\mathbb N$ with the following property. For every linear order $\prec$ on $[R]^{t}$, there exist $n$-element sets
$J_1, \ldots, J_t \subseteq [R]$, a permutation $\sigma \in \mathbb S_t$, and a function $\zeta: [t] \to \{-1, +1\}$
such that the linear order $\prec$ of the $t$-tuples $(x_1, \ldots, x_t) \in J_1 \times \cdots \times J_t$ is the
lexicographic order given by $(\zeta(1)x_{\sigma(1)}, \dots, \zeta(t)x_{\sigma(t)})$, where for each $i\in [t]$, if $\zeta(i)=+1$, then the  $\sigma(i)$-th coordinate is compared
with the natural order on $[R]$, while if $\zeta(i)=-1$, it is compared
with the reverse natural order.
\end{lemma}

Using~\cref{lem:order lemma}, we next prove the version of~\cref{lm4.8} for multiple palettes.

\begin{lemma}\label{thm6.dierorder} 
For every $k\ge 3$ and $t\in\mathbb N$, let $\mathscr{P}_1, \ldots, \mathscr{P}_t$ and $\mathscr{P}_0$ be $k$-palettes. Suppose that for every $s\in[t]$ there exist ordered $k$-graphs $H_s^+$ and $H_s^-$ such that
\begin{itemize}
\item $H_s^+$ and $H_s^-$ are $\mathscr P_s$-colorable,
\item the unordered $k$-graphs $\widetilde H_s^+$ and $\widetilde H_s^-$ are  $\mathscr P_{s'}$-colorable for every $s' \in [t] \setminus \{s\}$,
\item $H_s^+$ is not $\mathscr P_0$-colorable, and
\item $H_s^-$ is not ${\rm rev}(\mathscr{P}_0)$-colorable.
\end{itemize}
Then, there exists a $k$-graph $H$ that is $\mathscr{P}_s$-colorable for every $s \in [t]$ but not
$\mathscr P_0$-colorable.
\end{lemma}

\begin{proof}
Fix the $k$-palettes $\mathscr{P}_s = (\mathcal C_s, \mathcal{T}_s)$ for $s \in [t] \cup \{0\}$ and the ordered $k$-graphs
$H_s^+$ and $H_s^-$ for $s \in [t]$, as in the statement of the lemma. Let \[
c:=\max \{|\mathcal C_s|:s\in[t]\},
\qquad
n:=\max\bigl\{|V(H_s^+)|,\ |V(H_s^-)|:s\in[t]\bigr\},
\qquad
p:=c^{-t n^{k-1}}.
\]
Apply \cref{lem:order lemma} with $n$ and $t$ to obtain an integer $R$. Choose $N$ sufficiently large such that
\begin{equation}\label{eq:random-multicase}
N!\cdot (1-p)^{N^2/(R^{4t}(R^t!)^t)}<1.
\end{equation}
Such an $N$ exists because $N!$ is at most $e^{N \log N}$ and the integers $c, n, t$ and $R$ are fixed.

We now define a random $k$-graph $H$ on the vertex set $[N]$. First choose linear orders $\prec_1, \dots, \prec_t$ on $[N]$ independently and uniformly at random. Next, for each $(k-1)$-subset $X\in\binom{[N]}{k-1}$ choose independently and uniformly a color vector \[
\mathbf f(X)=\bigl(f_1(X),\dots,f_t(X)\bigr)\in \mathcal C_1\times\cdots\times\mathcal C_t.
\] Finally, a $k$-subset $e=\{v_1,\dots,v_k\}\subset [N]$ forms an edge of $H$ if and only if for every $s\in[t]$, when the vertices of $e$ are listed in increasing order with respect to $\prec_s$ as $v_1'\prec_s \dots \prec_s v'_k$, we have
\[(f_s(e\backslash\{v_1'\}),f_s(e\backslash\{v_{2}'\}),\dots,f_s(e\backslash\{v_k'\}))\in\mathcal{T}_s.\]  
By construction, the resulting $k$-graph $H$ is $\mathscr P_s$-colorable for every $s\in[t]$.

We now show that with positive probability $H$ is not $\mathscr P_0$-colorable. Let $L$ be a linear $R^t$-graph on $[N]$ with the maximum possible number of edges. Since an $R^t$-edge can intersect fewer than
\[
\binom{R^t}{2}\binom{N-2}{R^t-2}
\]
other $R^t$-edges in the complete $R^t$-graph on $[N]$, we have
\[
|E(L)|\ge
\frac{\binom{N}{R^t}}{\binom{R^t}{2}\binom{N-2}{R^t-2}}
=
\frac{2N(N-1)}{R^{2t}(R^t-1)^2}
\ge
\frac{N^2}{R^{4t}}.
\]
Enumerate the edges of $L$ as $F_1,\dots,F_m$ with $m:=|E(L)|$.
Index the vertices of each $F_i$ by the set $[R]^t$, writing
$F_i=\{v_{\mathbf j}^i:\mathbf j\in[R]^t\}$.
For each $i\in[m]$ and each $s\in[t]$, let $E_s^i$ be the event that the order $\prec_s$ is consistent with the $s$-th coordinate on $F_i$, that is,
\[
v_{\mathbf j}^i\prec_s v_{\mathbf j'}^i
\qquad\text{whenever}\qquad
j_s<j_s'.
\]
Since among the $(R^t)!$ possible orders of the vertices of $F_i$ at least one has this property, we have
\[
\mathbb P(E_s^i)\ge \frac1{(R^t)!}.
\]
Since the orders $\prec_1,\dots,\prec_t$ are chosen independently, we have
\[
\mathbb P\Bigl(\bigcap_{s=1}^t E_s^i\Bigr)\ge \frac1{(R^t!)^t}.
\]
We call $F_i$ \emph{good} if all events $E_s^i$ occur. Hence the expected number of good sets is at least
\[
m\cdot \frac1{(R^t!)^t}
\ge
\frac{N^2}{R^{4t}(R^t!)^t}.
\]
Therefore there exists a choice of the orders
$\prec_1,\dots,\prec_t$
for which the number of good sets is at least
\[
M:=\frac{N^2}{R^{4t}(R^t!)^t}.
\]
Fix such a choice of the orders. From now on, only the color vectors $\mathbf f(X)$ remain random.

Fix any linear order $\prec$ on $[N]$. For each good set $F_i$, identify it with $[R]^t$ via the above indexing and apply \cref{lem:order lemma} to the order $\prec$ restricted to $F_i$. 
Then there exist $n$-element sets
$J_1, \ldots, J_t \subseteq [R]$,
a permutation $\sigma$ of $[t]$, and a sign function
$\zeta:[t]\to\{-1,+1\}$ such that the restriction of $\prec$ to
$J_1\times\cdots\times J_t$
is lexicographic in the signed permuted coordinates
$\bigl(\zeta(1)x_{\sigma(1)},\dots,\zeta(t)x_{\sigma(t)}\bigr)$.
In particular, on the vertices of $F_i$ indexed by
$J_1\times\cdots\times J_t$, one of the following the
conclusions holds:
\begin{itemize}
\item if $\zeta(1)=1$, $v_{\mathbf j}^i \prec v_{\mathbf j'}^i$ for  $\mathbf j=(j_1,\ldots,j_t)$, $\mathbf j'=(j_1',\ldots,j_t') \in J_1 \times \cdots \times J_t$ with $j_{\sigma(1)} < j_{\sigma(1)}'$;
\item if $\zeta(1)=-1$, $v_{\mathbf j}^i \prec v_{\mathbf j'}^i$ for  $\mathbf j=(j_1,\ldots,j_t)$, $ \mathbf j'=(j_1',\ldots,j_t') \in J_1 \times \cdots \times J_t$ with $j_{\sigma(1)} > j_{\sigma(1)}'$.
\end{itemize}

We now choose one of the ordered $k$-graphs $H_{\sigma(1)}^+$ and $H_{\sigma(1)}^-$
depending on the sign of $\zeta(1)$. If $\zeta(1)=+1$, let $G:=H_{\sigma(1)}^+$.
If $\zeta(1)=-1$, let $G:=H_{\sigma(1)}^-$.
For each $s\in[t]$, choose an ordering $<_s$ of $V(G)$ as follows:
\begin{itemize}
\item if $s=\sigma(1)$, take the given order of $G$ witnessing that $G$ is $\mathscr P_{\sigma(1)}$-colorable;
\item if $s\neq \sigma(1)$, choose any order of $V(G)$ witnessing that the $k$-graph $\widetilde{G}$ is $\mathscr P_s$-colorable.
\end{itemize}
Since $|J_s|=n\ge |V(G)|$ for every $s\in[t]$, 
we choose, for each $s\in[t]$, an injective map
$\iota_s:V(G)\to J_s$
that preserves the order $<_s$. For each vertex $w\in V(G)$, set
\[
\mathbf j(w):=\bigl(\iota_1(w),\dots,\iota_t(w)\bigr)\in J_1\times\cdots\times J_t.
\]
Since $F_i$ is good, for each $s\in[t]$ the order $\prec_s$ on the vertices $\{v_{\mathbf j(w)}^i: w\in V(G)\}$
agrees with the chosen order $<_s$. 
Thus, if the random color vectors on all $(k-1)$-subsets of these vertices match the witness colorings for $\mathscr P_s$-colorable of $G$ described above, then
there exists a copy of $G$ on the vertices $\{v_{\mathbf j(w)}^i: w\in V(G)\}$. Note that this happens with probability at least
\[
\prod_{j=1}^t |\cc_j|^{-\binom{|V(G)|}{k-1}} \geq \prod_{j=1}^t |\cc_j|^{-n^{k-1}} \geq c^{-tn^{k-1}}=p.
\]

If $\zeta(1)=+1$, then the order $\prec$ on these vertices agrees with the $\sigma(1)$-th coordinate order. Thus, $H^{\prec}[F_i]$
contains an ordered copy of $H_{\sigma(1)}^+$. Since $H_{\sigma(1)}^+$ is not $\mathscr P_0$-colorable, this implies that $H^{\prec}[F_i]$ is not $\mathscr P_0$-colorable.
If $\zeta(1)=-1$, then the order $\prec$ on these vertices is the reverse of the $\sigma(1)$-th coordinate order. 
Let
$\rm{rev}(H_{\sigma(1)}^{-})$
denote the ordered $k$-graph obtained from $H_{\sigma(1)}^{-}$ by reversing its linear order. Then
$H^{\prec}[F_i]$ contains an ordered copy of $\rm{rev}(H_{\sigma(1)}^{-})$. Since $\rm{rev}(H_{\sigma(1)}^{-})$ is not $\mathscr P_0$-colorable, this again implies that $H^{\prec}[F_i]$ is not $\mathscr P_0$-colorable.
Therefore, for every good set $F_i$, we have
\[
\mathbb P\bigl(H^{\prec}[F_i]\text{ is }\mathscr P_0\text{-colorable}\bigr)\le 1-p.
\]

Since the events
$\bigl\{H^{\prec}[F_i]\text{ is }\mathscr P_0\text{-colorable}\bigr\}$
over all good sets $F_i$ are independent, we obtain
\[
\mathbb P\bigl(H^{\prec}\text{ is }\mathscr P_0\text{-colorable}\bigr)\le (1-p)^M.
\]

Finally, taking the union bound over all $N!$ linear orders on $[N]$, we conclude that
\[
\mathbb P\bigl(H\text{ is }\mathscr P_0\text{-colorable}\bigr)
\le
N!\cdot (1-p)^M
<
1
\]
by \eqref{eq:random-multicase}. Therefore, with positive probability, the random $k$-graph $H$ is $\mathscr P_s$-colorable for every $s\in[t]$ but not $\mathscr P_0$-colorable. 
\end{proof}

Finally, we use \cref{lemma: key lemma for classification for multicase} and \cref{thm6.dierorder} to prove \cref{thm: classification for multicase}.

\begin{proof}[Proof of  \cref{thm: classification for multicase}]
Suppose first that there exists a $k$-graph $H$ that is $\mathscr P_i$-colorable for every $i\in[t]$ but not $\mathscr P_0$-colorable. Fix $j\in[t]$. Let $\prec$ be an order witnessing that $H$ is $\mathscr P_j$-colorable. By \cref{obs:symmetrization-order}, for every $i\in[t]\setminus\{j\}$, the ordered $k$-graph $H^\prec$ is $\mathscr P_i^{\mathcal S}$-colorable. Hence $H^\prec$ is colorable by the $k$-palette $\mathscr P_j\times\prod_{i\in[t]\setminus\{j\}}\mathscr P_i^{\mathcal S}$.
If there exists a homomorphism from this $k$-palette to $\mathscr P_0$, then $H$ would be $\mathscr P_0$-colorable, a contradiction.  Thus, no such homomorphism exists.

Conversely, suppose that for every $j\in[t]$, there is no homomorphism from
$\mathscr P_j\times\prod_{i\in[t]\setminus\{j\}}\mathscr P_i^{\mathcal S}$
to $\mathscr P_0$ or to ${\rm rev}(\mathscr P_0)$. Fix $i\in [t]$.
By \cref{lemma: key lemma for classification for multicase}, there exists an ordered $k$-graph $H_i^{+}$ that is ${\mathscr P}_i$-colorable but not ${\mathscr P}_0$-colorable, and $\widetilde H_i^{+}$ is ${\mathscr P}_{i'}$-colorable, for every $i' \in [t] \setminus \{i\}$. Applying the same contrapositive of~\cref{lemma: key lemma for classification for multicase} with ${\rm rev}(\mathscr P_0)$,
there exists an ordered $k$-graph $H_i^{-}$ that is ${\mathscr P}_i$-colorable but not ${\rm rev}({\mathscr P}_0)$-colorable, and $\widetilde H_i^{-}$ is ${\mathscr P}_{i'}$-colorable for every $i' \in [t] \setminus \{i\}$.

Thus, the hypotheses of \cref{thm6.dierorder} are satisfied. By \cref{thm6.dierorder}, there exists a $k$-graph $H$ that is ${\mathscr P}_i$-colorable for every $i \in [t]$ but not ${\mathscr P}_0$-colorable.
\end{proof}

%%%%%%%%%%%%%%%%%%%%%%5
\section{New values of $(k-2)$-uniform Tur\'an density}\label{sec:example}

In this section, we shall use the
following consequence of \cref{Thm:main theorem 1} and
\cref{thm: classification for multicase} to prove~\cref{thm:values of k-2-pi} and~\cref{thm:non-principal family}.

\begin{lemma}\label{lem:palette-certificate}
Given $k\ge 3$ and $t\in \mathbb N$, let $\alpha\in[0,1]$, and let
$\mathscr P_0,\mathscr P_1,\dots,\mathscr P_t$ be $k$-palettes such that
$d(\mathscr P_0)=\alpha$. Suppose that the following two conditions hold.
\stepcounter{propcounter}
\begin{enumerate}[label = {{\rm (\Alph{propcounter}\arabic{enumi})}}]
\item\label{palette step 1} For every $i\in[t]$, there is no homomorphism from
\[
\mathscr P_i\times \prod_{j\in[t]\setminus\{i\}}\mathscr P_j^{\mathcal S}
\]
to $\mathscr P_0$ or to ${\rm rev}(\mathscr P_0)$.
\item \label{palette step 2} For every $k$-palette $\mathscr P$ with $d(\mathscr P)>\alpha$, there exists
$i\in[t]$ such that $\mathscr P_i$ admits a homomorphism to $\mathscr P$ or to
${\rm rev}(\mathscr P)$.
\end{enumerate}
Then there exists a $k$-graph $F$ with
$\pi_{k-2}(F)=\alpha$ that is $\mathscr{P}_i$-colorable for $i\in[t]$ but not $\mathscr{P}_0$-colorable.
\end{lemma}

\begin{proof}
By \cref{thm: classification for multicase}, condition~\ref{palette step 1}
gives a $k$-graph $F$ that is $\mathscr P_i$-colorable for every
$i\in[t]$ but is not $\mathscr P_0$-colorable.  Hence, by
\cref{Thm:main theorem 1},
\[
\pi_{k-2}(F)=\pi_k^{\rm pal}(F)\ge d(\mathscr P_0)=\alpha .
\]
On the other hand, let $\mathscr P$ be any $k$-palette with
$d(\mathscr P)>\alpha$.  By condition~\ref{palette step 2}, there exists
$i\in[t]$ such that $\mathscr P_i$ admits a homomorphism to $\mathscr P$ or to
${\rm rev}(\mathscr P)$.  Since $F$ is $\mathscr P_i$-colorable, it follows
that $F$ is $\mathscr P$-colorable; in the second case, we can reverse the
vertex order witnessing the ${\rm rev}(\mathscr P)$-coloring.  Thus, no
$k$-palette of density greater than $\alpha$ contributes to the supremum in
$\pi_k^{\rm pal}(F)$.  Again using \cref{Thm:main theorem 1}, we get
$\pi_{k-2}(F)=\pi_k^{\rm pal}(F)\le \alpha$.
\end{proof}

\begin{remark}
\cref{lem:palette-certificate} serves as a useful tool that reduces the problem of determining the exact value of $\pi_{k-2}$ to two finite homomorphism problems for palettes.
In the applications below, the $k$-palettes $\mathscr P_i$ are chosen such that the
non-existence of homomorphisms to $\mathscr P_0$ is witnessed by a small
``forbidden coordinate pattern'', while the upper bound follows from some
extremal argument on coordinate-support sets or auxiliary digraphs.
This is the main advantage of using \cref{lem:palette-certificate}:
it separates the construction of the forbidden $k$-graph from the verification
of the density threshold.
\end{remark}

Before proving~\cref{thm:values of k-2-pi}, we first introduce two elementary extremal results for digraphs that will later be used in the proof of~\cref{thm:values of k-2-pi}. 

Recall that a \emph{transitive tournament} $T_r$ is a digraph on $r$ vertices where the vertices can be ordered as $v_1, \dots, v_r$ such that $v_iv_j$ is an arc of $T_r$ for all $ 1\le i<j\le r$. The following lemma follows from Brown and Harary~\cite{brown1970extremal}, by noting that they determined the Tur\'an number of tournaments.

\begin{lemma}[\cite{brown1970extremal}]\label{lem:turan number of tournaments}
For any $n\ge r\ge 2$, if a loopless digraph $D$ on $n$ vertices contains no copy of
$T_{r+1}$, then the number of arcs of $D$ is at most $\frac{r-1}{r}n^2$.
\end{lemma}

A \emph{directed walk} of length $r$ in a digraph $D$ is a sequence $(v_0, v_1, \dots, v_r)$ such that $v_{i-1}v_i$ is an arc in $D$ for each $1\le i\le r$. The vertices and arcs in a directed walk are allowed to repeat.

\begin{lemma}\label{lem:di-walk}
For any $n> r\ge 1$, if a digraph $D$ on $n$ vertices contains no directed walk of
length $r$, then the number of arcs of $D$ is at most $\frac{r-1}{2r}n^2$.
\end{lemma}

\begin{proof}
For each vertex $v\in V(D)$, let $\ell(v)$ be the maximum length of a directed
walk ending at $v$.  Since there is no directed walk of length $r$, we have
$\ell(v)\in\{0,1,\dots,r-1\}$. Let $L_i=\{v\in V(D):\ell(v)=i\}$ for each $0\le i\le r-1$.
Note that every arc goes from some $L_i$ to some $L_j$ with $i<j$.  Hence,
\[
|E(D)|
\le \sum_{0\le i<j\le r-1}|L_i||L_j|
=
\frac12\left(n^2-\sum_{i=0}^{r-1}|L_i|^2\right)
\le
\frac12\left(n^2-\frac{n^2}{r}\right)
=
\frac{r-1}{2r}n^2.
\]
\end{proof}

\begin{proof}[Proof of \cref{thm:values of k-2-pi}]
We prove each value in~\cref{thm:values of k-2-pi} by~\cref{lem:palette-certificate}, respectively.

\smallskip
\noindent\textbf{(1) Proving $\frac{r-1}{r}\in \Pi^{(k)}_{\rm u}$.}
Let
\[
\mathscr P_0=([r],\mathcal T_0),
\qquad
\mathcal T_0=\{(x_1,\dots,x_k)\in[r]^k:x_1\ne x_2\}.
\]
Then $d(\mathscr P_0)=(r-1)/r$.  Let $T_{r+1}$ and $T'_{r+1}$ be two disjoint transitive tournaments with
$V(T_{r+1})=\{p_1,\dots,p_{r+1}\}$ and 
$V(T'_{r+1})=\{q_1,\dots,q_{r+1}\}$. Define a $k$-palette $\mathscr P_1=(\mathcal C_1,\mathcal T_1)$ as follows. For every
arc $p_ap_b$ of $T_{r+1}$ with $a<b$, put one $k$-tuple $(p_a, p_b, \alpha^{ab}_3,\dots, \alpha^{ab}_k)$
into $\mathcal T_1$, where all filler colors are pairwise distinct.
For every arc $q_aq_b$ of $T'_{r+1}$ with $a<b$, put one $k$-tuple $(\beta^{ab}_1,\dots,\beta^{ab}_{k-2},q_a,q_b)$ into $\mathcal T_1$, again using new pairwise distinct filler colors. $\mathcal C_1$ consists of all colors in the $k$-tuples of $\mathcal T_1$.

We first verify condition~\ref{palette step 1} of
\cref{lem:palette-certificate}. There is no homomorphism from $\mathscr P_1$ to $\mathscr P_0$. Indeed, any
map from $\{p_1,\dots,p_{r+1}\}$ to $[r]$ identifies two vertices, and the
corresponding arc of $T_{r+1}$ is mapped to a $k$-tuple whose
first two coordinates are equal, which is forbidden in $\mathcal T_0$. Similarly, there is no homomorphism from $\mathscr P_1$ to
${\rm rev}(\mathscr P_0)$, using the arc of $T'_{r+1}$ and the last two coordinates.

Next, we verify condition~\ref{palette step 2} of
\cref{lem:palette-certificate}. 
Let $\mathscr P=(\mathcal C,\mathcal T)$ be a $k$-palette with
$d(\mathscr P)>(r-1)/r$. Define two auxiliary digraphs $L$ and $R$ on
$\mathcal C$ by
\[
ab\in E(L)
\quad\Longleftrightarrow\quad
\text{there exists }(x_1,\dots,x_k)\in\mathcal T
\text{ with }x_1=a,\ x_2=b,
\]
and
\[
ab\in E(R)
\quad\Longleftrightarrow\quad
\text{there exists }(x_1,\dots,x_k)\in\mathcal T
\text{ with }x_{k-1}=a,\ x_k=b.
\]
Since $|\mathcal T|\le |E(L)|\cdot|\mathcal C|^{k-2}$ and $|\mathcal T|\le |E(R)|\cdot|\mathcal C|^{k-2}$, 
we have $|E(L)|, |E(R)|>\frac{r-1}{r}|\mathcal C|^2$. Combined with~\cref{lem:turan number of tournaments}, both $L$ and $R$ either contain a loop or a transitive tournament of $r+1$ vertices.
If $L$ contains a loop, then there is a map from $\mathcal C_1$ to $\mathcal C$ by sending all vertices $p_1,\dots,p_{r+1}$ to the loop vertex and using a witnessing admissible tuple
for each arc.  
If $L$ contains a copy of $T_{r+1}$, map $p_1,\dots,p_{r+1}$ to the vertices of
this copy in the transitive order. Since all other colors in the admissible tuples associated with $T_{r+1}$ are distinct, this gives a homomorphism from the first tournament gadget to $\mathscr P$. 
By symmetry, if $R$ contains a loop or a transitive tournament of $r+1$ vertices,
we similarly obtain a homomorphism from the second tournament gadget to $\mathscr P$.
Since the two parts of
$\mathscr P_1$ use distinct colors, these two maps combine to a homomorphism
from $\mathscr P_1$ to $\mathscr P$.

Thus, the two conditions of \cref{lem:palette-certificate} hold with $t=1$ and
$\alpha=(r-1)/r$, implying $\alpha=(r-1)/r \in \Pi^{(k)}_{\rm u}$.

\smallskip
\noindent{\bf (2) Proving $\frac{(r-1)^2}{r^2}\in \Pi^{(k)}_{\rm u}$.} For $k=3$, this value was obtained in \cite[Theorem~1.9]{LWZZ2025}.  We give
the proof for $k\ge4$.  Let
\[
\mathscr P_0=([r],\mathcal T_0),\qquad
\mathcal T_0=
\{(x_1,\dots,x_k)\in[r]^k:x_1\ne x_2
\text{ and }x_{k-1}\ne x_k\}.
\]
Since the two constrained pairs of coordinates are disjoint for $k\ge4$, we have $d(\mathscr P_0)=\frac{(r-1)^2}{r^2}$. Moreover, $\mathscr P_0={\rm rev}(\mathscr P_0)$.

Let $T_{r+1}$ be a transitive tournament on vertices $p_1,\dots,p_{r+1}$.  Define a $k$-palette $\mathscr P_1=(\mathcal C_1,\mathcal T_1)$ as follows. For every
arc $p_ap_b$ of $T_{r+1}$ with $a<b$, put one $k$-tuple $(p_a, p_b, \alpha^{ab}_3,\dots, \alpha^{ab}_k)$ into $\mathcal T_1$, where all filler colors are pairwise distinct.  Let
$\mathcal C_1$ be the set of all colors appearing in these tuples.

There is no homomorphism from $\mathscr P_1$ to $\mathscr P_0$.  Indeed, any
map from $\{p_1,\dots,p_{r+1}\}$ to $[r]$ identifies two vertices, say
$p_a$ and $p_b$ with $a<b$.  Since $p_ap_b$ is an arc of $T_{r+1}$, the
corresponding admissible tuple of $\mathscr P_1$ is mapped to a tuple whose
first two coordinates are equal, which is forbidden in $\mathcal T_0$.
Since $\mathscr P_0={\rm rev}(\mathscr P_0)$, the same argument excludes
homomorphisms to ${\rm rev}(\mathscr P_0)$.

Let $\mathscr P=(\mathcal C,\mathcal T)$ be a $k$-palette with $d(\mathscr P)>\frac{(r-1)^2}{r^2}$.
Define two digraphs $L$ and $R$ on $\mathcal C$ by recording the first pair and
the last pair of an admissible tuple:
\[
ab\in E(L)
\Longleftrightarrow
\text{there exists }(x_1,\dots,x_k)\in\mathcal T\text{ with }x_1=a,\ x_2=b,
\]
and
\[
ab\in E(R)
\Longleftrightarrow
\text{there exists }(x_1,\dots,x_k)\in\mathcal T\text{ with }x_{k-1}=a,\ x_k=b.
\]
Since the first two coordinates and the last two coordinates of admissible tuples are disjoint, we have
\[
|\mathcal T|\le |E(L)|\cdot|E(R)|\cdot|\mathcal C|^{k-4}.
\]
Hence, at least one of $L$ and $R$ has more than
$\frac{r-1}{r}|\mathcal C|^2$ arcs.

If $L$ has this many arcs, then, as in the previous case, either $L$ has a
loop or it contains a copy of $T_{r+1}$; in both cases $\mathscr P_1$ admits a
homomorphism to $\mathscr P$.  If $R$ has this many arcs, then the similar
argument applied to the first two coordinates of ${\rm rev}(\mathscr P)$ shows
that $\mathscr P_1$ admits a homomorphism to ${\rm rev}(\mathscr P)$.  

Applying \cref{lem:palette-certificate}, we obtain $\frac{(r-1)^2}{r^2}\in\Pi_{\rm u}^{(k)}$ for every $k\ge4$. Together with the cited $k=3$ case, this proves the assertion for all
$k\ge3$.

\smallskip
\noindent\textbf{(3) Proving $\frac{r-1}{2r}\in \Pi^{(k)}_{\rm u}$.}
Let
\[
\mathscr P_0=([r],\mathcal T_0),
\qquad
\mathcal T_0=\{(x_1,\dots,x_k)\in[r]^k:x_1<x_k\}.
\]
Then $d(\mathscr P_0)=\frac12-\frac1{2r}$.  Let $\mathscr P_1=(\mathcal C_1,\mathcal T_1)$ be the $k$-palette with $|\mathcal C_1|=r(k-1)+1$ and $\mathcal T_1=\{(c_i, c^i_1, \dots, c^i_{k-2}, c_{i+1}): i\in[r]\}$, where
\[
  c_1,\ldots,c_{r+1},
  c^1_1,\ldots,c^1_{k-2},\ldots,
  c^r_1,\ldots,c^r_{k-2}
\]
are pairwise distinct.   There is no
homomorphism from $\mathscr P_1$ to $\mathscr P_0$ or to ${\rm rev}(\mathscr P_0)$,
since such a homomorphism would give a strictly increasing, respectively
strictly decreasing, sequence of length $r+1$ in $[r]$.

Let $\mathscr P=(\mathcal C,\mathcal T)$ be a $k$-palette with
$d(\mathscr P)>\frac12-\frac1{2r}$. 
Next, we prove that there is a homomorphism from $\mathscr{P}_1$ to $\mathscr{P}$.
Define a digraph $D$ on vertex set
$\mathcal C$ by putting an arc $a\to b$ whenever some admissible $k$-tuple in
$\mathcal T$ has first coordinate $a$ and last coordinate $b$. If there is no homomorphism from $\mathscr{P}_1$ to $\mathscr{P}$, then $D$ contains no directed walk of
length $r$. 
By~\cref{lem:di-walk}, we have
\begin{align*}
 d(\mathscr{P})&\leq\frac{|\mathcal{C}|^{k-2}|E(D)|}{|\mathcal{C}|^{k}}
\leq\frac{r-1}{2r},
\end{align*}
which leads to a contradiction. 
By \cref{lem:palette-certificate}, we obtain $\frac{r-1}{2r}\in\Pi_{\rm u}^{(k)}$ for every $k\ge3$.

\smallskip
\noindent\textbf{(4) Proving $\frac{(k-1)^k}{k^k}\in \Pi^{(k)}_{\rm u}$.}
Let
\[
\mathscr P_0=([k],\mathcal T_0),
\qquad
\mathcal T_0=\{(x_1,\dots,x_k)\in[k]^k: x_i\ne i\text{ for all }i\in[k]\}.
\]
Then $d(\mathscr P_0)=\frac{(k-1)^k}{k^k}$.
Let $\mathscr{P}_1=(\mathcal{C}_1, \mathcal{T}_1)$ be a $k$-palette with one special color
$a$ and
\[
\mathcal T_1=\{A_i=(c^i_1,\ldots,c^i_{i-1},a,c^i_{i+1},\ldots,c^i_k):
  i\in[k]\},
\]
where all colors other than $a$ are pairwise distinct.

If there is a homomorphism $f:\mathcal{C}_1 \to [k]$. Let $f(a)=s\in[k]$. Then we have 
\[
(f(c^s_1),\dots,f(c^s_{s-1}),f(a),f(c^s_{s+1}), \dots,f(c^s_{k}))\in \mathcal{T}_0,
\]
which leads to a contradiction. Thus, there is no homomorphism from $\mathscr{P}_1$ to $\mathscr{P}_0$. Similarly, there is no homomorphism from $\mathscr P_1$ to
${\rm rev}(\mathscr P_0)$. Since
\[
{\rm rev}(\mathcal T_0)
=
\{(y_1,\dots,y_k)\in[k]^k:y_i\ne k+1-i\text{ for all }i\in[k]\},
\]
if $f(a)=s$, then the admissible tuple of $\mathscr P_1$ in which $a$
appears in coordinate $k+1-s$ is mapped to a tuple whose $(k+1-s)$-th
coordinate is $s$, which is forbidden in ${\rm rev}(\mathscr P_0)$.

Next, we  prove that for every $k$-palette $\mathscr{P}=(\mathcal{C}, \mathcal T)$ with $d(\mathscr{P})>\frac{(k-1)^k}{k^k}$ there exists a homomorphism from $\mathscr{P}_1$ to $\mathscr{P}$. Suppose for a contradiction that no such homomorphism exists.
For each $i\in[k]$, define 
\begin{equation*}
L_i:=\{x\in \mathcal{C}:x~ \text{is the}~ i\text{-th coordinate of some } k\text{-tuple in}~ \mathcal T\}.
\end{equation*}
We claim that $\bigcap_{i=1}^kL_i=\emptyset$. Otherwise, let $c\in \bigcap_{i=1}^k L_i$. Then $f(a)=c$ can be extended to a homomorphism from $\mathscr{P}_1$ to $\mathscr{P}$, since all other colors in $\mathcal{C}_1\setminus\{a\}$ appear only once in some $k$-tuple in $\mathcal{T}_1$. If the intersection is empty, then every color
belongs to at most $k-1$ of the sets $L_i$, so
$\sum_{i=1}^k |L_i|\le (k-1)|\mathcal C|$.  Consequently,
\begin{align*}
d(\mathscr{P}) \leq \frac{\prod_{i=1}^k |L_i|}{|\mathcal{C}|^k}\leq \frac{(\sum_{i=1}^k |L_i|)^k}{(k|\mathcal{C}|)^k}\le \frac{(k-1)^k}{k^k},
\end{align*}
which leads to a contradiction. 
By \cref{lem:palette-certificate}, we obtain $\frac{(k-1)^k}{k^k}\in\Pi_{\rm u}^{(k)}$ for every $k\ge3$.

\smallskip    
\noindent{\bf (5) Proving $\frac{4(k-2)^{k-2}}{k^k}\in \Pi^{(k)}_{\rm u}$.} 
Let $\mathcal C_0=\mathcal C'_0\dot\cup \mathcal C''_0$ with $|\mathcal C'_0|=2$ and $|\mathcal C''_0|=k-2$. Define $\mathscr P_0=(\mathcal C_0,\mathcal T_0)$ by
\[
\mathcal T_0=
\{(x_1,\dots,x_k)\in\mathcal C_0^k:
x_1,x_k\in \mathcal C'_0,\ x_2,\dots,x_{k-1}\in \mathcal C''_0\}.
\]
Then $d(\mathscr P_0)=
\frac{|\mathcal C'_0|^2|\mathcal C''_0|^{k-2}}{(|\mathcal C'_0|+|\mathcal C''_0|)^k}
=
\frac{4(k-2)^{k-2}}{k^k}$.  Moreover, $\mathscr P_0={\rm rev}(\mathscr P_0)$. For each $i\in\{2,\dots,k-1\}$, define a $k$-palette $\mathscr{P}_i=(\mathcal C_i,\mathcal T_i)$ with one special color $a_i$ and
two admissible tuples $(a_i,\alpha^i_2,\dots,\alpha^i_k)$ and $(\beta^i_1,\dots,\beta^i_{i-1},a_i,\beta^i_{i+1},\dots,\beta^i_k)$,
where all colors other than $a_i$ are pairwise distinct.

We  verify condition~\ref{palette step 1} of
\cref{lem:palette-certificate}.  For each $i$, there is no homomorphism from
$\mathscr P_i$ to $\mathscr P_0$ or to ${\rm rev}(\mathscr P_0)$, since the
special color $a_i$ appears once in an endpoint coordinate and once in a middle coordinate, while $\mathcal C'_0$ and $\mathcal C''_0$ are disjoint.  

We claim that, for any $i,j\in\{2,\dots,k-1\}$, there is a homomorphism from $\mathscr P_i$ to $\mathscr P_j^{\mathcal S}$.  Indeed, in the
symmetrization of $\mathscr P_j$, the special color $a_j$ can be placed, with
the same superscript, in any prescribed coordinate of either of the two
symmetrized admissible tuples.  This follows directly from the fact that, for
fixed $p,q\in[k]$ and fixed $\sigma\in\mathbb S_{k-1}$, there is a unique
$\tau\in\mathbb S_k$ such that $\tau(p)=q$ and $\partial_p\tau=\sigma$.
Since all other colors of $\mathscr P_i$ are new, the desired homomorphism
extends arbitrarily over them.
Consequently, $\mathscr P_i$ admits a homomorphism to
\[
\mathscr P_i\times \prod_{j\in[k]\setminus\{1,i,k\}}\mathscr P_j^{\mathcal S}
\]
by using the identity map in the $\mathscr P_i$-coordinate and the above
homomorphisms in all other coordinates.  Hence any homomorphism from this
mixed product to $\mathscr P_0$ or to ${\rm rev}(\mathscr P_0)$ would imply a
homomorphism from $\mathscr P_i$ to $\mathscr P_0$ or to
${\rm rev}(\mathscr P_0)$, a contradiction.

We verify condition~\ref{palette step 2} of
\cref{lem:palette-certificate}. Let $\mathscr{P}=(\mathcal{C},\mathcal{T})$ be a  $k$-palette with $d(\mathscr{P})>\frac{4(k-2)^{k-2}}{k^k}$.  We claim that there is a homomorphism from some $\mathscr{P}_i$ to $\mathscr{P}$ or to ${\rm rev}(\mathscr{P})$ for $i\in\{2,\dots,k-1\}$. If not, for each $j\in [k]$ let 
\[
C_j:=
\{c\in\mathcal C:\text{ there exists }(x_1,\dots,x_k)\in\mathcal T
\text{ with }x_j=c\}.
\]
Then for every
$i\in\{2,\dots,k-1\}$, $C_1\cap C_i=\emptyset$
and $C_k\cap C_i=\emptyset$.
Therefore every middle coordinate of every admissible tuple of $\mathscr P$
lies in $\mathcal C\setminus
\bigl(C_1\cup C_k\bigr)$.
Set $n=|\mathcal C|$ and $x=|C_1\cup C_k|/n$.
Then we have
\[
d(\mathscr{P})\leq\frac{\prod_{i=1}^k|C_i|}{|\cc|^k}\leq \frac{|C_1||C_k|(|\cc|-|C_1\cup C_k|)^{k-2}}{|\cc|^k}\leq x^2(1-x)^{k-2}.
\]
Since \[
\max_{0\le x\le1}x^2(1-x)^{k-2}
=\frac{4(k-2)^{k-2}}{k^k}
\]
for every $k\ge3$, we obtain $d(\mathscr{P})\le \frac{4(k-2)^{k-2}}{k^k}$, which leads to a contradiction. Hence condition~\ref{palette step 2} holds. By \cref{lem:palette-certificate}, $\frac{4(k-2)^{k-2}}{k^k}\in\Pi_{\rm u}^{(k)}$.

\smallskip
\noindent{\bf (6) Proving $\frac{4(k-2)^{k-2}}{3k^k}\in \Pi^{(k)}_{\rm u}$.}
Let $\mathcal C_0=C_{1}\dot\cup C_{2}\dot\cup C_{3}\dot\cup C_{4} $ with $|C_{1}|=| C_2|=|C_{3}|=2$ and $|C_{4}|=3(k-2)$. Thus $|\mathcal C_0|=3k$.  Define $\mathscr P_0=(\mathcal C_0,\mathcal T_0)$ by
\[
\mathcal T_0=
\{(x_1,\dots,x_k)\in\mathcal C_0^k:
x_1\in C_{1}\cup C_3,\ x_k\in C_2\cup C_3,\ x_2,\dots,x_{k-1}\in C_4 \text{ and } \{x_1, x_k\}\not\subset C_3\}.
\]
Hence, 
\[
d(\mathscr P_0)
=
\frac{\bigl(|C_1||C_2|+|C_1||C_3|+|C_3||C_2|\bigr)|C_4|^{k-2}}
{|\mathcal C_0|^k}
=
\frac{12\cdot (3(k-2))^{k-2}}{(3k)^k}
=
\frac{4(k-2)^{k-2}}{3k^k}.
\]
For each $i\in\{2,\dots,k-1\}$, define a $k$-palette $\mathscr{P}_i=(\mathcal C_i,\mathcal T_i)$ with one special color $a_i$ and
two admissible tuples $(a_i,\alpha^i_2,\dots,\alpha^i_k)$ and $(\beta^i_1,\dots,\beta^i_{i-1},a_i,\beta^i_{i+1},\dots,\beta^i_k)$,
where all colors other than $a_i$ are pairwise distinct. Define  $\mathscr{P}_k=(\mathcal C_k,\mathcal T_k)$ with two special colors $b$ and $c$ and three admissible tuples \[
(\rho_1,\dots,\rho_{k-1},b),\qquad
(b,\sigma_2,\dots,\sigma_{k-1},c),\qquad
(c,\tau_2,\dots,\tau_k),
\]
where all filler colors are pairwise distinct too. 

 As above the proof, for any
$i,j\in\{2,\dots,k-1\}$, $\mathscr P_i$ admits a homomorphism to
$\mathscr P_j^{\mathcal S}$.  Moreover, each $\mathscr P_i$ admits a
homomorphism to $\mathscr P_k^{\mathcal S}$: in the symmetrization of
$\mathscr P_k$, the color $b$ can be placed with the same superscript in an
endpoint coordinate and in any prescribed middle coordinate.
Conversely, $\mathscr P_k$ admits a homomorphism to $\mathscr P_i^{\mathcal S}$
for every $i\in\{2,\dots,k-1\}$.  Indeed, the symmetrization of $\mathscr P_i$
allows two fixed formal colors to occupy the three endpoint positions required
by the three tuples of $\mathscr P_k$; all remaining colors of $\mathscr P_k$
are new and can be mapped to the remaining coordinates of the corresponding
symmetrized tuples.
Thus each distinguished $k$-palette among $\mathscr P_2,\dots,\mathscr P_{k-1},\mathscr P_k$
admits a homomorphism to its corresponding mixed product.  Therefore any
homomorphism from such a mixed product to $\mathscr P_0$ or to
${\rm rev}(\mathscr P_0)$ would yield a homomorphism from the distinguished
palette itself to $\mathscr P_0$ or to ${\rm rev}(\mathscr P_0)$. 
For each $i\in\{2,\dots,k-1\}$, there is no homomorphism from $\mathscr P_i$
to $\mathscr P_0$ or to ${\rm rev}(\mathscr P_0)$, because the special color
$a_i$ appears once in an endpoint coordinate and once in a middle coordinate,
while in every admissible tuple of $\mathscr P_0$ and of
${\rm rev}(\mathscr P_0)$ the endpoint coordinates lie in
$C_1\cup C_2\cup C_3$ and the middle coordinates lie in the disjoint set $C_4$.
There is also no homomorphism from $\mathscr P_k$ to
$\mathscr P_0$. Indeed, the first and second $k$-tuples in $\mathcal T_k$ force the image of $b$ to lie in $(C_1\cup C_3)\cap (C_2\cup C_3)=C_3$
and the second and third $k$-tuples in $\mathcal T_k$ force the image of $c$ to lie in $C_3$, which implies that the
middle $k$-tuple has both endpoint colors in $C_3$, which is forbidden by the
definition of $\mathcal T_0$.  The same argument applies to
${\rm rev}(\mathscr P_0)$. Therefore, for each $j\in\{2,\dots,k-1\}$, there is no homomorphism from
\[
\mathscr P_j\times \prod_{j'\in[k]\setminus\{1,j\}}\mathscr P_{j'}^{\mathcal S}
\]
to $\mathscr P_0$ or to ${\rm rev}(\mathscr P_0)$.

For any $k$-palette $\mathscr{P}=(\mathcal{C},\mathcal{T})$ with $d(\mathscr{P})>\frac{4(k-2)^{k-2}}{3k^k}$, we claim that there is a homomorphism from some $\mathscr{P}_i$ to $\mathscr{P}$ or to ${\rm rev}(\mathscr{P})$ for $i\in[2,k]$. If not, 
for each $j\in [k]$ let 
\[
C'_j:=
\{c\in\mathcal C:\text{ there exists }(x_1,\dots,x_k)\in\mathcal T
\text{ with }x_j=c\}.
\]
Since there is no  homomorphism from $\mathscr{P}_i$ to $\mathscr{P}$ or to ${\rm rev}(\mathscr{P})$, we have $C'_1\cap C'_i=\emptyset$
and $C'_k\cap C'_i=\emptyset$ for each $i\in \{2, \dots, k-1\}$.
Let $X=C'_1\cap C'_k$, $Y= C'_1\setminus X$ and $Z=C'_k\setminus X$.
Since there is no  homomorphism from $\mathscr{P}_k$ to $\mathscr{P}$ or to ${\rm rev}(\mathscr{P})$, each $(c_1, \dots, c_k)$ in $\mathcal{T}$ has the property that $c_1$ and $c_k$ cannot belong to $X$ at the same time. Thus, $|\mathcal T|\le \bigl(|X||Y|+|X||Z|+|Y||Z|\bigr)
\bigl(|\mathcal C|-|X|-|Y|-|Z|\bigr)^{k-2}$. 
Note that for any reals $x, y, z \ge0$ and $x+y+z\le1$, we have
\[
(xy+xz+yz)(1-x-y-z)^{k-2}
\le
\frac{4(k-2)^{k-2}}{3k^k}.
\]
Therefore, we have
\begin{align*}
  d(\mathscr{P})\leq\frac{\bigl(|X||Y|+|X||Z|+|Y||Z|\bigr)
    	\bigl(|\mathcal C|-|X|-|Y|-|Z|\bigr)^{k-2}}{|\cc|^k}\le \frac{4(k-2)^{k-2}}{3k^k},
\end{align*}
which leads to a contradiction. By \cref{lem:palette-certificate}, $\frac{4(k-2)^{k-2}}{3k^k} \in\Pi_{\rm u}^{(k)}$.
\end{proof}

Finally, we will prove \cref{thm:non-principal family} by repeatedly applying  \cref{Thm:main theorem 1} and \cref{lem:palette-certificate}.

\begin{proof}[Proof of \cref{thm:non-principal family}] 
We treat the cases $k=3$ and $k\ge4 $ separately, since the construction for $k\ge4$
uses the two disjoint coordinate pairs $\{1,k\}$ and $\{2,k-1\}$, which
degenerate when $k=3$.  In both cases, the proof follows the same strategy:
we construct two $k$-palettes $\mathscr P_1,\mathscr P_2$ which give lower bounds
for the individual densities, and then show that any $k$-palette of density above a smaller threshold
admits a homomorphism from at least one of them.

\smallskip
\noindent\textbf{The case  1: $k=3$.}
Let $\mathscr{P}_1=([5], \mathcal{T}_1)$ with $\mathcal{T}_1=\left\{(1,2,3),(2,4,5)\right\}$ and $\mathscr{P}'=([3],\mathcal{T}')$ with $\mathcal{T}'=\left\{(x_1,2,x_3):x_1,x_3\in\{1,3\}\right\}$. Then $d(\mathscr{P}')=|\mathcal{T}'|/|\cc|^3=4/27$.

We first show that no homomorphism exists from $\mathscr{P}_1$ to $\mathscr{P}'$ or $\rev(\mathscr{P}')$. Suppose for contradiction $f:\mathscr{P}_1\to\mathscr{P}'$ is a homomorphism. Then $(f(1),f(2),f(3))\in\mathcal{T}'$ implies $f(2)=2$ and $f(1),f(3)\in\{1,3\}$. Meanwhile, $(f(2),f(4),f(5))\in\mathcal{T}'$ requires $f(2)\in\{1,3\}$, a contradiction. Since $\rev(\mathscr{P}')=\mathscr{P}'$, the same holds for $\rev(\mathscr{P}')$.

Let $\mathscr{P}=(\cc,\mathcal{T})$ be an arbitrary 3-palette with $d(\mathscr{P})>4/27$. 
For $i\in[3]$, let $A_i$ be the set of colors appearing in the $i$-th coordinate of some tuple in $\mathcal{T}$. We claim there exists a homomorphism from $\mathscr{P}_1$ to $\mathscr{P}$ or $\rev(\mathscr{P})$. Suppose not.

We first show $A_1\cap A_2=\emptyset$. If $a\in A_1\cap A_2$, there exist $(a,b,c)\in\mathcal{T}$ (as $a\in A_1$) and $(d,a,e)\in\mathcal{T}$ (as $a\in A_2$). Define $f:[5]\to\cc$ by $f(1)=d,f(2)=a,f(3)=e,f(4)=b,f(5)=c$. Then $(f(1),f(2),f(3))=(d,a,e)\in\mathcal{T}$ and $(f(2),f(4),f(5))=(a,b,c)\in\mathcal{T}$, so $f:\mathscr{P}_1\to\mathscr{P}$ is a homomorphism, a contradiction. Thus, $A_1\cap A_2=\emptyset$. Similarly, no homomorphism to $\rev(\mathscr{P})$ implies $A_3\cap A_2=\emptyset$.

Since every tuple in $\mathcal{T}$ has coordinates in $A_1,A_2,A_3$ respectively, $|\mathcal{T}|\le |A_1||A_2||A_3|$. As $A_1,A_3\subseteq\cc\setminus A_2$, we have $|A_1|,|A_3|\le|\cc|-|A_2|$. Thus,
\[
d(\mathscr{P})=\frac{|\mathcal{T}|}{|\cc|^3}\le\frac{|A_2|(|\cc|-|A_2|)^2}{|\cc|^3}.
\]
The function $x\mapsto x(n-x)^2/n^3$ attains its maximum $4/27$ at $x=n/3$, so $d(\mathscr{P})\le4/27$, a contradiction. Hence the claim holds. By \cref{lem:palette-certificate},  there exists a 3-graph $F_1$ with $\pi_1(F_1)=4/27$ that is $\mathscr{P}_1$-colorable but not $\mathscr{P}'$-colorable.

Now let $\mathscr{P}_2=([5],\mathcal{T}_2)$ with $\mathcal{T}_2=\left\{(1,2,3),(3,4,5)\right\}$ and let  $\mathscr{P}''=([2],\mathcal{T}'')$ with $\mathcal{T}''=\{(1,x,2):x\in\{1,2\}\}$. Then $d(\mathscr{P}'')=\frac{1}{4}$.

There is no homomorphism from $\mathscr{P}_2$ to $\mathscr{P}''$ or $\rev(\mathscr{P}'')$: a homomorphism to $\mathscr{P}''$ would require $f(1)=1,f(3)=2$, but then $(f(3),f(4),f(5))=(2,\cdot,\cdot)\notin\mathcal{T}''$. For $\rev(\mathscr{P}'')$, the contradiction is symmetric. 

Let $\mathscr{P}=(\cc,\mathcal{T})$ be a 3-palette with $d(\mathscr{P})>1/4$. For $i\in[3]$, let $A_i$ be as before. Then,
\[
\frac{|A_1||\cc||A_3|}{|\cc|^3}\ge d(\mathscr{P})>\frac{1}{4}.
\]
By AM-GM inequality, $|A_1|+|A_3|\ge2\sqrt{|A_1||A_3|}>|\cc|$, so $A_1\cap A_3\neq\emptyset$. Let $a\in A_1\cap A_3$, then there exist $(x_1,x_2,a)\in\mathcal{T}$ (as $a\in A_3$) and $(a,y_1,y_2)\in\mathcal{T}$ (as $a\in A_1$). Define $f:[5]\to\cc$ by $(f(1),f(2),f(3))=(x_1,x_2,a)$ and $(f(3),f(4),f(5))=(a,y_1,y_2)$. By definition, $f:\mathscr{P}_2\to\mathscr{P}$ is a homomorphism. By \cref{lem:palette-certificate},  there exists a 3-graph $F_2$  with $\pi_1(F_2)=1/4$ that is $\mathscr{P}_2$-colorable but not $\mathscr{P}''$-colorable.

Finally, we show $\pi_1(\{F_1,F_2\})=1/27<\min\{4/27,1/4\}$.
For the upper bound, let $\mathscr{P}_0=(\cc,\mathcal{T}_0)$ be a 3-palette with $d(\mathscr{P}_0)>1/27$. We claim that there exists a homomorphism from one of $\mathscr{P}_1$, $\rev(\mathscr{P}_1)$ or $\mathscr{P}_2$ to $\mathscr{P}_0$. Suppose not.  For $i\in[3]$, let $A_i$ be as before. Since there is no homomorphism from $\mathscr{P}_1$ to $\mathscr{P}_0$, $A_1\cap A_2=\emptyset$.
Since there is no homomorphism from $\rev(\mathscr{P}_1)$ to $\mathscr{P}_0$, $A_2\cap A_3=\emptyset$. Since there is no homomorphism from $\mathscr{P}_2$ to $\mathscr{P}_0$, $A_1\cap A_3=\emptyset$. Thus $A_1,A_2,A_3$ are pairwise disjoint, so $|A_1|+|A_2|+|A_3|\le|\cc|$. By AM-GM inequality,
\[
|\mathcal{T}_0|\le|A_1||A_2||A_3|\le\left(\frac{|A_1|+|A_2|+|A_3|}{3}\right)^3\le\frac{|\cc|^3}{27},
\]
implying $d(\mathscr{P}_0)\le1/27$, a contradiction. Recall that if there is a homomorphism from $\mathscr{P}$ to $\mathscr{P}'$, then every $k$-graph that is $\mathscr{P}$-colorable
is also $\mathscr{P}'$-colorable. By our claim above, either $F_1$ or $F_2$ is $\mathscr{P}_0$-colorable.  By \cref{Thm:main theorem 1}, $\pi_1(\{F_1,F_2\})\le1/27$.

For the lower bound, let $\mathscr{P}_3=([3],\mathcal{T}_3)$ with $\mathcal{T}_3=\{(1,2,3)\}$, then $d(\mathscr{P}_3)=1/27$. Trivially, there is a homomorphism from $\mathscr{P}_3$ to $\mathscr{P}'$ and a homomorphism from $\mathscr{P}_3$ to $\mathscr{P}''$. Hence, $F_1$ and $F_2$ are not $\mathscr{P}_3$-colorable (as $F_1$, $F_2$ is not $\mathscr{P}',\mathscr{P}''$-colorable respectively),  implying $\pi_1(\{F_1,F_2\})\ge d(\mathscr{P}_3)=1/27$.

Combining bounds, $\pi_1(\{F_1,F_2\})=1/27$, completing the case $k=3$.

\smallskip
\noindent\textbf{The case 2: $k\ge 4$.}
Let $\mathscr{P}_1=([2k-1],\mathcal{T}_1)$ with $\mathcal{T}_1=\left\{(1,2,\dots,k),(k,k+1,\dots,2k-1)\right\}$. Let $\mathscr{P}'=([2],\mathcal{T}')$ with $\mathcal{T}'=\{(1,x_2,x_3,\dots,x_{k-1},2): x_i\in\{1,2\}\ \text{for}\ i\in[2,k-1]\}$. Then  $d(\mathscr{P}')=2^{k-2}/2^k=1/4$.

There is no homomorphism  from $\mathscr{P}_1$ to $\mathscr{P}'$ or $\rev(\mathscr{P}')$: a homomorphism to $\mathscr{P}'$ would require $f(1)=1,f(k)=2$, but then $(f(k),\dots,f(2k-1))=(2,\dots)\notin\mathcal{T}'$. For $\rev(\mathscr{P}')$, the contradiction is symmetric. 

Let $\mathscr{P}=(\cc,\mathcal{T})$ be a $k$-palette with $d(\mathscr{P})>1/4$. For $i\in\{1,k\}$, let $A_i$ be the set of colors appearing in the $i$-th coordinate of some tuple in $\mathcal{T}$. Then,
\[
\frac{|A_1||\cc|^{k-2}|A_k|}{|\cc|^k}\ge d(\mathscr{P})>\frac{1}{4}.
\]
By AM-GM inequality, $A_1\cap A_k\neq\emptyset$. Let $a\in A_1\cap A_k$, then there exist $(x_1,\dots,x_{k-1},a)\in\mathcal{T}$ (as $a\in A_k$) and $(a,y_1,\dots,y_{k-1})\in\mathcal{T}$ (as $a\in A_1$). Define $f:[2k-1]\to\cc$ by $(f(1),\dots,f(k))=(x_1,\dots,x_{k-1},a)$ and $(f(k),\dots,f(2k-1))=(a,y_1,\dots,y_{k-1})$. $f$ is a homomorphism from $\mathscr{P}_1$ to $\mathscr{P}$. By \cref{lem:palette-certificate}, there exists a $k$-graph $F_1$ with  $\pi_{k-2}(F_1)=1/4$ that is $\mathscr{P}_1$-colorable but not $\mathscr{P}'$-colorable.

Now let $\mathscr{P}_2=([2k-1],\mathcal{T}_2)$ with $\mathcal{T}_2=\left\{(1,2,\dots,k),(k+1,k-1,k+2,k+3,\dots,2k-1)\right\}$ (the second tuple has length $k$, with second coordinate $k-1$). Let $\mathscr{P}''=([2],\mathcal{T}'')$ with $\mathcal{T}''=\{(x_1,1,x_3,\dots,x_{k-2},2,x_k):x_i\in\{1,2\}\ \text{for}\ i\in[k]\setminus\{2,k-1\}\}$. Then  $d(\mathscr{P}'')=2^{k-2}/2^k=1/4$.

Similarly, no homomorphism exists from $\mathscr{P}_2$ to $\mathscr{P}''$ or $\rev(\mathscr{P}'')$: a homomorphism to $\mathscr{P}''$ would require $f(k-1)=2$ (from the first tuple in $\mathcal{T}_2$) and $f(k-1)=1$ (from the second tuple in $\mathcal{T}_2$), a contradiction. Let $\mathscr{P}=(\cc,\mathcal{T})$ be a $k$-palette with $d(\mathscr{P})>1/4$. For $i\in\{2,k-1\}$, let $A_i$ be the set of colors appearing in the $i$-th coordinate of some tuple in $\mathcal{T}$.
By the same argument as for $F_1$, we have $A_2\cap A_{k-1}\neq \emptyset$, and there exists a homomorphism from $\mathscr{P}_2$ to $\mathscr{P}$. By \cref{lem:palette-certificate}, there exists a $k$-graph $F_2$ with $\pi_{k-2}(F_2)=1/4$ that is $\mathscr{P}_2$-colorable but not $\mathscr{P}''$-colorable.

Finally, we show $\pi_{k-2}(\{F_1,F_2\})=1/16<1/4$.

For the upper bound, let $\mathscr{P}_0=(\cc,\mathcal{T}_0)$ be a $k$-palette with $d(\mathscr{P}_0)>1/16$. For $i\in\{1,2,k-1,k\}$, let $A_i$ be the set of colors appearing in the $i$-th coordinate of some tuple in $\mathcal{T}_0$. Then,
\[
\frac{|A_1||A_2||\cc|^{k-4}|A_{k-1}||A_k|}{|\cc|^k}\ge d(\mathscr{P}_0)>\frac{1}{16}.
\]
Thus either $|A_1||A_k|/|\cc|^2>1/4$ or $|A_2||A_{k-1}|/|\cc|^2>1/4$ (otherwise their product is at most $1/16$). If the former holds, $A_1\cap A_{k}\neq\emptyset$ by AM-GM inequality and there is a homomorphism $f:\mathscr{P}_1\to\mathscr{P}_0$ (as before), which implies $F_1$ is $\mathscr{P}_0$-colorable. If the latter holds, $A_2\cap A_{k-1}\neq\emptyset$ by AM-GM inequality  and there is a homomorphism $f:\mathscr{P}_2\to\mathscr{P}_0$, which implies $F_2$ is $\mathscr{P}_0$-colorable. By \cref{Thm:main theorem 1}, $\pi_{k-2}(\{F_1,F_2\})\le1/16$.
For the lower bound, let $\mathscr{P}_3=([2],\mathcal{T}_3)$ with $\mathcal{T}_3=\{(1,1,x_3,\dots,x_{k-2},2,2):x_i\in\{1,2\}\ \text{for}\ i\in[3,k-2]\}$, where the first two coordinates are fixed to 1, the last two coordinates are fixed to 2, and the middle $k-4$ coordinates are arbitrary element in $[2]$. Then $|\mathcal{T}_3|=2^{k-4}$, so $d(\mathscr{P}_3)=2^{k-4}/2^k=1/16$. The identity map is a homomorphism $\mathscr{P}_3\to\mathscr{P}'$  and a homomorphism $\mathscr{P}_3\to\mathscr{P}''$. Then neither $F_1$ nor $F_2$ is $\mathscr{P}_3$-colorable.  Hence  $\pi_{k-2}(\{F_1,F_2\})\ge1/16$.

Combining bounds, $\pi_{k-2}(\{F_1,F_2\})=1/16$, completing the case $k\ge4$.
\end{proof}

\section{Concluding remarks}\label{sec:concluding-remarks}

In this paper, we developed a palette framework for the $(k-2)$-uniform Tur\'an density of $k$-graphs and used it to obtain new values as well as non-principal families.  We close with \emph{other levels of uniform density} in which the present approach may be extended.

The density notion studied in this paper is the top non-trivial member of a larger hierarchy introduced by Reiher, R\"odl and Schacht~\cite{RRS-Mantel}.  It is natural to ask whether the palette viewpoint also captures the other levels of this hierarchy.  We recall the general definition here.

\begin{definition}\label{def:j-dense}
Given integers $n\ge k\ge 2$ and $0<j<k$, let $d\in [0,1]$, $\mu>0$, and let $H=(V,E)$ be a $k$-graph with $|V|=n$.  We say that $H$ is \emph{$(d,\mu,j)$-dense} if
\begin{equation}\label{eq:general-j-dense}
\left|\mathcal K_k(G^{(j)})\cap E\right|
\ge
 d\left|\mathcal K_k(G^{(j)})\right|-\mu n^k
\end{equation}
holds for every $j$-graph $G^{(j)}$ with vertex set $V$.
\end{definition}

For a family $\mathcal F$ of $k$-graphs, define
\[
\begin{split}
\pi_j(\mathcal F)=\sup\{d\in[0,1]:&\ \text{for every }\mu>0\text{ and every }n_0\in\mathbb N,\text{ there exists an }\mathcal F\text{-free}\ \\
&\quad (d,\mu,j)\text{-dense }k\text{-graph }H\text{ with }|V(H)|\ge n_0\}.
\end{split}
\]
Thus $\pi_{k-2}$ is the parameter considered throughout the paper.  The case $j=1$ for $k=3$ is the original uniform Tur\'an density, while $j=0$ would correspond to the classic Tur\'an density.

Our $k$-palettes color $(k-1)$-sets and decide whether a $k$-set is an edge from the ordered pattern of its $(k-1)$-shadow.  This is exactly the right level for $(k-2)$-uniform density.  For smaller $j$, however, one expects several layers of colors to be relevant.  A natural candidate is the following multi-level palette.

\begin{definition}[$j$-level palettes]\label{def:j-level-palette}
Fix integers $k\ge 3$ and $1\le j\le k-2$.  A \emph{$(k,j)$-palette} is a tuple
\[
\mathfrak P=(\mathcal C_{j+1},\mathcal C_{j+2},\dots,\mathcal C_{k-1},\mathcal T),
\]
where each $\mathcal C_i$ is a finite set of colors and
\[
\mathcal T\subseteq \prod_{i=j+1}^{k-1}\mathcal C_i^{\binom{k}{i}}
\]
is a set of admissible ordered multi-level patterns.  We fix once and for all a lexicographic ordering of each $\binom{[k]}{i}$, so that the above product is unambiguous.  The density of $\mathfrak P$ is
\[
d(\mathfrak P):=
\frac{|\mathcal T|}{\prod_{i=j+1}^{k-1}|\mathcal C_i|^{\binom{k}{i}}}.
\]
An ordered $k$-graph $F^\prec$ is \emph{$\mathfrak P$-colorable} if, for every $i\in\{j+1,\dots,k-1\}$, there is a coloring
\[
\varphi_i:\binom{V(F)}{i}\to\mathcal C_i
\]
such that for every edge $e=\llbracket v_1,\dots,v_k\rrbracket\in E(F)$, the pattern
\[
\left(
\left(\varphi_i(\{v_s:s\in S\})\right)_{S\in\binom{[k]}{i}}
\right)_{i=j+1}^{k-1}
\]
belongs to $\mathcal T$.
\end{definition}

When $j=k-2$, this definition reduces to the $k$-palettes used in this paper, up to the harmless convention used to order the $k$ many $(k-1)$-subsets of a $k$-set.  For $j<k-2$, the additional color layers reflect the additional layers present in the hypergraph regularity lemma.  This phenomenon is closely related to the obstruction emphasized by Lamaison for uniform Tur\'an density beyond $3$-graphs: pair-color palettes alone are not expected to capture all higher-uniformity behavior, and one is naturally led to palettes that also color triples, quadruples, and so on up to the top proper faces.

Given a family $\mathcal F$ of $k$-graphs, one may define the corresponding palette density by
\[
\pi^{\mathrm{pal}}_{k,j}(\mathcal F)
:=
\sup\{d(\mathfrak P):\mathfrak P\text{ is a }(k,j)\text{-palette and no }F\in\mathcal F\text{ is }\mathfrak P\text{-colorable}\}.
\]
The following problem asks whether our main theorem is the top-level case of a more general principle.

\begin{problem}\label{prob:j-palette-framework}
For which integers $k\ge3$ and $1\le j\le k-2$, and for which families $\mathcal F$ of $k$-graphs, does one have
\[
\pi_j(\mathcal F)=\pi^{\mathrm{pal}}_{k,j}(\mathcal F)?
\]
If the equality fails in this form, what additional structure must be built into the definition of a $(k,j)$-palette?
\end{problem}

The lower-bound direction is suggested by the usual random construction: independently color all $i$-sets with colors from $\mathcal C_i$ for $i=j+1,\dots,k-1$, and declare a $k$-set to be an edge precisely when its multi-level color pattern is admissible.  The main difficulty is the reverse direction.  In the case $j=k-2$, only the top layer of the regular partition is active, and the reduced structure can be contracted to an ordinary $k$-palette.  For $j<k-2$, several layers interact, and any analogue of the contraction step must preserve the correlations between colors on different levels.  This seems to be the principal new obstacle.

\bibliographystyle{abbrv}
\bibliography{ref}
\end{document}